\documentclass[11pt, a4paper]{amsart}

\usepackage{amssymb, textcomp, dsfont}
\usepackage[all]{xy}
\usepackage{hyperref,aliascnt}
\usepackage{mathtools,array}
\usepackage{tikz-cd}
\usepackage{lmodern}
\usepackage[inline]{enumitem}

\newcommand{\llcurly}{{\mathrlap{\prec}\hspace{.2em}{\prec}}}

\numberwithin{equation}{section}

\newtheorem{lma}{Lemma}[section]

\newaliascnt{thmCt}{lma}
\newtheorem{thm}[thmCt]{Theorem}
\aliascntresetthe{thmCt}

\newaliascnt{corCt}{lma}
\newtheorem{cor}[corCt]{Corollary}
\aliascntresetthe{corCt}

\newaliascnt{prpCt}{lma}
\newtheorem{prp}[prpCt]{Proposition}
\aliascntresetthe{prpCt}

\newtheorem*{thm*}{Theorem}
\newtheorem*{cor*}{Corollary}
\newtheorem*{prp*}{Proposition}

\theoremstyle{definition}

\newaliascnt{pgrCt}{lma}
\newtheorem{pgr}[pgrCt]{}
\aliascntresetthe{pgrCt}

\newaliascnt{dfnCt}{lma}

\aliascntresetthe{dfnCt}

\newaliascnt{rmkCt}{lma}
\newtheorem{rmk}[rmkCt]{Remark}
\aliascntresetthe{rmkCt}

\newaliascnt{rmksCt}{lma}

\aliascntresetthe{rmksCt}

\newaliascnt{prblCt}{lma}

\aliascntresetthe{prblCt}

\newaliascnt{exaCt}{lma}
\newtheorem{exa}[exaCt]{Example}
\aliascntresetthe{exaCt}

\newaliascnt{exasCt}{lma}

\aliascntresetthe{exasCt}

\newaliascnt{conjCt}{lma}

\aliascntresetthe{conjCt}

\newcommand{\W}{\CatW}

\newcommand{\N}{\mathbb{N}}

\newcommand{\Z}{\mathbb{Z}}

\newcommand{\K}{\mathrm{K}}

\newcommand{\CatCu}{\ensuremath{\mathrm{Cu}}}

\newcommand{\CatW}{\mathrm{W}}

\newcommand{\ca}{$C^*$-al\-ge\-bra}

\newcommand{\axiomO}[1]{(O#1)}

\newcommand{\axiomW}[1]{(W#1)}

\DeclareMathOperator{\Lat}{Lat}

\DeclareMathOperator{\Cu}{Cu}
\DeclareMathOperator{\Lsc}{Lsc}

\DeclareMathOperator{\supp}{supp}

\DeclareMathOperator{\QT}{QT}

\makeatletter

\newcommand{\Rmnum}[1]{\expandafter\@slowromancap\romannumeral #1@}
\makeatother

\SetLabelAlign{Center}{\makebox[2em]{#1}}
\newcommand{\beginEnumStatements}{\begin{enumerate}[label=(\arabic*),align=Center,leftmargin=*, widest=iiii]}
	\newcommand{\beginEnumConditions}{\begin{enumerate}[label={\rm (\roman*)},align=Center, leftmargin=*, widest=iiii]}
		
		\usepackage{stmaryrd}

\theoremstyle{plain}
\newcounter{thmintroctr}

\newtheorem{thmintro}[thmintroctr]{Theorem}

\theoremstyle{definition}
\newcounter{goalintroctr}

\newtheorem{goalintro}[goalintroctr]{}

\usepackage{xcolor}

\newcommand{\dyn}{\leqr{G}}

\newcommand{\Dyn}{\W_G}
\newcommand{\CuDyn}{\mathrm{Cu}_G}
\newcommand{\CatGW}{\mathrm{G}\text{-}\CatW}
\newcommand{\CatGCu}{\mathrm{G}\text{-}\CatCu}
\newcommand{\her}{\mathrm{her}}

\makeatletter
 \newcommand{\miniscule}{\@setfontsize\miniscule{5}{6}}
\makeatother

\newcommand{\leqr}[1]{\leq_\text{\miniscule $#1$}}
\newcommand{\precr}[1]{\prec_\text{\miniscule $#1$}}

\newcommand{\lesssimupr}[1]{\lesssim^\text{\miniscule $#1$}}
\newcommand{\lequpr}[1]{\leq^\text{\miniscule $#1$}}

\newcommand{\fun}[2]{\operatorname{F}_\text{\miniscule $#1$}^\text{\miniscule $#2$}}

\newcommand{\precsimr}[1]{\precsim_\text{\miniscule $#1$}}
\newcommand{\simr}[1]{\sim_\text{\miniscule $#1$}}

\newcommand{\precuprr}[2]{(\prec^\text{\miniscule $#1$})^\text{\miniscule $#2$}}

\newcommand{\precupr}[1]{\prec^\text{\miniscule $#1$}}

\newcommand{\lequprr}[2]{(\leq^\text{\miniscule $#1$})^\text{\miniscule $#2$}}

\newcommand{\alphar}[1]{\alpha_\text{\miniscule $#1$}}
\newcommand{\alphaupr}[1]{\alpha^\text{\miniscule $#1$}}
\newcommand{\talphar}[1]{\tilde{\alpha}_\text{\miniscule $#1$}}
\newcommand{\talphaupr}[1]{\tilde{\alpha}^\text{\miniscule $#1$}}

\newcommand{\alphauprr}[2]{(\alpha^\text{\miniscule $#1$})^\text{\miniscule $#2$}}

\newcommand{\app}{\approx_\text{\miniscule $G$}}

\begin{document}

\title[Dynamical Cuntz semigroup and ideal-free quotients]{The Dynamical Cuntz semigroup and ideal-free quotients of Cuntz semigroups}

\author{Joan Bosa}
\author{Francesc Perera}
\author{Jianchao Wu}
\author{Joachim Zacharias}

\date{\today}

\address{JB, Departamento de Matematicas,
	Universidad de Zaragoza, C/Pedro Cerbuna 12,
	50009 Zaragoza, Spain}
\email[]{jbosa@unizar.es}
\urladdr{https://personal.unizar.es/jbosa/}

\address{FP, Departament de Matem\`{a}tiques,
	Universitat Aut\`{o}noma de Barcelona,
	\linebreak 08193 Bellaterra, Barcelona, Spain, and
	Centre de Recerca Matem\`atica, Edifici Cc, Campus de Bellaterra,  08193 Cerdanyola del Vall\`es, Barcelona, Spain}
\email[]{francesc.perera@uab.cat}
\urladdr{https://mat.uab.cat/web/perera}

\address{JW, Shanghai Center for Mathematical Sciences, Fudan University, 
	2005 Songhu Rd., Shanghai 200438, China}
\email[]{jianchao\_wu@fudan.edu.cn}

\address{JZ, School of Mathematics and Statistics 	
	University of Glasgow , University Place
	Glasgow G12 8QQ
	United Kingdom}
\email[]{joachim.zacharias@glasgow.ac.uk}
\urladdr{https://www.maths.gla.ac.uk/~jzacharias/}


\thanks{
JB was partially supported by the Spanish State Research Agency through Consolidación Investigadora program No.\ CNS2022-135340 and the grant No.\  PID2020-113047GB-I00/AEI/10.13039/501100011033, and by the Dep. ~Ciencia, Universidad y sociedad del conocimiento del Gobierno de Arag\'on (grupo E22-23R).
FP was partially supported by the Spanish Research State Agency (grant No.\  PID2020-113047GB-I00/AEI/10.13039/501100011033), by the Comissionat per Universitats i Recerca de la Ge\-ne\-ralitat de Ca\-ta\-lu\-nya (grant No.\ 2021-SGR-01015), and by the Spanish State Research Agency through the Severo Ochoa and María de Maeztu Program for Centers and Units of Excellence in R\&D (CEX2020-001084-M).
JW was partially supported by NSFC Key Program No.~12231005 and National
Key R\&D Program of China 2022YFA100700. 
JZ was partially supported by the Engineering and Physical Sciences Research Council Grant EP/I019227/2.
}

\begin{abstract}
	We develop a theory of general quotients for $\CatW$- and $\CatCu$-semigroups beyond the case of quotients by ideals. To this end,  we introduce the notion of a normal pair, which allows us to take quotients of $\CatW$-semigroups in a similar way as normal subgroups arise as kernels of group homomorphisms. 
	
	We use this to define the dynamical Cuntz semigroup as the universal object induced from an action of a group $G$ on a $\CatW$-semigroup. In the C*-algebraic framework, under mild assumptions, the universality of this dynamical invariant helps us tap into the structure of the Cuntz semigroup of crossed product C*-algebras. 
\end{abstract}

\maketitle



\newcommand{\CatWenr}{\CatW_{\mathrm{enr}}}
\newcommand{\CatGWenr}{\CatGW_{\mathrm{enr}}}

\newcommand{\calU}{\mathcal{U}}
\newcommand{\calV}{\mathcal{V}}
\newcommand{\calW}{\mathcal{W}}

\newcommand{\Lscb}{\Lsc_{\operatorname{b}}}

\newcommand{\phantomtoggle}[1]{}

\newcommand{\compareunder}[2]{\mathrel{\:{#1}_{#2}\:}}

\newcommand{\compareover}[2]{\mathrel{\:{#2}^{#1}\:}}

\newcommand{\underlesssim}[1][R]{\compareunder{\lesssim}{#1}}
\newcommand{\undergtrsim}[1][R]{\compareunder{\gtrsim}{#1}}
\newcommand{\underprec}[1][R]{\compareunder{\prec}{#1}}
\newcommand{\undersucc}[1][R]{\compareunder{\succ}{#1}}
\newcommand{\undersim}[1][R]{\compareunder{\sim}{#1}}
\newcommand{\underlnsim}[1][R]{\compareunder{\lnsim}{#1}}

\newcommand{\underlessapprox}[1][R]{\compareunder{\lessapprox}{#1}}
\newcommand{\undergtrapprox}[1][R]{\compareunder{\gtrapprox}{#1}}
\newcommand{\underpreccurlyeq}[1][R]{\compareunder{\preccurlyeq}{#1}}
\newcommand{\undersucccurlyeq}[1][R]{\compareunder{\succcurlyeq}{#1}}
\newcommand{\underapprox}[1][R]{\compareunder{\approx}{#1}}
\newcommand{\underlnapprox}[1][R]{\compareunder{\lnapprox}{#1}}

\newcommand{\underll}[1][R]{\compareunder{\ll}{#1}}
\newcommand{\underlll}[1][R]{\compareunder{\lll}{#1}}
\newcommand{\undergg}[1][R]{\compareunder{\gg}{#1}}
\newcommand{\underggg}[1][R]{\compareunder{\ggg}{#1}}

\newcommand{\underprecsim}[1][R]{\compareunder{\precsim}{#1}}
\newcommand{\undersuccsim}[1][R]{\compareunder{\succsim}{#1}}

\newcommand{\overlesssim}[1][R]{\compareover{#1}{\lesssim}}
\newcommand{\overgtrsim}[1][R]{\compareover{#1}{\gtrsim}}
\newcommand{\overprec}[1][R]{\compareover{#1}{\prec}}
\newcommand{\oversucc}[1][R]{\compareover{#1}{\succ}}
\newcommand{\oversim}[1][R]{\compareover{#1}{\sim}}
\newcommand{\overlnsim}[1][R]{\compareover{#1}{\lnsim}}

\newcommand{\overlessapprox}[1][R]{\compareover{#1}{\lessapprox}}
\newcommand{\overgtrapprox}[1][R]{\compareover{#1}{\gtrapprox}}
\newcommand{\overpreccurlyeq}[1][R]{\compareover{#1}{\preccurlyeq}}
\newcommand{\oversucccurlyeq}[1][R]{\compareover{#1}{\succcurlyeq}}
\newcommand{\overapprox}[1][R]{\compareover{#1}{\approx}}
\newcommand{\overlnapprox}[1][R]{\compareover{#1}{\lnapprox}}

\newcommand{\overll}[1][R]{\compareover{#1}{\ll}}
\newcommand{\overlll}[1][R]{\compareover{#1}{\lll}}
\newcommand{\overgg}[1][R]{\compareover{#1}{\gg}}
\newcommand{\overggg}[1][R]{\compareover{#1}{\ggg}}

\newcommand{\overprecsim}[1][R]{\compareover{#1}{\precsim}}
\newcommand{\oversuccsim}[1][R]{\compareover{#1}{\succsim}}

\newcommand{\precsimher}{\precsim_{\operatorname{her}}}
\newcommand{\leqher}{\leq_{\operatorname{her}}}
\newcommand{\succsimher}{\succsim_{\operatorname{her}}}
\newcommand{\llher}{\ll_{\operatorname{her}}}
\newcommand{\ggher}{\gg_{\operatorname{her}}}

\newcommand{\relsubseteq}{{\ref{relsubseteq} \;{\Rightarrow}\;}} \label{relsubseteq}
\newcommand{\relsupseteq}{{\ref{relsupseteq} \;{\Leftarrow}\;}} \label{relsupseteq}
\newcommand{\releq}{{\ref{releq} \;{\Leftrightarrow}\;}} \label{releq}
\newcommand{\relcap}{\ref{relcap}\cap} \label{relcap}
\newcommand{\relcup}{\ref{relcup}\cup} \label{relcup}

\newcommand{\relletter}[1]{{\;{\mbox{\scriptsize $\left\langle{#1}\right\rangle$}}\;}}

\newcommand{\doubleslash}{{\ref{doubleslash}\hphantom{.}\mathclap{/}\hphantom{.}\mathclap{/}\hphantom{.}}} \label{doubleslash}

\newcommand{\ContourGeq}[2]{\left( {#1} \geq {#2} \right)}

\newcommand{\osupp}{\supp^{\operatorname{o}}}

\newcommand{\underprecsimmult}[2][\alpha]{ {\overprecsim[#1]}^{_{(#2)}} }
\newcommand{\OpenSets}[1][X]{\mathcal{O}\left(#1\right)}
\newcommand{\FreeMonoidOpenSets}[1][X]{\OpenSets[#1]^*}
\newcommand{\FreeAbelianMonoidOpenSets}[1][X]{\N\OpenSets[#1]}
\newcommand{\PowerSet}[1][X]{\mathcal{P}\left(#1\right)}
\newcommand{\FreeMonoidSubsets}[1][X]{\PowerSet[#1]^*}
\newcommand{\Premeasures}[2][X]{\operatorname{Pr}_{#2}\left(#1\right)}

\section{Introduction}
\label{sec:intro}

In this paper we introduce for a class of \emph{continuous} positively ordered monoids called $\CatW$-semigroups (along with a subclass called $\CatCu$-semigroups) a new concept of quotients, different from classical quotients by ideals, by forming the quotient on the level of a preorder relation. A need for forming such quotients naturally occurs if a group acts on such a semigroup and we want to identify elements in an orbit, thereby obtaining a new semigroup with a different order relation. Such a construction is not very hard to do with ordinary positively ordered monoids, as the additive and order structures both pass to quotients in a natural way, but for \emph{continuous} positively ordered monoids it is much more delicate, since the continuous structure demands extra care. 

Our main motivation to develop this theory lies in the theory of operator algebras, specifically, applications to the theory of
the \emph{Cuntz semigroup}, an invariant for C*-algebras. 
These analytical objects, sometimes dubbed as noncommutative topological spaces, have deep-rooted connections with topological dynamical systems, particularly via the construction of crossed product C*-algebras; hence,
our theory is, at a deeper level, motivated by the need to understand topological (as well as C*-) dynamical systems and their crossed products. 

It is worth pointing out that the Cuntz semigroup was originally introduced in \cite{Cun78DimFct} as a tool to study tracial states on C*-algebras (these are analogous to invariant measures on topological dynamical systems), and that   
the topological nature of C*-algebras gives rise to the \emph{continuous} structures on their Cuntz semigroups, on top of the structure of positively ordered monoids. 
This continuous structure distinguishes Cuntz semigroups from their more classical counterpart, namely the Murray-von Neumann semigroups. 

Renewed interest in the Cuntz semigroups grew\textemdash and continues to this day, with growing momentum\textemdash since the discovery of their role in the celebrated Elliott classification program of simple, amenable C*-algebras. 
Indeed, here the Cuntz semigroup first appeared in an antagonistic role\textemdash as an invariant to help establish a counterexample to the original classification conjecture \cite{Toms}. 
This landmark counterexample was a deciding factor in shaping the revised form of the classification conjecture, which was eventually verified in a series of groundbreaking articles \cite{TiWhWi17,GonLinNiu20ClassifZstable1,GonLinNiu20ClassifZstable2,EllGonLinNiu15arX:classFinDR2}. 
The Cuntz semigroup then plays two important roles in this and later developments: 
First, it provides the framework to formulate the notion of \emph{strict comparison}, a regularity property for C*-algebras that serves as a prerequisite for classifiability in the revised conjecture \cite{TiWhWi17}. 
Second, it has been useful in more recent developments aiming at classifying C*-algebras beyond the revised conjecture (\cite{AntPerRobThiDuke,EllLiNiu24}), and it serves as the framework to design novel invariants \cite{AnLiu23,AnLiu24,AnElliottLiu24}. 

Meanwhile, throughout the development of the Elliott classification program, inte\-rac\-tions with topological dynamical systems have been an integral part of the theory. 
It is often fruitful to analyze C*-algebraic properties and constructions through the lens of dynamical systems. 
In particular, the question of characterizing classifiability of crossed product C*-algebras in dynamical language has been and is under intensive research. 
To this end, there appears to be a close relationship between strict comparison and notions in topological dynamical systems such as the small boundary property and mean dimension zero \cite{KerrSza20}. 
One of the main applications of the theory developed in this paper is a notion of dynamical strict comparison and its close relation to a new regularity condition for topological dynamical systems, almost elementariness, which we will define and study in subsequent articles \cite{BosPerWuZacAE, BosPerWuZacAEDSC}.

To describe what we do in more detail, we need to take a closer look at the theory of Cuntz semigroups. 
It turns out that for each C*-algebra $A$, there are two intimately related constructions, denoted by $\CatW(A)$ and $\CatCu(A)$ respectively, each being useful in its own right and commonly referred to as the Cuntz semigroup. 
To clarify the relationship between the two, it is not enough to just treat these as positively ordered monoids \textemdash\ the continuous structure is crucial. 
To this end, a theory of \emph{continuous} (suitably interpreted) positively ordered monoids has been established \cite{CowEllIva08CuInv, AntBosPer11CompletionsCu, AntPerThi14arX:TensorProdCu}. 
This theory is centered around the category $\CatW$ of (abstract) $\CatW$-semigroups and a subcategory $\CatCu$ of (abstract) $\CatCu$-semigroups, together with a completion functor $\gamma \colon \CatW \to \CatCu$ that is adjoint to the canonical inclusion. In this framework, the assignments $A \mapsto \CatW(A)$ and $A \mapsto \CatCu(A)$ form functors from the category of C*-algebras to the categories $\CatW$ and $\CatCu$. Together with $\gamma$, these functors form a commuting diagram. (See \cite[Chapter 3]{AntPerThi14arX:TensorProdCu}.)
At a more technical level, 
the category $\CatW$ contains a class of abelian monoids equipped with an \emph{auxiliary relation} (it is called so since it is auxiliary to a natural preorder and is compatible with the monoid structure) that satisfy a set of natural axioms. Prototypical examples of $\CatW$-semigroups include:   
\begin{itemize}
	\item $\operatorname{Lsc}(X,\N)$, the lower semicontinuous non-negative integer valued functions on a compact Hausdorff space $X$,  
	\item $\CatW(A)$ for a C*-algebra $A$, 
	\item $[0,\infty)$ and $[0, \infty]$ with the usual order and addition, and with the auxiliary relation given by the usual strict order $<$. 
\end{itemize}
We recall the basics of both categories in the preliminaries and in \autoref{sec:ideals}. A useful perspective in this general theory is to understand how $*$-homomorphisms between C*-algebras give rise to $\CatCu$- and $\CatW$-\emph{morphisms}, i.e., order-preserving monoid homomorphisms compatible with the auxiliary relations. 

In the category $\CatCu$, there is a natural notion of $\CatCu$-ideal, which can be used to define quotients. This was studied in \cite{CiuRobSan10CuIdealsQuot} for C*-algebras, where it is shown that for any $\sigma$-unital ideal $I$ of a C*-algebra $A$, the natural short exact sequence $0\to I\to A\to A/I\to 0$ induces a short exact sequence in the category $\CatCu$. 
In other words, a surjective $*$-homomorphism between C*-algebras induces a surjective $\CatCu$-morphism with its kernel matching the $\CatCu$-semigroup of the C*-kernel. For abstract $\Cu$-semigroups, these notions were analyzed in \cite[5.1]{AntPerThi14arX:TensorProdCu}.

However, the full story about quotients in the categories $\CatW$ and $\CatCu$ is a more delicate topic. 
In particular, an injective $*$-homomorphism between C*-algebras usually does not induce an injective $\CatW$-morphism. This happens in many natural situations, and sometimes the $\CatW$-morphism one gets is surjective but non-injective, and thus it should be understood as a quotient map, but not by ideals. 

As a simple example, consider the natural diagonal $*$-homomorphism $ \mathbb{C} \oplus \mathbb{C} \hookrightarrow M_2(\mathbb{C})$. Applying the functor $\CatW$ (or the functor $\CatCu$) to it, we obtain the morphism $\mathbb{N} \oplus \mathbb{N} \to \mathbb{N}$, defined by $x\oplus y\mapsto x+y$. This is surjective and non-injective. As a quotient map, it is induced not by an ideal (the preimage of 0 is {0}), but merely by an equivalence relation. 
Keeping these sort of examples in mind, our first goal is:

\begin{goalintro}\label{goalintro:quotient}
	Develop a theory of general quotients for $\CatW$-semigroups to include this type of phenomenon.   
\end{goalintro} 

Turning our attention to dynamical systems, a pivotal construction relating topological dynamical systems and C*-algebras is that of a crossed product C*-algebra $C(X) \rtimes G$ associated to a group action $G \curvearrowright X$. 
More generally, a C*-dynamical system $G \curvearrowright A$ also gives rise to a crossed product $A \rtimes G$.
This construction 
forms a major bridge between the theory of dynamical systems and that of C*-algebras. 
Note that the above example of $ \mathbb{C} \oplus \mathbb{C} \hookrightarrow M_2(\mathbb{C})$ can be understood through such a construction: it is a special case of the natural embedding $C(X) \hookrightarrow C(X) \rtimes G$ arising from a topological dynamical system by a discrete group --- just take $G = \mathbb{Z}/2$ with its unique nontrivial action on $X = \mathbb{Z}/2$. 

As indicated above, a central problem of current focus is to transfer the groundbreaking results on classification of amenable C*-algebras to the world of topological dynamical systems.   
This would improve our understanding of the structure theory and classification of the latter, in ways similar to how ergodic theory of amenable groups benefited from Connes' classification theorem of simple amenable von Neumann algebras \cite{Con76}, and how the theory of Cantor minimal systems benefited from Elliott's classification theorem of AF algebras \cite{GioPutSkaCrelle95}. 

In view of the C*-regularity properties appearing in the aforementioned revised form of the classification program, Kerr put forth in \cite{KerrEur20} analogous dynamical properties for actions on compact metric spaces, i.e., abelian C*-algebras, namely \emph{finite tower dimension}, \emph{almost finiteness}, and \emph{dynamical strict comparison} (of open sets), and asks to what extent these conditions are equivalent to each other and to the classical notion of \emph{small boundary property} (or, alternatively, \emph{mean dimension zero}). 
A crucial ingredient in the definitions of dynamical strict comparison and almost finiteness is a notion of dynamical subequivalence\footnote{Such a notion had appeared earlier in lectures of Winter. } between subsets in a topological dynamical system. 
This notion was inspired by the so-called Cuntz subequivalence that appears in the construction of the Cuntz semigroup and can itself be used to formulate concepts of type semigroups (\cite[Section~13]{KerrEur20}, \cite{Ma21GeneralizedType}) that provide a natural framework to study dynamical strict comparison. 
These type semigroups can be considered as analogs of Cuntz semigroups in the setting of topological dynamical systems\footnote{More precisely, the type semigroup in \cite[Section~13]{KerrEur20} is closer in spirit to the Murray-von Neumann semigroup, while the generalized type semigroup in \cite{Ma21GeneralizedType} is closer in spirit to the Cuntz semigroup. }.

All this strongly motivates the definition and study of Cuntz semigroups associated to C*-dynamical systems, in a unifying and systematic fashion.  

\begin{goalintro}\label{goalintro:dynamical-Cuntz-semigroup}
	Construct analogs of Cuntz semigroups for C*-dynamical systems.
\end{goalintro}

Note that there is already a very natural object to study in this context, namely $\CatW(A \rtimes G)$ (or $\CatCu(A \rtimes G)$), the Cuntz semigroup of the crossed product associated to a C*-dynamical system. 
However, despite its significance, this semigroup is often difficult to concretely describe and compute. 
It is hoped that our construction of dynamical Cuntz semigroups, which avoids much of the analytic difficulty, can help us tap into the structure of $\CatCu(A \rtimes G)$. 
Since an action of a group $G$ on a C*-algebra $A$ naturally induces an action of $G$ on $\W(A)$, this motivates the following:

\begin{goalintro}\label{goalintro:crossed-product}
	Study the relationship between $\CatW(A \rtimes G)$ and the dynamical Cuntz semigroup of the action $G \curvearrowright \CatW(A)$.
\end{goalintro}

In the three subsections below, we describe how we fulfill \autoref{goalintro:quotient} and \autoref{goalintro:dynamical-Cuntz-semigroup} and make some progress in the direction of \autoref{goalintro:crossed-product}. 
For the sake of simplicity, we typically state simplified forms of the results that we actually prove. The reader may follow the references to find the complete versions.

\subsection{General (ideal-free) quotients of $\CatW$-semigroups} \label{subsec:intro-quotient}

We establish a solid and systematic foundation for the construction of quotient $\CatW$-semigroups in the first part of the paper going far beyond quotients by ideals and use this to tackle \autoref{goalintro:quotient}.

As shown in the example just before \autoref{goalintro:quotient}, this general notion of quotients is necessary since for the map $\N\oplus\N\to\N$ that takes $x\oplus y$ to $x+y$, the preimage of $0$ consists of only $0$. However, given a $\CatCu$-semigroup $S$ and a nontrivial ideal $I$ of $S$, the preimage of $0$ in $S/I$ is $I$. 

The cornerstones of this generalized theory are summarized below. 
To keep the presentation simple, let us assume that our $\CatW$-semigroups are \emph{faithful} in the sense that they satisfy \axiomW{2} and the induced preorder is a partial order\footnote{We are following the definition of $\CatW$-semigroups in \cite{AntPerThi20:JLMS}, which allows for preorders, as opposed to \cite{AntPerThi14arX:TensorProdCu}, which requires a partial order instead. When we refer to the latter case, we use the adjective \emph{faithful}. }. 
\begin{enumerate}
	\item As argued above, for a $\CatW$-morphism $f \colon S \to T$, the equivalence relation on $S$ induced by $f$ must be the congruence relation of the preorder $f^*(\leq_T)$, that is, the pullback of the partial order $\leq_T$ on $T$. Given the pivotal role the preorder $f^*(\leq_T)$ plays (it contains enough information to recover $\leq_T$, provided that $f$ is surjective), we call it the \emph{kernel} of $f$. 
	\item It is possible to characterize \emph{intrinsically} whether a preorder in a $\CatW$-semigroup $S$ arises as the kernel of some $\CatW$-morphism from $S$ to some other $\CatW$-semigroup. We say these preorders are \emph{normal}. 
	For technical convenience in dealing with the continuous structure in $\CatW$-semigroups, instead of working with preorders alone, we work with \emph{normal pairs}, modelled on the pair consisting of a normal relation and its induced preorder. 
	\item These normal pairs play the role of ideals in our generalized theory: for any $\CatW$-semigroup $S$ and any normal pair $\alpha$, we construct a \emph{quotient $\CatW$-semigroup} $S / \alpha$. 
\end{enumerate}

As suggested by the terminology, there is a strong parallelism with the theory of group homomorphisms and quotients. Indeed, we have the following analog of the fundamental theorem of group homomorphisms. 

\begin{thmintro}[{see \autoref{thm:fundamental-theorem0} and \autoref{rmk:fundamental-theorem}}]\label{thmintroA}
	Let $S$ and $T$ be $\W$-semigroups with $T$ {faithful}. Let $f \colon S \to T$ be a $\W$-morphism. Let $\alpha$ be a normal pair on $S$, and let $\pi_\alpha \colon S \to S/{\alpha}$ be the natural quotient $\W$-morphism (as in \autoref{prp:normal}). Then there is an order-preserving $\W$-morphism $h \colon S_{\alpha} \to T$ such that 
	the diagram
	\[
	\xymatrix{
		S \ar@{->>}[d]_{\pi_\alpha} \ar[rd]^{f} &   \\
		S/{\alpha} \ar@{-->}[r]_{h} & T
	}
	\]
	commutes 
	if and only if $\alpha\leq\ker(f)$. Furthermore, $\alpha=\ker(f)$ if and only if $h$ is an order-embedding.
\end{thmintro}

This theory perfectly models the behaviour of the functor $\Cu$ from the category of C*-algebras to the category $\CatCu$. In this situation, we get the following:
\begin{thmintro}[{see \autoref{thm:Cu-quotients-from-star-homomorphism}}] \label{thmintroB}
	
	Let $\varphi \colon A \to B$ be an injective $*$-homomorphism between C*-algebras. 
	Then the resulting $\operatorname{Cu}$-morphism can be decomposed as follows: 
	\[
	\xymatrix{
		\operatorname{Cu}(A) \ar[rrr]^{\operatorname{Cu}(\varphi)} \ar@{->>}[rd] & & & \operatorname{Cu}(B) \\
		& \operatorname{Cu} (A) / \ker(\operatorname{Cu}(\varphi))  \ar[r]^-{\cong} &   \left\{ \left[ \varphi(a) \right] \in \operatorname{Cu}(B) \colon a \in (A \otimes K)_+ \right\}  \ar@{^{(}->}[ru] & 
	}
	\]
	where the arrows $\twoheadrightarrow$ and $\hookrightarrow$ indicate surjectivity and injectivity, respectively. 
\end{thmintro}

In fact, we deal with the general case of possibly non-injective $*$-homomorphisms in \autoref{thm:Cu-quotients-from-star-homomorphism}, where ones sees how our notion of generalized quotients contrasts with the traditional kind of quotients by ideals.

To apply the theory developed above to the dynamical setting (to be discussed in Subsection~\ref{subsec:intro-dynamical}), we explain a mechanism to generate normal pairs from an arbitrary relation  that satisfies a continuity assumption. 
Again, the idea here mimics what happens in group theory. One may obtain a quotient group by modding out certain relations, which is achieved technically by first generating a normal subgroup using these relations. As a pedagogical byproduct, this mechanism allows us to give an abstract characterisation of the Cuntz subequivalence in \autoref{cor:precsimher-axiomatic}. 
Namely, for a C*-algebra $A$, the Cuntz subequivalence $\precsim_{\operatorname{Cu}}$ is the smallest preorder $\leq$ on the set $A_+$ of positive elements such that for all $a, b \in A_+$, any of the following conditions implies $a \leq b$:
\beginEnumConditions
\item $\mathrm{Her}(a) \subseteq \mathrm{Her}(b)$, where we write $\mathrm{Her}(a) = \overline{aAa}$, the hereditary subalgebra generated by $a$; 
\item there exists $x \in A$ such that $a = x b x^*$; 
\item for any $c \in \bigcup_{\varepsilon > 0} \mathrm{Her}( (a - \varepsilon)_+ )$, 
we have $c \leq b$, where $(a - \varepsilon)_+$ stands for the positive part of the self-adjoint element $(a - \varepsilon \cdot 1)$. 
\end{enumerate}
This answers a question raised by G.~A.~Elliott. 
Note that the last condition above points to an essential feature of the continuous structure in Cuntz semigroups and is analogous to (inner) regularity in measure theory. 

Said mechanism is developed in \autoref{sec:GenNormalPreorders} and \autoref{sec:extendingnormal}, where we begin with an axiomatic approach to define the normal pair generated by a given relation, in the spirit of the above characterization of $\precsim_{\operatorname{Cu}}$, but then also give more explicit characterisations (see \autoref{thm:concrete-merging} and \autoref{cor:one-step}), in the spirit of the classical definition of the Cuntz subequivalence $\precsim_{\operatorname{Cu}}$, i.e., $a \precsim_{\operatorname{Cu}} b$ if and only if $a = \lim_{n} x_n b x_n^*$ for some sequence $x_n \in A$. 
In fact, it is via the explicit constructions that we show the pair given by the axiomatic approach is indeed normal. 

The reason we divide this task into two sections is the following: 
In the former section, we only work with the continuous order structure in $\CatW$-semigroups, disregarding the additive structure (in more technical terms, we work with $\CatW$-sets). The additive structure is blended in only in the latter section, since this adds an extra layer of complexity related to transitivity (see \autoref{pgr:almost-refinement}). 

These results provide us with a suitable framework to define the dynamical Cuntz semigroup. 
This, as we will subsequently explain, is the largest (general) quotient $\CatW$-se\-mi\-group in which each element is identified with all its translates.

\subsection{Dynamical Cuntz semigroups} \label{subsec:intro-dynamical}

As mentioned in the first part of the introduction, our approach to \ref{goalintro:dynamical-Cuntz-semigroup} is to work in the framework of $\CatW$-\emph{semigroups}. In this setting, recall that an action of a group $G$ on a C*-algebra $A$ induces an action of $G$ on $\CatW(A)$ by automorphisms in $\CatW$. These are examples of what we call \emph{$\CatGW$-semigroups}.

Given a faithful $\CatGW$-semigroup $S$ equipped with a preorder $\leq$ and an auxiliary relation $\prec$, we use our results to determine the strongest preorder  $\lessapprox$ on $S$ satisfying that, for all $a, b \in S$, any of the following conditions implies $a \lessapprox b$ 
(note the similarity between the first three condition below and the three conditions above characterising $\precsim_{\operatorname{Cu}}$): 
\label{listintro:axiomatic-dynamical-subequivalence}
\beginEnumConditions
\item $a \leq b$; 
\item $a = g \cdot b$ for some $g \in G$; 
\item for any $c \in S$ with $c \prec a$, we have $c \lessapprox b$;
\item there are $a_1, a_2, b_1, b_2$ in $S$ such that $a = a_1 + a_2$, $b = b_1 + b_2$, and $a_k \lessapprox b_k$ for $k = 1, 2$.
\end{enumerate}
This preorder is denoted by $\dyn$ and it fits into a normal pair $(\prec, \dyn)$. 
We define $S/G$ to be the quotient $\CatW$-semigroup by this normal pair, and refer to this as the \emph{dynamical Cuntz semigroup}. In other words, 
$S/G$ is the quotient set of $S$ by the symmetrization of $\dyn$, equipped with 
an auxiliary relation induced from~$\prec$. 

Note that restricting to abelian C*-algebras $A$, our construction agrees with the dynamical subequivalence defined by Kerr (see \cite{KerrEur20}). Further, our theory can be used to give a somewhat different construction of Ma's generalized type semigroup (\cite{Ma21GeneralizedType}). This will be explored in a follow-up paper \cite{BosPerWuZacTypeSemigroup}.

The semigroup $S/G$ may also be characterized by the following universal property:

\begin{thmintro}[{see \autoref{thm:dynamical-universal-property}}]\label{thmIntroC}
Let $S$ be a $\CatGW$-semigroup and let $T$ be a faithful $\W$-semigroup. Then for any $G$-invariant $\W$-morphism $f \colon S \to T$, there is a unique order-preserving $\W$-morphism $f_G \colon S / G \to T$ such that the following diagram commutes:  
\[
\xymatrix{
S \ar[d]_{\pi_G} \ar[rd]^{f} &  \\
S / G \ar@{-->}[r]_{f_G} & T
}
\]
Moreover, the pair $(S / G, \pi_G)$ is the unique pair of a $\W$-semigroup and a $\W$-morphism that satisfies the above universal property. 
\end{thmintro}

This universal property leads to many nice consequences, mainly formulated in catego\-ry-theoretical language, and we will summarize them briefly:  Firstly, considering the category $\CatGW$ of $\CatGW$-semigroups and G-equivariant $\CatW$-morphisms, and using the embedding $\CatW$ into $\CatGW$ using trivial actions, the universal property yields a \emph{reflector}, i.e., a functor left adjoint to the inclusion functor $\CatW \hookrightarrow \CatGW$ (see \autoref{thm:categories-W-full-reflective-in-W-G}). 
Secondly, Theorem \ref{thmIntroC} can be used to verify that the construction of the dynamical Cuntz semigroup is compatible with the $\gamma$-completion functor from $\CatW$ to $\CatCu$ studied in \cite{AntPerThi14arX:TensorProdCu}, in the sense that the completion of the dynamical Cuntz semigroup of a $\CatGW$-semigroup $S$ and the one corresponding to $\gamma(S)$ with the induced action yield the same object in $\CatCu$ (see \autoref{cor:categories-Cu-full-reflective-in-W-G}). 
Finally, we also use Theorem \ref{thmIntroC} to show that there is a bijection between the set of G-invariant functionals on $S$ and the set of functionals on $S/G$. In particular, there are natural topologies on both sets, under which this bijection becomes a homeomorphism (see Proposition \ref{prp:WfunCufun}).

If the group $G$ acts on a C*-algebra $A$, then we write $\W_G(A)=\W(A)/G$ and $\Cu_G(A)=\gamma(\W(A)/G)=\gamma(\Cu(A)/G)$, where the last equality is an instance of the second point mentioned in the paragraph above.~Note that if $I$ is a $G$-invariant closed ideal, the action passes to the quotient C*-algebra  $A/I$.~We prove that the dynamical Cuntz semigroup construction is compatible with quotients, as follows:
\begin{thmintro}[{see \autoref{thm:dynquotients}}]\label{thmIntroD}	
Let $A$ be a C*-algebra, and let $G$ be a discrete group acting on $A$. For any $G$-invariant ideal $I$ of $A$, we have that $\Dyn(I)$ and $\CuDyn(I)$ are (isomorphic to) ideals of $\Dyn(A)$ and $\CuDyn(A)$, respectively. Moreover,
\[
\Dyn(A)/\Dyn(I)\cong\Dyn(A/I) \text{ and }\CuDyn(A)/\CuDyn(I)=\CuDyn(A/I).
\]
\end{thmintro}

\subsection{Applications: dynamical strict comparison and the Cuntz semigroup of a crossed product} \label{subsec:intro-DSC}

Recall that a simple C*-algebra $A$ has \emph{strict comparison of positive elements} if, for any $a,b\in \W(A)$, we have $a\leq b$ whenever $\lambda(a)<\lambda(b)$ for any functionals $\lambda$ on $\W(A)$.~This can be suitably generalized to non-simple C*-algebras, and it is known to be equivalent to the so-called property of almost unperforation in $\W(A)$.

Now, if a discrete group $G$ acts on $A$,  there is a natural concept of dynamical strict comparison, this time using $G$-invariant functionals on $\CatW(A)$,~which we introduce in \autoref{sec:comparison}.~We show that, as in the case of no group action, ~this property can be characterized by the property of almost unperforation in $\W_G(A)$.

This allows us to make progress regarding  \autoref{goalintro:crossed-product}. If $A\rtimes G$ denotes the crossed product by the action of $G$ and $\iota\colon A\to A\rtimes G$ is the natural embedding, the universal property of the dynamical Cuntz semigroup 
yields the following commutative diagram:
\[
\xymatrix{
\CatCu(A) \ar[rd]^-{\iota} \ar[d]_{\pi_G} &   \\
\CuDyn(A) \ar@{-->}[r]_-{\kappa} &	\CatCu(\iota(A))\subseteq\CatCu(A\rtimes G)
}
\] 
where $\kappa$ is induced by $\iota$.

A natural problem is to understand how close is the $\CatCu$-morphism $\kappa$ to being an isomorphism. By a direct computation, one can show that $\kappa$ is an isomorphism  in the aforementioned example $\mathbb{C} \oplus \mathbb{C}\cong C(\Z/2)\hookrightarrow C(\Z/2)\rtimes\Z/2\cong M_2(\mathbb{C})$ and, more generally, for $G$-C*-algebras of the form $C_0(G) \otimes F$, where $F$ is a finite dimensional C*-algebra and $G$ acts by translation on the first factor and trivially on the second factor. 
The latter kind of C*-dynamical systems may be termed \emph{elementary}: in particular, they involve elementary C*-algebras and give rise to elementary crossed products. 
In the follow-up papers (\cite{BosPerWuZacAE,BosPerWuZacAEDSC}), we introduce a C*-dynamical regularity property termed \emph{almost elementariness} requiring the existence of an approximation of the action by \emph{elementary} systems as above, up to a dynamical small remainder. We will show in there that almost elementariness implies dynamical strict comparison.

At the current stage, it appears difficult to give general solutions to the natural problem above, e.g., characterizing when $\kappa$ is an order embedding or when it has a dense image. 
However, we expect answers in the positive direction under suitable regularity properties on the C*-dynamical systems. 
We prove one such result, using dynamical strict comparison:

\begin{thmintro}[{see \autoref{thm:DynamicalOrderEmbedding}}]\label{thmintroE} 
Let $G$ be a discrete group acting minimally on a C*-algebra $A$. If $A$ has dynamical strict comparison, then the restriction of $\kappa \colon \CuDyn(A) \to \CuDyn(A \rtimes G)$ to the soft elements of $\CuDyn(A)$ is an order embedding.
\end{thmintro}	
See \autoref{sec:comparison} for the definition of soft elements.

\section{Preliminaries}

We begin our discussion in a generality greater than semigroups, by allowing preorders instead of partial orders, and by disregarding the additive structure.

\begin{pgr}[Relations]\label{dfn:relations}
	Recall that a \emph{relation} on a set $X$ is a subset of $X \times X$. Given a relation $R$ on $X$ and $a, b \in X$, we sometimes write $a R b$ for the expression $(a,b) \in R$. For the sake of clarity, given two relations $R$ and $R'$ on $X$, we write $R \relsubseteq R'$ for the containment $R \subseteq R'$ as subsets of $X \times X$, and $R \releq R'$ for $R = R'$, i.e., the two relations are identical.
	
	Let $R_1, \ldots, R_n$ be relations on a set $X$. Then the \emph{composition} $R_1 \circ \cdots \circ R_n$ is the relation on $X$ defined so that for any $a,b \in X$, we have $(a,b) \in R_1 \circ \cdots \circ R_n$ if and only if there are $c_1, \ldots, c_{n-1} \in X$ such that $(a,c_1) \in R_1$, $(c_{n-1}, b) \in R_{n}$ and $(c_{k-1}, c_{k}) \in R_k$ for $k = 2, \ldots, n-1$.
	
	When all the relations in the above are equal to a single relation $R$, we may use the abbreviation $R^{\circ n}$. The notation $R^{\circ 0}$ stands for the identity relation on $X$. 
	
	We also define the \emph{preorder closure} of $R$ to be the following relation: 
	\[
	R^{\circ \mathbb{N}} \releq \bigcup_{n=0}^{\infty} R^{\circ n} \; .
	\]
	Let $A$ and $B$ be two subsets of $X$. Then $A$ is \emph{cofinal in $B$ with regard to $R$} (or simply \emph{$R$-cofinal} for $B$) if, for any $b \in B$, there is $a \in A$ such that $(b,a) \in R$. Note that we do not require $A \subseteq B$.
\end{pgr}


\begin{rmk}\label{rmk:composition-basics}
	It is clear that $(R_1 \circ R_2) \circ R_3 \releq R_1 \circ R_2 \circ R_3 \releq R_1 \circ (R_2 \circ R_3)$. In particular, we have $R^{\circ m} \circ R^{\circ n} \releq R^{\circ (m + n)}$. Also, if $R$ is a preorder, then $R \circ R \releq R$. In fact, a relation $R$ is transitive if and only if we have $R \circ R \relsubseteq R$. It follows that the preorder closure is the smallest preorder that contains $R$. 
\end{rmk}


The following accounts for compatibility of relations with a possibly existing additive structure.

\begin{pgr}\label{dfn:additive-closure}
	Let $(S, +, 0)$ be an abelian monoid. All monoids in this paper will be abelian, written additively, with neutral element denoted by $0$. Let $R$ and $R'$ be two relations on $S$. Then we write $R + R'$ for the relation on $S$ given by 
	\[
	\left\{ (a + a', b + b') \in S \times S \colon (a, b) \in R \text{ and } (a', b') \in R' \right\} \; .
	\] 
	
	A relation $R$ on $S$ is said to be \emph{additively closed} if $R + R \relsubseteq R$, or equivalently, for any $a, a', b, b' \in S$ satisfying $a' R a$ and $b' R b$, we have $(a' + b') R (a + b)$. 
	
	The \emph{additive closure} of a relation $R$, denoted by $R_+$, is the relation on $S$ equal to the union of all relations of the form $R + \cdots + R$, or equivalently, it is defined so that for any $a, b \in S$, we have $a R_+ b$ if and only if there are $a_1, \ldots, a_n$ and $b_1, \ldots, b_n$ in $S$ such that $a = a_1 + \cdots + a_n$, $b = b_1 + \cdots + b_n$, and $a_k R b_k$ for $k = 1, \ldots, n$.
	
	It is clear that $R_+$ is the smallest additively closed relation that contains $R$. 
\end{pgr}

We now recall the axioms used to define the category $\W$. This category was originally introduced in \cite{AntPerThi14arX:TensorProdCu} in order to model the classical Cuntz semigroup $\W(A)$ of a local C*-algebra $A$. It will be our basic framework in which we will develop our construction of ideal free quotients and, in particular, of the dynamical Cuntz semigroup.

For this we first need to recall the notion of auxiliary relation:
\begin{pgr}[Auxiliary relations {\cite[Definition I-1.11]{GieHof+03Domains}}]
	\label{pgr:auxiliary}
	Let $(X,\leq)$ be a preordered set. A relation $\prec$ on $X$ is said to be \emph{auxiliary} for $\leq$ if it satisfies the following conditions:
	\beginEnumConditions
	\item If $a\prec b$, then  $a\leq b$, for $a,b\in X$.
	\item If $a\leq b\prec c\leq d$, then $a\prec d$ for any $a,b,c,d\in X$.
\end{enumerate}
It is easy to check that $\prec$ is automatically transitive. In case $X$ is furthermore a preordered monoid, then $\prec$ is \emph{additive} in case it is additively closed and $0\prec a$ for all $a\in X$.

Now, given a relation $\prec$ on a set $X$ and $a\in X$, let us write 
\[
	a^\prec=\{b\in X\colon b\prec a\}
\]
and define a preorder $\leq_\prec$ on $X$ by stating $a\leq_\prec b$ if and only if $a^\prec\subseteq b^\prec$. We shall refer to $\leq_\prec$ as the \emph{induced preorder} on $X$ by the relation $\prec$. 
\end{pgr}

\begin{pgr}[Compact containment]
\label{dfn:preorder-way-below}
A standard way to produce an auxiliary relation from a preorder is the following.

Let $X$ be a set, let $\leq$ be a preorder on $X$, and suppose that there is a smallest element of $(X,\leq)$, that we denote by $0$. For any $a, b \in X$, we say $a$ is \emph{compactly contained in} (or \emph{way below}) $b$, and write $a \ll b$, if whenever $(b_n)_{n \in \mathbb{N}}$ is an increasing sequence in $(X, \leq)$ whose  supremum $\sup_n^{\leq} b_n$ with regard to $\leq$ exists and $b \leq \sup_n^{\leq} b_n$, we have $a \leq b_{n_0}$ for some $n_0 \in \mathbb{N}$.

It is routine to verify that the compact containment relation $\ll$ is auxiliary for $\leq$ and that moreover $0\ll a$ for any $a\in X$.
\end{pgr}


\begin{pgr}[$\CatW$-axioms]
\label{pgr:W-axioms} 
Let $X$ be a set and let $\prec$ be a transitive relation on $X$. We recall four properties that $\prec$ may satisfy, that were already considered in \cite{AntPerThi14arX:TensorProdCu}:
\begin{itemize}
	\item[\axiomW{1}] For any $a \in X$, there exists a sequence $(a_k)_k$ in $a^{\prec_{\phantomtoggle{P}}}$ such that $a_k\prec a_{k+1}$ for any $k\in \N$, and $(a_k)_k$ is cofinal in $a^{\prec_{\phantomtoggle{P}}}$ with regard to $\prec_{\phantomtoggle{P}}$, i.e., for any $b \in a^{\prec_{\phantomtoggle{P}}}$, there is $k \in \mathbb{N}$ such that $b \prec_{\phantomtoggle{P}} a_k$. 
	\item[\axiomW{2}](Suppose that $X$ has a preorder $\leq$ for which $\prec$ is auxiliary.) For any $a \in X$, we have $a = \sup^{\leq_{\phantomtoggle{P}}} a^{\prec_{\phantomtoggle{P}}}$, i.e., for any $c \in X$ satisfying $b \leq_{\phantomtoggle{P}} c$ for any $b \in a^{\prec_{\phantomtoggle{P}}}$, we have $a \leq_{\phantomtoggle{P}} c$.
	\item[\axiomW{3}] (Suppose that $X$ is a monoid.) For any $a, a', b, b' \in X$ satisfying $a' \prec a$ and $b' \prec b$, we have $a' + b' \prec a + b$.
	\item[\axiomW{4}] (Suppose that $X$ is a monoid.) For any $a, b, c \in X$ satisfying $a \prec b + c$, then there exist $b', c' \in X$ such that $a \prec b' + c'$, $b' \prec b$ and $c' \prec c$. 
\end{itemize}
Notice that \axiomW{1} above implies, in particular, that the relation $\prec$ satisfies $\mathrel{{\prec}\relsubseteq{\prec\circ\prec}}$; see \autoref{dfn:dense}. Since it is also transitive, we get ${\prec_{\phantomtoggle{P}}}  \releq   {\prec_{\phantomtoggle{P}}}  \circ   {\prec_{\phantomtoggle{P}}}$.

Notice also that \axiomW{2} implies that the preorder $\leq_\prec$ induced by $\prec$ is stronger than $\leq$. Indeed, if $a^\prec\subset b^\prec$, then $c\leq b$ for any $c\prec a$, hence $a\leq b$.
\end{pgr}


\begin{pgr}[$\W$-sets and $\W$-semigroups]\label{pgr:W-preorder}
In this paper we shall also work with pointed sets, where the distinguished element is denoted by $0$.

A pointed set $(X,0)$ is a \emph{$\W$-set} provided there is a transitive relation $\prec$ on $X$ such that $X$ satisfies \axiomW{1} and $0\prec a$ for any $a\in X$. We shall write $(X,0,\prec)$ to refer to a $\W$-set. 

A \emph{morphism} $f\colon(X,\prec)\to (Y,\preccurlyeq)$ between sets equipped with relations (hence, in particular, of $\W$-sets) is a \emph{monotone} map (that is, $f(a)\preccurlyeq f(b)$ in $Y$ whenever $a\prec b$ in $X$) which is also \emph{continuous} in the following sense:
	\[
	\text{For any } a\in X\text{ and } b\in f(a)^\preccurlyeq \text{ there is }a'\in a^\prec \text{ such that } b\preccurlyeq f(a').
	\]
(We also require that $f(0)=0$ in case $X,Y$ are pointed with distinguished element $0$.)

In case our set has a monoid structure, then we shall require compatibility properties between addition and $\prec$. A \emph{$\W$-semigroup} is an abelian monoid $S$ together with an additive, transitive relation $\prec$, such that $(S,0,\prec)$ is a $\W$-set that satisfies \axiomW{4}. A $\W$-morphism between $\W$-semigroups is morphism as $\W$-sets which is moreover additive. We will usually write $(S,\prec)$ to refer to a $\W$-semigroup.
\end{pgr}

Our prototypical examples of $\W$-preorders will be given on the set of positive elements of (local) C*-algebras. Among these examples, the following is in a sense the most fundamental.

\begin{exa}\label{exa:precsimher}
Let $A$ be a C*-algebra. For any $a\in A_+$, denote $\her(a)=\overline{aAa}$. We define a preorder on $A_+$, the set of positive elements, as follows:  $a \leqher b$ if and only if $a \in \her (b)$ or, equivalently, if $\her(a)\subseteq \her (b)$. In other words, $\leqher$ is induced from the set containment relation on singly generated hereditary subalgebras. Let $\llher$ be the associated compact containment relation. Then $(A_+,\leqher, \llher)$ is a $\W$-set. Indeed, for any positive element $a$ and any increasing sequence $(a_n)_{n \in \mathbb{N}}$ in $A_+$, we make the following observations:
\beginEnumConditions
\item\label{exa:precsimher-enum:union} if $\overline{a A a} = \overline{ \bigcup_{k=1}^{\infty} \overline{a_n A a_n}}$, then $a$ is a supremum of $(a_n)_{n \in \mathbb{N}}$;
\item\label{exa:precsimher-enum:sum} the element $\sum_{k=1}^{\infty} \frac{a_n}{2^n \|a_n\|}$ satisfies the above condition, and thus suprema of increasing sequences always exist;
\item\label{exa:precsimher-enum:cutdowns} the increasing sequence $\left( \left(a-\frac{1}{n}\right)_+ \right)_{n \in \mathbb{N}}$ has $a$ as a supremum;
\item\label{exa:precsimher-enum:way-below} for any positive element $b$, we have $a \ll_{\operatorname{her}} b$ if and only if there exists $\delta > 0$ such that $a \leqher (b - \delta)_+$;
\end{enumerate}
Hence increasing sequences as constructed in \ref{exa:precsimher-enum:cutdowns} witness Axioms~(W1) and~(W2).
\end{exa}

The prototypical example of $\W$-semigroup will be given by the Cuntz semigroup of a (local) C*-algebra. We recall the construction below; see \cite{Cun78DimFct}.
\begin{exa}
Let $A$ be a C*-algebra. For positive elements $a,b\in A_+$, write $a\precsimr{Cu}b$ if, given $\varepsilon>0$, there is $x\in A$ such that $a\approx_\varepsilon xbx^*$ (meaning that $\| a - xbx^* \| < \varepsilon$). We say in this case that $a$ is \emph{Cuntz subequivalent} to $b$. Write $a\simr{Cu} b$ provided $a\precsimr{Cu}b$ and $b\precsimr{Cu}a$.

Set $\W(A)=M_\infty(A)_+/{\simr{Cu}}$. Then $\W(A)$ is a $\W$-semigroup with addition given by $[a]+[b]=[\left(\begin{smallmatrix}
a & 0 \\ 0 & b
\end{smallmatrix}\right)]$, order induced by $\precsimr{Cu}$, and auxiliary relation defined by $[a]\precr{Cu}[b]$ provided $a\precsimr{Cu} (b-\delta)_+$ for some $\delta>0$.
\end{exa}	


\begin{rmk}
As observed in \cite[Remark 2.6]{AntPerThi20:JLMS}, the general theory of $\W$-se\-mi\-groups can be carried out without assuming \axiomW{2}, as this can always be enforced. We shall use this approach, which is why we haven't required said axiom in our definition.
\end{rmk}	

\begin{pgr}[Antisymmetrizations]
\label{pgr:anti}	
In various situations below we will work with partial orders, mostly in connection with semigroups arising from C*-algebras. Observe that the use of preorders instead of partial orders is only a matter of technical convenience, rather than any essential mathematical content. Indeed, given a preorder $\leq$ on a pointed set $(X, 0)$, it is clear from the definition that $\leq$ induces a partial order on the quotient $(X /{\leq}_{\phantomtoggle{P}}, [0])$, where $a\equiv b$ if and only if $a\leq b$ and $b\leq a$. We shall denote the equivalence classes by $[a]$ for $a\in X$. 
\end{pgr}

\section{Normality}\label{sec:Prenormality}

In this section we introduce the notion of normality for relations on sets equipped with a so-called dense transitive relation, and we show that these are the suitable relations for quotients in our categories to work. We term these relations \emph{prenormal}, as the terminology \emph{normal} is reserved for relations that are moreover additively closed. We first recall the notion of density.

\begin{pgr}[Dense and continuous relations]
	\label{dfn:dense}
	A relation $\prec$ on a set $X$ is \emph{dense} provided that $\mathrel{{\prec}\relsubseteq{\prec}\circ{\prec}}$. In other words, given $a,b\in X$, there is $c\in X$ such that $a\prec c\prec b$. Dense, transitive relations are sometimes called \emph{idempotent}.

Now, let $f\colon(X,\prec)\to (Y,\preccurlyeq)$  be a morphism (in the sense of \autoref{pgr:W-preorder}). The pull-back relation of $\preccurlyeq$ through $f$ is a relation on $X$, denoted by $f^*(\preccurlyeq)$, and defined by $af^*(\preccurlyeq)b$ if $f(a)\preccurlyeq f(b)$. Using monotonicity and continuity, one may check that the pull-back relation $f^*(\preccurlyeq)$ is weaker than $\prec$ and dense. Further, continuity also implies that, if we write $\llcurly=f^*(\preccurlyeq)$, then for any $a\in X$, the set $a^\prec$ is $\llcurly$-cofinal in $a^\llcurly$.
	
	If $(X,\prec)$ is a set equipped with a relation, we say that another relation $\preccurlyeq$ on $X$ is \emph{left $\prec$-continuous} provided $\mathrel{{\prec}\circ{\preccurlyeq}\relsubseteq{\prec}\circ{\preccurlyeq}\circ{\prec}}$. This is equivalent to saying that the identity map $(X,\prec)\to (X,\mathrel{{\prec}\circ{\preccurlyeq}})$ is continuous. If $\preccurlyeq$ is moreover weaker than $\prec$, then being left $\prec$-continuous is equivalent to the fact that the identity map is a morphism as defined above.
\end{pgr}


The following lemma characterizes when a dense transitive relation is a pull-back relation. We only provide a few details of its proof.
\begin{lma}
	\label{lma:pullback_char}
	Let $\prec$ be a dense transitive relation on a set $X$ and let $\preccurlyeq$ be another relation on $X$. Then, the following conditions are equivalent:
	\beginEnumConditions
	\item The relation $\preccurlyeq$ is the pull-back of a dense transitive relation $\prec'$ on a set $Y$ through a morphism $f\colon (X,\prec)\to (Y,\prec')$.
	\item The relation $\preccurlyeq$ is dense, transitive,  and the identity map $(X,\prec)\to (X,\preccurlyeq)$ is a morphism.
	\item The relation $\preccurlyeq$ is transitive, weaker than $\prec$, and satisfies that $a^\prec$ is $\preccurlyeq$-cofinal in $a^\preccurlyeq$ for any $a\in X$.
	\item The relation $\preccurlyeq$ is transitive, weaker than $\prec$, and satisfies $\mathrel{{\preccurlyeq}={\preccurlyeq}\circ{\prec}}$.
	\item The relation $\preccurlyeq$ is transitive, weaker than $\prec$, satisfies $\mathrel{{\preccurlyeq}={\preccurlyeq}\circ{\prec}\circ{\preccurlyeq}}$, and is left $\prec$-continuous.
\end{enumerate}
\end{lma}
\begin{proof}
(i)$\implies$(ii): If there is a set $(Y,\prec')$ with a dense relation $\prec'$ and a morphism $f\colon(X,\prec)\to (Y,\prec')$ such that $\preccurlyeq=f^*(\prec')$, then clearly $a\prec b$ implies that $f(a)\prec' f(b)$, that is, $a\preccurlyeq b$. Thus the identity map $(X,\prec)\to (X,\preccurlyeq)$ is monotone. If now $a\preccurlyeq b$, then $f(a)\prec'f(b)$ and, since $f$ is continuous, there is $b_0\prec b$ such that $f(a)\prec' f(b_0)$. This means that $a\preccurlyeq b_0$ and $b_0\prec b$, that is, the identity map is continuous.

(ii)$\implies$(iii): That $\preccurlyeq$ is weaker than $\prec$ follows from monotonicity of the identity, whilst the fact that $a^\prec$ is $\preccurlyeq$-cofinal in $a^\preccurlyeq$ is a consequence of continuity.

(iii)$\implies$(iv): If $a\preccurlyeq b$, then by cofinality $a\preccurlyeq b'$ with $b'\prec b$, and thus $a\mathrel{{\preccurlyeq}\circ{\prec}}b$. The rest of assertions follow trivially.

(iv)$\implies$(v): This is a routine check.

(v)$\implies$(i): The assumptions imply that the identity map $(X,\prec)\to (X,\preccurlyeq)$ is a morphism, whose pull-back relation is clearly $\preccurlyeq$. We further have, using also that $\preccurlyeq$ is weaker than $\prec$, that $\mathrel{{\preccurlyeq}={\preccurlyeq}\circ{\prec}\circ{\preccurlyeq}\relsubseteq{\preccurlyeq}\circ{\preccurlyeq}}$. Thus $\preccurlyeq$ is dense, as desired.
\end{proof}

\begin{pgr}[Prenormality]
\label{pgr:prenormal_closed_preorder}
Let $(X,\prec_{\phantomtoggle{P}})$ be a set equipped with a dense transitive relation. We say that another transitive relation $\preccurlyeq$ on $X$ is \emph{$\prec$-prenormal}, provided it satisfies any of the equivalent conditions of \autoref{lma:pullback_char}. We will mostly use condition (iv), that is, $\preccurlyeq$ is weaker than $\prec$ and also $\mathrel{{\preccurlyeq}={\preccurlyeq\circ\prec}}$.  Notice that $\prec$ is clearly $\prec$-prenormal. It is not difficult to verify that prenormal relations as just defined satisfy the following universal property, whose proof we omit:

If $(X,\prec_{\phantomtoggle{P}})$ is a set equipped with a dense transitive relation and $\preccurlyeq$ is a $\prec$-prenormal relation on $X$, then given any other set $(Y,\prec'_{\phantomtoggle{P}})$ with a dense transitive relation $\prec'$ and a monotone continuous map $f\colon (X,\prec)\to (Y,\prec')$, there is a unique monotone continuous map $\overline{f}\colon (X,\preccurlyeq)\to (Y,\prec')$ such that the diagram
\[
\xymatrix{
	(X,\prec) \ar[d]_{\mathrm{id}} \ar[rd]^{f} &  \\
	(X,\preccurlyeq) \ar@{-->}[r]_{\overline{f}} & (Y,\prec')
}
\]
commutes (that is, $f=\overline{f}\circ\mathrm{id}$ as morphisms) if and only if $\mathrel{{\preccurlyeq}\relsubseteq{f^*(\prec')}}$. (Indeed, the map $\overline{f}$ is defined by $\overline{f}(a)=f(a)$.)
\end{pgr}

\begin{pgr}[Admissible pairs]
\label{pgr:admissible}
Let $(X,\prec)$ be a set equipped with a dense transitive relation $\prec$. Given a preorder $\leq$ on $X$, we say that the pair $\alpha=(\prec,\leq)$ is \emph{admissible} provided that $a\leq b$ whenever $a^{\leq\circ\prec}\subseteq b^{\leq\circ\prec}$. Notice that $a^\prec\subseteq b^\prec$ always implies that $a^{\leq\circ \prec}\subseteq b^{\leq\circ\prec}$.

Perhaps the most natural example of an admissible pair consists of the pair $\alpha_\prec:=(\prec, \leq_{\prec})$, which we call \emph{minimal admissible}. To see this, we use temporarily $\leq$ for $\leq_\prec$, and we need to check that $a^{\leq\circ\prec}\subseteq  b^{\leq\circ\prec}$ implies $a\leq b$. Thus, let $x\prec a$, and choose $y$ such that $x\prec y\prec a$. As $y\leq y\prec a$, there is $z$ with $y\leq z\prec b$ and since $x\prec y$ we have $x\prec z\prec b$, hence $x\prec b$. Since $x$ was arbitrary, we conclude that $a\leq b$.

On the other hand, if $\prec$ is auxiliary for $\leq$, then $\alpha=(\prec,\leq)$ is admissible if and only if $a^\prec \subseteq b^\prec$ implies $a\leq b$. In this case, we say that $\alpha$ is an \emph{admissible auxiliary} pair.

Let $\alpha_1=(\prec_1,\leq_1)$ and $\alpha_2=(\prec_2,\leq_2)$ be admissible pairs, where $\prec_i$ are other dense transitive relations on $X$. We write $\alpha_1\leq\alpha_2$ whenever
\beginEnumConditions
\item $\mathrel{{\prec_1}\relsubseteq{\prec_2}}$, 
\item $a^{\leq_1\circ\prec_1}\subseteq b^{\leq_1\circ\prec_1}\implies a^{\leq_2\circ\prec_2}\subseteq b^{\leq_2\circ\prec_2}$,
and 
\item $\mathrel{{\leq_1}\relsubseteq{\leq_2}}$.
\end{enumerate}
This is clearly a partial order amongst admissible pairs. By definition the pair $\alpha_\prec$ is the smallest amongst all admissible pairs of the form $(\prec,\leq)$. Indeed, $(\prec,\leq)$ is an admissible pair and $a,b\in X$, then since as noted above $a^\prec\subseteq b^\prec$ implies that $a^{\leq\circ \prec}\subseteq b^{\leq\circ\prec}$, we see that $\mathrel{{\leq_\prec}\relsubseteq{\leq}}$. From this it follows easily that $\alpha_\prec\leq (\prec,\leq)$.

More is true: It can be shown that $\alpha_\prec$ is the smallest amongst all those admissible pairs $\alpha=(\preccurlyeq,\leq)$ such that $\preccurlyeq$ is $\prec$-prenormal.

We say that an admissible pair $\alpha=(\preccurlyeq,\leq)$ is \emph{prenormal} (or $\prec$-prenormal) provided $\preccurlyeq$ is $\prec$-prenormal and $\leq$ is left $\prec$-continuous. In other words, $\mathrel{{\prec}\relsubseteq {\preccurlyeq}}$, $\mathrel{{\preccurlyeq}={\preccurlyeq}\circ{\prec}}$, and $\mathrel{{\prec}\circ{\leq}\relsubseteq{\prec}\circ{\leq}\circ{\prec}}$. Note that  an admissible pair of the form $(\prec,\leq)$ is prenormal  if and only if $\leq$ is $\prec$-continuous. This is the case for the minimal admissible pair $\alpha_\prec=(\prec,\leq_\prec)$.

An admissible pair $\alpha=(\prec,\leq)$ is \emph{left closed} (or more precisely, \emph{left $\prec$-closed}) if, for any $a,b\in X$,  whenever $(c, b) \in \mathrel{{\prec}  \circ {\leq}}$ for all $c \in a^{\prec_{\phantomtoggle{P}}}$, we have $(a, b) \in {\leq_\prec}  \circ  \leq$.

It follows from the definition, and using the fact that $\mathrel{{\prec}\circ{\leq_\prec}={\prec}}$, that the minimal admissible pair $\alpha_\prec=(\prec,\leq_\prec)$ is always left $\prec$-closed.
\end{pgr}
We single out the following useful observation:
\begin{lma}
\label{lma:auxiliary}
Let $(X,\prec)$ be a set equipped with a dense transitive relation.
\beginEnumConditions
\item If $\alpha_1=(\prec_1,\leq_1)$ and $\alpha_2=(\prec_2,\leq_2)$ are admissible pairs and $\prec_1,\prec_2$ are $\prec$-prenormal, then $\alpha_1\leq\alpha_2$ if and only if $\mathrel{{\prec_1}\relsubseteq{\prec_2}}$ and $\mathrel{{\leq_1}\relsubseteq{\leq_2}}$.
\item Let $\alpha=(\preccurlyeq,\leq)$ be a $\prec$-prenormal (admissible) pair. If $\mathrel{{\leq}\circ{\preccurlyeq}\relsubseteq{\preccurlyeq}}$, then $\alpha$ is auxiliary.
\end{enumerate}
\end{lma}	
\begin{proof}
(i): Following the definition, we need to verify that, if $\mathrel{{\prec_1}\relsubseteq{\prec_2}}$, $\mathrel{{\leq_1}\relsubseteq{\leq_2}}$, and $a^{\leq_1\circ\prec_1}\subseteq b^{\leq_1\circ\prec_1}$, then $a^{\leq_2\circ\prec_2}\subseteq b^{\leq_2\circ\prec_2}$. Let $x\leq_2x'\prec_2a$. Since $\prec_1$ and $\prec_2$ are $\prec$-prenormal, we have $\mathrel{{\prec_2}\relsubseteq{\prec_2}\circ{\prec_1}}$, and thus there is $x''$ such that $x'\prec_2 x''\prec_1a$. Therefore $x''\leq_1 x''\prec_1 a$, which by assumption implies that there is $z$ with $x''\leq_1 z\prec_1 b$. Now, this clearly implies that $x\leq_2z\prec_2 b$.

(ii): We only need to verify that $\mathrel{{\preccurlyeq}\circ{\leq}\relsubseteq{\preccurlyeq}}$. This follows by applying that $\preccurlyeq$ is $\prec$-prenormal at the first step, that $\leq$ is $\prec$-continuous at the second step, that $\preccurlyeq$ is weaker than $\prec$ at the third step, and our assumption at the last step below:
\[
\mathrel{{\preccurlyeq}\circ{\leq}={\preccurlyeq}\circ{\prec}\circ{\leq}\relsubseteq{\preccurlyeq}\circ{\prec}\circ{\leq}\circ{\prec}\relsubseteq{\preccurlyeq}\circ{\leq}\circ{\preccurlyeq}\relsubseteq{\preccurlyeq}}.\qedhere
\]
\end{proof}	
The following connects prenormality for preorders and dense relations.
\begin{prp}
\label{prp:pre_dense_implies_pre_order}
Let $(X,\prec)$ be a set equipped with  a dense transitive relation. Let $(\prec,\leq)$ be an admissible pair. Then, the following conditions are equivalent:
\beginEnumConditions
\item $(\prec,\leq)$ is $\prec$-prenormal and left $\prec$-closed.
\item With  $\mathrel{{\preccurlyeq}}:=\mathrel{{\leq}\circ{\prec}\circ{\leq}}$, the pair $(\preccurlyeq,\leq)$ is $\prec$-prenormal, minimal admissible, and auxiliary.
\item There is a dense relation $\preccurlyeq$ on $X$ such that the pair $(\preccurlyeq,\leq)$ is $\prec$-prenormal, minimal admissible and auxiliary.
\item There is a dense relation $\preccurlyeq$ on $X$ such that the pair $(\preccurlyeq,\leq)$ is $\prec$-prenormal, auxiliary and such that $(X,\preccurlyeq,\leq)$ satisfies \axiomW{2}.
\end{enumerate}
\end{prp}
\begin{proof}
(i)$\implies$(ii): By definition $\mathrel{{\preccurlyeq}}=\mathrel{{\leq}\circ{\prec}\circ{\leq}}$, and we have that $\mathrel{{\prec}\relsubseteq{\preccurlyeq}\relsubseteq{\leq}}$. It follows from the definition of $\preccurlyeq$ that it is a dense transitive relation. To see it is also $\prec$-prenormal we verify condition (iii) in \autoref{lma:pullback_char}. To this end we need to show that, for any $a\in X$, we have that $a^\prec$ is $\preccurlyeq$-cofinal in $a^\preccurlyeq$. If $x\preccurlyeq a$, then $x\mathrel{\leq\circ\prec\circ\leq} a$ and, since $\leq$ is left $\prec$-continuous, we have $x\mathrel{\leq\circ\prec\circ\leq\circ\prec} a$. Thus there is $y\prec a$ such that $x\preccurlyeq y$.

Let us now check that $a\leq b$ if and only if $a^\preccurlyeq\subseteq b^\preccurlyeq$. The forward implication follows from the fact that $\preccurlyeq$ is, by construction, auxiliary for $\leq$. For the converse, since $(\prec,\leq)$ is admissible and left $\prec$-closed, it is enough to check that $(c,b)\in {\prec}\circ{\leq}$ for any $c\in a^\prec$. This will imply that $(a,b)\in\leq_\prec\circ\leq$, and thus $a\leq b$.

Thus, if $c\prec a$, use density to find $c'\in X$ with $c\prec c'\prec a$. Then, since $\preccurlyeq$ is weaker than $\prec$, we have $c'\in a^\preccurlyeq$ and thus $c'\in b^\preccurlyeq$. Using again that $\leq$ is weaker than $\preccurlyeq$ we obtain $c\prec c'\leq b$, as required. Thus, since $\leq$ is induced by $\preccurlyeq$, we have that $(\preccurlyeq,\leq)$ is minimal admissible.

(ii)$\implies$(iii) is trivial.

(iii)$\implies$(iv):

To verify \axiomW{2} on $(X,\preccurlyeq,\leq)$, first notice that $c\leq a$ for any $c\in a^\preccurlyeq$. Next, let $x\in X$ be such that $c\leq x$ for any $c\preccurlyeq a$. Using density for $\preccurlyeq$, we have for any such $c$ an element $c'$ such that $c\preccurlyeq c'\preccurlyeq a$. Then $c'\leq x$ and thus, since $\preccurlyeq$ is auxiliary for $\leq$, we conclude $c\preccurlyeq x$. Thus $a^\preccurlyeq \subseteq x^\preccurlyeq$ which, since the pair $(\preccurlyeq, \leq)$ is minimal admissible by assumption, already implies that $a\leq x$.

(iv)$\implies$(iii): Let us check that $\preccurlyeq$ as in the statement induces $\leq$. If $a^\preccurlyeq\subseteq b^\preccurlyeq$, then for any $x\in X$ such that $x\preccurlyeq a$, we have $x\preccurlyeq b$ and thus $x\leq b$. Using \axiomW{2} we obtain $a\leq b$. The converse is clear since $\preccurlyeq$ is auxiliary for $\leq$.

(iii)$\implies$(i): Since $\leq$ is weaker than $\preccurlyeq$, and the latter is weaker than $\prec$, we have that $\leq$ is also weaker than $\prec$. Let us check that $\leq$ is left $\prec$-continuous, hence assume that $a\prec b\leq c$, for some $a,b,c\in X$. By density, find $c'\in X$ such that $a\prec c'\prec b$, and thus $c'\preccurlyeq b$. Since $\preccurlyeq$ is $\prec$-prenormal, we have from condition (iv) in \autoref{lma:pullback_char} that there is $b'\prec b$ with $c'\preccurlyeq b'$. Therefore $a\prec c'\leq b'\prec b$, which proves that $\leq$ is left $\prec$-continuous.

Finally, we show that $(\prec,\leq)$ is left $\prec$-closed. Suppose that $(c,b)\in\mathrel{{\prec}\circ{\leq}}$ for all $c\in a^\prec$. It suffices to show that $a\leq b$, that is, $a^\preccurlyeq\subseteq b^\preccurlyeq$. Let $c\in X$ be such that $c\preccurlyeq a$. Then, since $\preccurlyeq$ is $\prec$-prenormal, there is $a'\prec a$ such that $c\preccurlyeq a'$, and thus by our assumption $(a',b)\in \mathrel{{\prec}\circ{\leq}}$, that is, $a'\prec b'\leq b$, for some $b'\in X$. Since $c\preccurlyeq a'$ and $\preccurlyeq$ is weaker than $\prec$, we have $c\preccurlyeq b$, as required.
\end{proof}

\begin{prp}
\label{cor:prenormalequivalence} Let $(X,\prec)$ be a set equipped with a dense transitive relation, let $\preccurlyeq$ be a $\prec$-prenormal relation, and let $\leq$ be a preorder. If $(\preccurlyeq,\leq)$ is admissible, then so is $(\prec,\leq)$. Moreover, in this case the following are equivalent:
\beginEnumConditions
\item $(\prec,\leq)$ is $\prec$-prenormal.
\item $(\preccurlyeq,\leq)$ is $\preccurlyeq$-prenormal.
\end{enumerate}
Further, $(\prec,\leq)$ is $\prec$-prenormal and $\prec$-closed if and only if $(\preccurlyeq,\leq)$ is $\preccurlyeq$-prenormal and $\preccurlyeq$-closed.
\end{prp}
\begin{proof}
Assume that $(\preccurlyeq,\leq)$ is admissible. We must prove that $a^{\leq\circ\prec}\subseteq b^{\leq\circ\prec}$ implies $a\leq b$. To see this, let $x,\tilde{x}\in X$ satisfy $x\leq \tilde{x}\preccurlyeq a$. Since $\preccurlyeq$ is $\prec$-prenormal,  we have from condition (iv) in \autoref{lma:pullback_char} that $\preccurlyeq$ is weaker than $\prec$ and that $\mathrel{{\preccurlyeq}={\preccurlyeq}\circ{\prec}}$. Therefore, there is $x'\in X$ such that $\tilde{x}\preccurlyeq x'\prec a$. Therefore $x\leq x'\prec a$ and so there is $y\in X$ such that $x\leq y\prec b$. This implies that $x\leq y\preccurlyeq b$. We have shown that $a^{\leq\circ\preccurlyeq}\subseteq b^{\leq\circ\preccurlyeq}$, hence $a\leq b$ since $(\preccurlyeq,\leq)$ is admissible.

A second usage of the fact that $\preccurlyeq$ is $\prec$-prenormal yields
\[
\mathrel{{\leq}\circ{\prec}\circ{\leq}}=\mathrel{{\leq}\circ{\preccurlyeq}\circ{\leq}}.
\]

(i)$\implies$(ii): By assumption we know that $\leq$ is left $\prec$-continuous. Hence we only need to check that $\leq$ is $\preccurlyeq$-continuous. Since
$\preccurlyeq$ is $\prec$-prenormal, condition (iv) in \autoref{lma:pullback_char} yields again that $\mathrel{{\preccurlyeq}={\preccurlyeq}\circ{\prec}}$. Using this at the first step, that $\leq$ is $\prec$-continuous (so $\mathrel{{\prec}\circ{\leq}\relsubseteq {\prec}\circ{\leq}\circ{\prec}}$) at the second step, and that $\leq$ and $\preccurlyeq$ are weaker than $\prec$ at the third step, we obtain
\[
\mathrel{{\preccurlyeq}\circ{\leq}={\preccurlyeq}\circ{\prec}\circ{\leq}\relsubseteq{\preccurlyeq}\circ{\prec}\circ{\leq}\circ{\prec}={\preccurlyeq}\circ{\leq}\circ{\preccurlyeq}},
\]
and thus $(\preccurlyeq,\leq)$ is $\preccurlyeq$-prenormal.

If, further, $(\prec,\leq)$ is $\prec$-closed, we know from (i)$\implies$(ii) in \autoref{prp:pre_dense_implies_pre_order} that $\leq$ is induced by $\mathrel{{\leq}\circ{\prec}\circ{\leq}}$ which by our observation at the beginning of the proof equals $\mathrel{{\leq}\circ{\preccurlyeq}\circ{\leq}}$. Now, we have proved that $\leq$ is $\preccurlyeq$-continuous, and then it easily follows that
$\mathrel{{\leq}\circ{\preccurlyeq}\circ{\leq}}$ is  $\preccurlyeq$-prenormal (verifying condition (iv) in \autoref{lma:pullback_char}). Since this relation induces (and is clearly auxiliary for) $\leq$, we obtain from (iii)$\implies$(i) of \autoref{prp:pre_dense_implies_pre_order} that $(\preccurlyeq,\leq)$ is left $\preccurlyeq$-closed.

(ii)$\implies$(i): Suppose that $(\preccurlyeq,\leq)$ is $\preccurlyeq$-prenormal. By our assumptions we easily obtain that $\leq$ is weaker than $\prec$, hence in order to prove (i) we only need to show that $\leq$ is left $\prec$-continuous. Using density of $\prec$ at the first step, that $\preccurlyeq$ is weaker than $\prec$ at the second step, that $\leq$ is left $\preccurlyeq$-continuous at the third step, and that $\preccurlyeq$ is $\prec$-prenormal at the fourth step, we obtain
\[
\mathrel{{\prec}\circ{\leq}={\prec}\circ{\prec}\circ{\leq}\relsubseteq{\prec}\circ{\preccurlyeq}\circ{\leq}\relsubseteq{\prec}\circ{\preccurlyeq}\circ{\leq}\circ{\preccurlyeq}={\prec}\circ{\preccurlyeq}\circ{\leq}\circ{\preccurlyeq}\circ{\prec}},
\]
which clearly yields $\mathrel{{\prec}\circ{\leq}\relsubseteq{\prec}\circ{\leq}\circ{\prec}}$, as desired.

The argument for closedness is similar to the one in the implication (i)$\implies$(ii).
\end{proof}


For the rest of the section we focus on $\CatW$-semigroups and we extend the notions and results just developed above to allow for compatibility with addition.
\begin{pgr}[Normal relations]
\label{pgr:normal}
Let $(S,\prec)$ be a $\W$-semigroup, that is, $(S,\prec)$ is a $\W$-set such that $\prec$ is additive and satisfies \axiomW{4}. We say that a relation $\preccurlyeq$ is $\prec$-normal if if is additively closed and $\prec$-prenormal. An admissible pair $(\preccurlyeq,\leq)$ is additively closed if both $\preccurlyeq$ and $\leq$ are additively closed. Finally, an admissible pair $(\preccurlyeq,\leq)$  on $S$ is $\prec$-normal if it is additively closed and $\prec$-prenormal; see \autoref{pgr:admissible}. 

For a relation $\preccurlyeq$, one might ask whether being $\prec$-normal is equivalent to $\preccurlyeq$ being the additive closure of a $\prec$-prenormal relation. We note that this is generally not true, as additive closures of transitive relations may fail to be transitive in general.
\end{pgr}
\begin{lma}
\label{lma:induced_normal}
Let $(S,\prec)$ be a $\W$-semigroup, and let $(\prec,\leq)$ be an admissible pair. Then $(\prec,\leq)$ is $\prec$-normal and left $\prec$-closed if and only $\leq$ is induced by a $\prec$-normal auxiliary relation $\preccurlyeq$.
\end{lma}
\begin{proof}
Suppose first that $(\prec,\leq)$ is $\prec$-normal and closed. Then we know from (i)$\implies$(ii) in \autoref{prp:pre_dense_implies_pre_order} that $\mathrel{{\preccurlyeq}={\leq}\circ{\prec}\circ{\leq}}$ is a $\prec$-prenormal relation that induces $\leq$. Notice that, since both $\prec$ and $\leq$ are additive, the same holds for $\preccurlyeq$, and thus $\preccurlyeq$ is $\prec$-normal.

Conversely, let $\preccurlyeq$ be an abstract $\prec$-normal relation that induces $\leq$ and is also auxiliary. It follows again from \autoref{prp:pre_dense_implies_pre_order} that $(\prec,\leq)$ is $\prec$-prenormal and closed, so it only remains to check that it is also additive. Suppose that $a_i\leq b_i$ for $i=1,2$ in $S$, and let $x\in S$ be such that $x\preccurlyeq a_1+a_2$. By condition (iv) in \autoref{lma:pullback_char} we have that $\mathrel{{\preccurlyeq}={\preccurlyeq}\circ{\prec}}$, hence we can find $y\in S$ such that $x\preccurlyeq y$ and $y\prec a_1+a_2$. Using \axiomW{4}, we have elements $a_1',a_2'\in S$ such that $a_i'\prec a_i$ for each $i$ and $y\prec a_1'+a_2'$. Since $\preccurlyeq$ is weaker than $\prec$ we have $a_i'\preccurlyeq a_i$ for each $i$, and since $a_i\leq b_i$ and $\preccurlyeq$ is auxiliary for $\leq$, we obtain that $a_i'\preccurlyeq b_i$ for each $i$. Using that $\preccurlyeq$ is additive, we finally get $x\preccurlyeq y\prec a_1'+a_2'\preccurlyeq b_1+b_2$, which implies that $x\preccurlyeq b_1+b_2$. Since $\preccurlyeq$ induces $\leq$, this shows that $a_1+a_2\leq b_1+b_2$.
\end{proof}
\begin{pgr}[The prequotient and quotient of a set by a pair]
\label{pgr:quotient} 
We define the \emph{prequotient of $X$} by a pair $\alpha=(\preccurlyeq,\leq)$, and we denote this by $X_\alpha$  to be the  set $X$ equipped with 
the relation $\preccurlyeq_\alpha$ defined by $a\preccurlyeq_\alpha b$, for $a,b\in X$, precisely when $a\mathrel{{\leq}\circ {\preccurlyeq}} b$. 
In this way, we obtain an admissible pair in $X_{\alpha}$ given by $(\preccurlyeq_\alpha,\leq)$. Note that if $\alpha$ is auxiliary, then $(\preccurlyeq_\alpha,\leq)=\alpha$. Observe also that, if $\prec$ is another relation such that $\alpha$ is $\prec$-normal, then so is $(\preccurlyeq_\alpha, \leq)$, as is easily verified from the definitions.

We denote by $\pi_\alpha\colon X\to X_{\alpha}$ the map which on elements is the identity and keeps track of the prequotient relation, and refer to it as \emph{the natural map}.

The \emph{quotient of $X$ by $\alpha$} is denoted by $X/{\alpha}$ and by definition is the antisymmetrization of $X_\alpha$, namely, $X/{\alpha}:=(X_\alpha)/{\leq}$, as in \autoref{pgr:anti}. The quotient map will also be referred to as the natural map and denoted by $\pi_\alpha\colon X\to X/{\alpha}$.

We observe here that our constructions remain unaffected after antisymmetrizing. For example, let $(S,\prec)$ be a $\W$-semigroup, and let $\alpha=(\preccurlyeq,\leq)$ be a $\prec$-normal admissible pair. Then we may equip $S/{\alpha}$ with the relation $\preccurlyeq_\alpha$ defined by: $[a]\preccurlyeq_\alpha [b]$ if $a\mathrel{{\leq}\circ {\preccurlyeq}} b$. Also, the pair $(\preccurlyeq_\alpha,\leq)$ is $\preccurlyeq_\alpha$-normal.

Let us verify that this definition does not depend on the representatives chosen. Use the implication (i)$\implies$(ii) in \autoref{cor:prenormalequivalence} to conclude that  $(\preccurlyeq,\leq)$ is $\preccurlyeq$-normal.  Now suppose that $a'\leq a$, $b\leq b'$ and $a\mathrel{{\leq}\circ{\preccurlyeq}}b$ in $S$. Thus, there is $a_0\in S$ with $a\leq a_0\preccurlyeq b\leq b'$. Since  $\leq$ is  $\preccurlyeq$-continuous, we have that $\mathrel{{\preccurlyeq}\circ{\leq}\relsubseteq{\preccurlyeq}\circ{\leq}\circ{\preccurlyeq}}$. Therefore, there are elements $a_1,a_2\in S$ such that $a_0\preccurlyeq a_1\leq a_2\preccurlyeq b'$. This implies that $a'\leq a_2\preccurlyeq b'$, as required.
\end{pgr}		
We will see now that normal admissible pairs parametrize the possible $\W$-structures on a $\W$-semigroup $S$ and its possible quotients. More precisely:
\begin{prp}
\label{prp:normal}
Let $(S,\prec)$ be a $\W$-semigroup, let $\preccurlyeq$ be another additive relation, and let $\alpha=(\preccurlyeq,\leq)$ be an admissible pair. Then
\beginEnumConditions
\item $\preccurlyeq$ is a $\prec$-normal relation if and only if $(S,\preccurlyeq)$ is a $\W$-semigroup and the identity map $(S,\prec)\to (S,0,\preccurlyeq)$ is a $\W$-morphism.
\item If ${\alpha}$ is $\prec$-normal, then $(S_\alpha,\preccurlyeq_\alpha)$ is a $\W$-semigroup and the natural map $\pi_\alpha\colon S\to S_{\alpha}$ is a $\W$-morphism. Also, $(S/{\alpha},\preccurlyeq_\alpha)$ is a $\W$-semigroup and $\pi_\alpha\colon S\to S/{\alpha}$ is a $\W$-morphism.
\end{enumerate}

\end{prp}
\begin{proof}
(i): Suppose that $\preccurlyeq$ is $\prec$-normal. Since, in particular, $\preccurlyeq$ is weaker than $\prec$, it is clear that $0\preccurlyeq a$ for all $a\in S$. By condition (iii) in \autoref{lma:pullback_char}, we also have that $a^\prec$ is $\preccurlyeq$-cofinal in $a^\preccurlyeq$ for each $a\in S$. Now, given $a\in S$, let $(a_k)$ be a $\prec$-increasing sequence, cofinal in $a^\prec$. If $x\preccurlyeq a$, then by cofinality of $a^\prec$ in $a^\preccurlyeq$ there is $y\in X$ with $y\prec a$ such that $x\preccurlyeq y$. Therefore, we may find $k$ with $y\prec a_k$, and thus $x\preccurlyeq y\prec a_k$. This implies that $(a_k)$ is also cofinal for $a^\preccurlyeq$ with regard to $\preccurlyeq$ and so $(S,\preccurlyeq)$ satisfies \axiomW{1}.

It is clear that \axiomW{3} is satisfied. As for \axiomW{4}, let $c\in X$ be such that $c\preccurlyeq a_1+a_2$. By condition (iv) in \autoref{lma:pullback_char}, $\preccurlyeq$ satisfies $\mathrel{{\preccurlyeq}={\preccurlyeq}\circ{\prec}}$ and thus there is $c'\prec a_1+a_2$ such that $c\preccurlyeq c'$. Now, using \axiomW{4} for $(S,\prec)$, we obtain elements $a_i'\prec a_i$ for $i=1,2$, and so $c'\prec a_1'+a_2'$. It clearly follows that $c\preccurlyeq a_1'+a_2'$ with also $a_i'\preccurlyeq a_i$ for each $i$. Thus $(S,\preccurlyeq)$ is also a $\W$-semigroup. Notice that the identity map $(S,\prec)\to (S,\preccurlyeq)$ is clearly monotone, and it is continuous using again that $\preccurlyeq$ satisfies $\mathrel{{\preccurlyeq}={\preccurlyeq}\circ{\prec}}$.

The converse implication is trivial.

(ii): Assume that $(\preccurlyeq, \leq)$ is $\prec$-normal.  Then, by \autoref{cor:prenormalequivalence}, we obtain that $(\prec,\leq)$ is admissible.

By (i), we know that $(S,\preccurlyeq)$ is a $\W$-semigroup and it follows easily that $S_{\alpha}$ (equipped with $\mathrel{{\preccurlyeq_\alpha}={\leq}\circ{\prec}}$) is also $\W$-semigroup. Indeed, to check \axiomW{1}, it suffices to take, for each $a\in S$ a $\preccurlyeq$-increasing sequence $(a_k)$ in $a^\preccurlyeq$ which is cofinal, and then since $a^\preccurlyeq\subseteq a^{\leq\circ\preccurlyeq}$ it follows that $(a_k)$ is also cofinal for $a^{\preccurlyeq_\alpha}$. That \axiomW{3} holds is immediate since all relations involved are additive. To see that \axiomW{4} also holds, let $c,a_1,a_2\in S$ be such that $c\preccurlyeq_\alpha a_1+a_2$. Then $c\leq c'\preccurlyeq a_1+a_2$ for some $c'\in S$, and thus, using \axiomW{4} in $(S,\preccurlyeq)$, we find elements $a_1',a_2'\in S$ with $a_i'\preccurlyeq a_i$ for $i=1,2$ and such that $c'\preccurlyeq a_1'+a_2'$. It follows that $c\preccurlyeq_\alpha a_1'+a_2'$ with $a_i'\preccurlyeq_\alpha a_i$ for each $i$.

That $\pi_\alpha$ is monotone is immediate since $\prec$ is idempotent and both $\leq$ and $\preccurlyeq$ are weaker relations. Continuity of $\pi_\alpha$ follows since, as observed in (i), we have $\mathrel{{\preccurlyeq}={\preccurlyeq}\circ{\prec}}$.

The final part of the statement follows easily.
\end{proof}	

\section{The fundamental theorems}
\label{sec:normalfundamental}

In this short section we establish the fundamental theorems of quotients of relations, which parallel those for groups, thus justifying our terminology for normal relations.
\begin{pgr}[Kernels]
	\label{pgr:kernel}
	Let $(S,\prec_S)$ and $(T,\prec_T)$ be $\W$-semigroups. Given a $\W$-morphism $f\colon S\to T$, we define its \emph{kernel} as the preorder $\leq_f=\{(a,b)\colon f(a)^{\prec_T}\subseteq f(b)^{\prec_T}\}$.  (This is sometimes denoted by $\ker(f)$, but we choose here a different notation for reasons that will be clear below.) To ease the notation, let us denote by $\leq_S$ and $\leq_T$ the preorders induced by $\prec_S$ and $\prec_T$ respectively (so, for example, $\mathrel{{\leq_S}={\leq_{\prec_S}}}$). It follows from the definitions that $\mathrel{{\leq_f}={f^*(\leq_T)}}$ and $\mathrel{{\leq_S}\relsubseteq{\leq_f}}$. For example, to check the second assertion, suppose that $a\leq_Sb$ and $x\prec_T f(a)$. Then one has by continuity that there is $a'\in S$ with $a'\prec_S a$ and $x\prec_Tf(a')$. Since $a\leq_Sb$, we have $a'\prec_Sb$ and thus by monotonicity we get $x\prec_Tf(a')\prec_Tf(b)$. Hence $a\leq_f b$. 
	We denote by $\ker(f)=(\prec_S,\leq_f)$, and refer to it as the \emph{kernel pair} of $f$.
\end{pgr}
\begin{prp}
	\label{lma:kernelnormal} Let $(S,\prec_S)$ and $(T,\prec_T)$ be $\W$-semigroups, and let $f\colon S\to T$ be a $\W$-morphism. Then:
	\beginEnumConditions
	\item $\ker(f)=(\prec_S,\leq_f)$ is an admissible $\prec_S$-normal closed pair.
	\item $(S_{\ker(f)},\prec_{\ker(f)})$ and $(S/{\ker(f)},\prec_{\ker(f)})$ are $\W$-semigroups.
\end{enumerate}
\end{prp}
\begin{proof}
(i): Let us first show that $\ker(f)$ is admissible. Thus assume that $a^{\leq_f\circ\prec_S}\subseteq b^{\leq_f\circ\prec_S}$, and we are to show that $a\leq_f b$, that is, $f(a)^{\prec_T}\subseteq f(b)^{\prec_T}$. Let $x\in T$ be such that $x\prec_T f(a)$. By continuity of $f$, there is $a'\prec_Sa$ with $x\prec_Tf(a')$. Choose $a''$ such that $a'\prec_Sa''\prec_Sa$. Then we have $f(a')\prec_Tf(a'')$, whence $a'\leq_fa''\prec_Sa$, that is, $a'\in a^{\leq_f\circ \prec_S}$. By assumption, this implies that $a'\in b^{\leq_f\circ\prec_S}$, and thus there is $b'\prec_S b$ such that $a'\leq_fb'$. Therefore, since $x\prec_T f(a')$ we obtain that $x\prec_Tf(b')\prec_Tf(b)$.

To see that $\ker(f)$ is $\prec_S$-normal, we only need to check that $\leq_f$ is left $\prec_S$-continuous. Hence, let $a,b,c\in S$ and assume that $a\prec_S b$ and that $b\leq_fc$, that is, $f(b)^{\prec_T}\subset f(c)^{\prec_T}$. Choose $a'\in S$ such that $a\prec_S a'\prec_S b$. Since $f(a')\prec_T f(c)$, there is by continuity an element $c'\in S$ with $c'\prec_S c$ and $f(a')\prec_T f(c')$. Thus $a\prec_S a'$, $a'\leq_fc'$, and $c'\prec_Tc$, as required.

Finally, to check that $\ker(f)$ is $\prec_S$-closed, assume that $(c,b)\in{\prec_S}\circ{\leq_f}$ for any $c\prec_Sa$. It is enough to show that $a\leq_f b$, that is, $f(a)^{\prec_T}\subset f(b)^{\prec_T}$. Let $d\in T$ be such that  $d\prec_T f(a)$, and choose $c'\prec_Sa$ with $d\prec_T f(c')$, again by continuity. Since by assumption $(c',b)\in{\prec_S}\circ{\leq_f}$, there is $c''\in S$ such that $c'\prec_S c''$ and $c''\leq_fb$, and hence $f(c')\prec_T f(b)$. It follows that $d\prec_T f(c')\prec_T f(b)$, as desired.

(ii) follows from (i) and condition (ii) in \autoref{prp:normal}.
\end{proof}
\begin{rmk}
In the context of \autoref{lma:kernelnormal}, and in order to ease the notation, we will write $S_{\ker(f)}$ and $S/{\ker(f)}$, omitting the corresponding relations.	
\end{rmk}

\begin{thm}
\label{thm:fundamental-theorem0}

Let $(S, \prec_S)$ and $(T, \prec_T)$ be $\W$-semigroups and let $f \colon S \to T$ be a $\W$-morphism. Let $\alpha=(\prec_S,\leq)$ be a $\prec_S$-normal pair and write $\alpha_T=(\prec_T,\leq_{\prec_T})$.
Then $f$ constitutes a $\W$-morphism $S_{\alpha} \to T_{\alpha_T}$ if and only if $\alpha\leq\ker(f)$. Furthermore, $\alpha=\ker(f)$ if and only if $f$ constitutes an order-embedding $S_{\alpha} \to T_{\alpha_T}$.

\end{thm}

\begin{proof}
We have from \autoref{prp:normal} (ii) that $S_{\alpha}$ is a $\W$-semigroup (with $a\prec_\alpha b$ if and only if there is $a'\in S$ with $a\leq a'\prec_S b$) and the map $\pi_\alpha\colon S\to S_{\alpha}$ is a $\W$-morphism. As we did before, let us denote $\leq_T=\leq_{\prec_T}$. We have that $T_{\alpha_T}$ is a $\W$-semigroup (with $\prec_{\alpha_T}=\leq_T\circ\prec_T$).

Now assume that $\alpha\leq\ker(f)$. The first part is clear since, if $a\leq b$ in $S_{\alpha}$, that is, $a\leq b$ in $S$, then using that $\mathrel{{\leq}\relsubseteq{\leq_f}}$ we have $f(a)^{\prec_T}\subseteq f(b)^{\prec_T}$, that is, $f(a)\leq_T f(b)$. 

Let us show that $f$ also constitutes a $\W$-morphism $S_{\alpha} \to T_{\alpha_T}$. If $a\prec_\alpha b$, then there is $a'$ such that $a\leq a'\prec_S b$. By the paragraph above, and since $f$ is a $\W$-morphism $S \to T$, we obtain $f(a)\leq_Tf(a')\prec_T f(b)$. Therefore $f(a) \prec_{\alpha_T} f(b)$ in $T_{\alpha_T}$. Next, if $c\prec_{\alpha_T} f(a)$, then there is $c'$ such that $c\leq_T c'\prec_T f(a)$ in $T$. By continuity of $f$, there is $a'\prec_S a$ such that $c'\prec_T f(a')$, and this implies that $c\prec_{\alpha_T} f(a')$, hence $f$ is continuous $S_{\alpha} \to T_{\alpha_T}$.

Conversely, suppose that $f$ is a $\W$-morphism $S_{\alpha}\to T_{\alpha_T}$. If $a\leq b$ in $S$, then clearly $a\leq b$ in $S_{\alpha}$ and thus $f(a) \leq_T f(b)$ in $T_{\alpha_T}$, that is, $f(a)\leq_T f(b)$ in $T$, which means that $a\leq_f b$. Therefore $\alpha\leq\ker(f)$.

Finally, assume that $\ker(f)=\alpha$ and that $h(a)=f(a)\leq_Tf(b)=h(b)$. This means that $a\leq_f b$, which by assumption translates into $a\leq b$. The converse is similar.
\end{proof}
\begin{rmk}
\label{rmk:fundamental-theorem}
Using arguments similar to the ones in \autoref{pgr:quotient}, the reader may check that \autoref{thm:fundamental-theorem0} remains valid after antisymmetrizing. Since it will be used later repeatedly, we state it here for convenience: Given $\W$-semigroups $(S,\prec_S)$, $(T,\prec_T)$, a $\W$-morphism $f\colon S\to T$, and a $\prec_S$-normal pair $\alpha=(\prec_S,\leq)$, then there is a unique order-preserving $\W$-morphism $h\colon S/{\alpha}\to T/{\alpha_T}$ such that 
\[
\xymatrix{
	S \ar[d]_{\pi_\alpha} \ar[r]^{f} & T\ar[d]^{\pi_{\alpha_T}}  \\
	S/{\alpha} \ar@{-->}[r]_{h} & T/{\alpha_T},
}
\]
is commutative if and only if $\alpha\leq\ker(f)$, and $h$ is an order-embedding precisely when $\alpha=\ker(f)$.	
Note that $h$ is uniquely determined by $f$. 
\end{rmk}

\begin{lma}
\label{lma:quopair}
Let $(S,\prec)$ be a $\W$-semigroup. Let $\alpha=(\preccurlyeq,\leq)$ be a $\prec$-normal admissible pair. Denote by $\preccurlyeq_\alpha$ the relation induced by $\preccurlyeq$ in $S_{\alpha}$, that is, $\mathrel{{\preccurlyeq_\alpha}={\leq}\circ{\preccurlyeq}}$ (see \autoref{pgr:quotient}). For any other $\prec$-normal pair $\alpha'=(\preccurlyeq',\leq')$ in $S$ such that $\alpha\leq\alpha'$, let $\alpha'_{\alpha}=(\preccurlyeq_{\alpha'_{\alpha}},\leq_{\alpha'_{\alpha}})$ be the pair in $S_{\alpha}$ defined by
\beginEnumConditions
\item $a\leq_{\alpha'_{\alpha}}b$ if $a\leq' b$;
\item $a\preccurlyeq_{\alpha'_{\alpha}}b$ if $a\mathrel{{\leq'}\circ{\preccurlyeq'}}b$.
\end{enumerate}
Then ${\alpha'}_{\alpha}$ is a $\preccurlyeq_\alpha$-normal (admissible) auxiliary pair in $S_{\alpha}$. In the case that $\alpha'=\alpha$, we have $\alpha_{\alpha}=(\preccurlyeq_\alpha,\leq)$. In particular, $\alpha_\alpha\leq\alpha'_{\alpha}$.
\end{lma}
\begin{proof}
Using the definitions, it is easy to verify that $\alpha'_{\alpha}$ is an admissible pair in $S_{\alpha}$. Indeed, this follows from the equality  $\mathrel{{\leq_{\alpha'_{\alpha}}}\circ{\preccurlyeq_{\alpha'_{\alpha}}}={\leq'}\circ{\leq'}\circ{\preccurlyeq'}={\leq'}\circ{\preccurlyeq'}}$.

We now show that $\leq'$ is also $\preccurlyeq$-continuous. To do so, we use below that $\preccurlyeq$ is $\prec$-normal at the first step, that $\preccurlyeq'$ is weaker than $\prec$ at the second step, that $\leq'$ is  $\preccurlyeq'$-continuous at the third step, that $\preccurlyeq'$ is $\prec$-normal and $\leq'$ weaker than $\preccurlyeq'$ at the fourth step, and that $\preccurlyeq$ is weaker than $\prec$ at the final step, hence we get:
\[
\mathrel{{\preccurlyeq}\circ{\leq'}={\preccurlyeq}\circ{\prec}\circ{\leq'}\relsubseteq{\preccurlyeq}\circ{\preccurlyeq'}\circ{\leq'}\relsubseteq{\preccurlyeq}\circ{\preccurlyeq'}\circ{\leq'}\circ{\preccurlyeq'}\relsubseteq{\preccurlyeq}\circ{\leq'}\circ{\prec}\relsubseteq{\preccurlyeq}\circ{\leq'}\circ{\preccurlyeq}}.
\]

Using this fact at the first step, that $\preccurlyeq$ is dense at the second step, and that $\preccurlyeq$ is weaker than $\leq$ at the final step, we get
\[
\mathrel{{\leq}\circ{\preccurlyeq}\circ{\leq'}\relsubseteq{\leq}\circ{\preccurlyeq}\circ{\leq'}\circ{\preccurlyeq}\relsubseteq{\leq}\circ{\preccurlyeq}\circ{\leq'}\circ{\preccurlyeq}\circ{\preccurlyeq}\relsubseteq{\leq}\circ{\preccurlyeq}\circ{\leq'}\circ{\leq}\circ{\preccurlyeq}},
\]
and thus $\leq_{\alpha'_{\alpha}}$ is $\preccurlyeq_\alpha$-continuous.

Finally, to check that $\preccurlyeq_{\alpha'_{\alpha}}$ is $\preccurlyeq_\alpha$-normal, this amounts to checking that \[\mathrel{{\leq'}\circ{\preccurlyeq'}={\leq'}\circ{\preccurlyeq'}\circ{\leq}\circ{\preccurlyeq}}.\]
Since $\preccurlyeq'$ is  $\prec$-normal and $\preccurlyeq,\leq$ are weaker than $\prec$ we obtain 
\[
\mathrel{{\leq'}\circ{\preccurlyeq'}={\leq'}\circ{\preccurlyeq'}\circ{\prec}\circ{\prec}\relsubseteq{\leq'}\circ{\preccurlyeq'}\circ{\leq}\circ{\preccurlyeq}},
\]
and the other inclusion is similar. That $\alpha'_{\alpha}$ is auxiliary follows from \autoref{lma:auxiliary}, and the last part of the statement is trivial.
\end{proof}

\begin{thm}\label{thm:fundamental-theorem}
Let $(S,\prec)$ be a $\W$-semigroup, and let $\alpha=(\preccurlyeq,\leq)$ be a $\prec$-normal admissible pair. Retain the notation in \autoref{lma:quopair}. Then, the natural $\W$-morphism $\pi_\alpha\colon S\to S_{\alpha}$ induces a one-to-one correspondence
\[
\tilde{\pi}\colon\{\alpha'\colon\alpha\leq\alpha'\text{ is a }\prec\text{-normal pair}\}\to\{\beta\colon \alpha_{\alpha}\leq\beta\text{ is a }\preccurlyeq_\alpha\text{-normal pair in }S_{\alpha}\}
\]
given by $\tilde{\pi}(\alpha)=\alpha'_{\alpha}$, which is bijective on admissible auxiliary pairs. Furthermore, we have a commutative diagram of $\W$-morphisms, of which the bottom row is a $\W$-isomorphism:
\[
\xymatrix{
S \ar[r]^{\pi_\alpha} \ar[d]_{\pi_{\alpha'}} & S_{\alpha} \ar[d]^{\pi_{\alpha'_{\alpha}}}  \\
S_{ \alpha'} \ar[r]^{\!\!\!\!\!\!\!\!\!\!\!\varphi} 	& ( S_{\alpha})_{(\alpha'_{\alpha})}.
} 
\]
\end{thm}
\begin{proof}
We already know from \autoref{lma:quopair} that $\tilde{\pi}(\alpha)$ is a $\preccurlyeq_\alpha$-normal pair in $S_{\alpha}$ and that the correspondence $\tilde{\pi}$ is one-to-one. 

Next, let $\beta$ be a $\preccurlyeq_\alpha$-normal auxiliary pair in $S_{\alpha}$ that contains $(\preccurlyeq_\alpha,\leq_\alpha)$, say $\beta=(\tilde{\preccurlyeq},\tilde{\leq})$.
Define a preorder $\leq'$ on $S$ by $a\leq' b$ if there is $b_0$ such that $a\leq b_0$ and $b_0 {\tilde{\leq}} b$. (It is easy to verify that $\leq'$ is a preorder since $\mathrel{{\leq_\alpha}\relsubseteq{\tilde{\leq}}}$.)

Define a relation $\preccurlyeq'$ on $S$ by $a\preccurlyeq'b$ if $a\leq' b_0$ and $b_0\tilde{\preccurlyeq}b$, which is easily seen to be dense and transitive. We claim that the pair $\alpha'=(\preccurlyeq',\leq')$ is $\prec$-normal (admissible) auxiliary and $\tilde{\pi}(\alpha')=\beta$.

To see that $\tilde{\pi}(\alpha')=\beta$ we use \autoref{lma:quopair} together with the fact that $\beta$ is auxiliary. For example, for $a,b\in S$ we have that $a\preccurlyeq_{\alpha'_{\alpha}}b$ if and only if there is $\tilde{b}\in S$ such that $a\leq'\tilde{b}\preccurlyeq' b$. Then, by definition of $\leq'$ and $\preccurlyeq'$,  we can find elements $b_0,c_0\in S$ such that $a\leq b_0$, $\tilde{b}\leq' c_0$, and $b_0\tilde{\leq}\tilde{b} $, $c_0\tilde{\preccurlyeq}b$. Therefore $a\tilde{\leq}b_0\tilde{\leq}c_0$ and $c_0\tilde{\preccurlyeq}b$, which implies that $a\tilde{\preccurlyeq}b$ since $\beta$ is auxiliary. With similar (and easier) arguments we see that $a\tilde{\preccurlyeq}b$ implies $a\preccurlyeq_{\alpha'_{\alpha}}b$ and also that $a\leq_{\alpha'_{\alpha}}b$ precisely when $a\tilde{\leq}b$.

Note that, by definition, $\mathrel{{\leq'}\circ{\preccurlyeq'}={\preccurlyeq'}}$. Also, if $a^{\preccurlyeq'}\subseteq b^{\preccurlyeq'}$, we have that $a^{\preccurlyeq_{\alpha'_{\alpha}}}\subseteq b^{\preccurlyeq_{\alpha'_{\alpha}}}$ and, using that $\tilde{\pi}(\alpha')=\beta$ and that $\beta$ is admissible, we get $a\tilde{\leq}b$, that is, $a\leq' b$. Therefore $\alpha'$ is also admissible.

That $\alpha'$ is $\prec$-normal auxiliary is a tedious routine check, using that $\beta$ is $\preccurlyeq_\alpha$-normal auxiliary and that $\alpha$ is $\prec$-normal, hence we omit the details. 

The last part of the statement follows directly using \autoref{lma:quopair}.
\end{proof}

\section{Ideals}
\label{sec:ideals}

In this section we introduce the notion of ideal of a $\W$-semigroup and relate it to the corresponding concept for the category of abstract $\Cu$-semigroups. We also show how ideals constitute a particular situation of normal admissible pairs.

\begin{pgr}[Ideals]
	\label{pgr:normalideals}
	Let $(S,\prec)$ be a $\W$-semigroup. A subset $I$ of $S$ is an \emph{ideal} (or a $\W$-ideal) provided that $I$ is a subsemigroup of $S$ such that $0\in I$ and $a\prec b$ with $b\in I$ implies that $a\in I$. We say that an ideal $I$ is \emph{closed} provided that $a\in I$ whenever $a^\prec\subseteq I$.
	
	We may also characterize closed ideals in terms of the induced order $\leq_\prec$, as follows. A subsemigroup $I$ containing $0$ of $S$ is a closed ideal provided that $a\leq_\prec b$ with $b\in I$ implies that $a\in I$, and also that $a\in I$ whenever $a^\prec\subseteq I$. In particular, if $a+b\in I$, then since $a\leq_\prec a+b$, we obtain that $a\in I$.
	
	Given an ideal $I$ of $S$, we define its \emph{closure} is $\overline{I}=\{a\in S\colon a^\prec \subset I\}$. By using axioms \axiomW{1} and \axiomW{3}, one may check that $\overline{I}$ is a closed ideal containing $I$, and that $I=\overline{I}$ precisely when $I$ is closed.
	
	We shall denote the lattice of ideals by $\mathrm{IdLat}_\W(S)$ and the lattice of closed ideals by $\Lat_\W(S)$. 
	
	We have that $\Lat_\W(S)\subseteq\mathrm{IdLat}_\W(S)$, and the closure operation just described yields a retract $\mathrm{cl}\colon\mathrm{IdLat}_\W(S)\to\overline{\Lat}_\W(S)$, by taking $\mathrm{cl}(I)=\overline{I}$.
\end{pgr}

\begin{lma}
	\label{lma:Galois}	
	Let $(S,\prec)$ be  a $\W$-semigroup, and let $\alpha_S=(\prec,\leq_\prec)$. Then, $I$ is a closed ideal of $S$ if, and only if, $I$ is a closed ideal of $S_{\alpha_S}$. Therefore, there is a natural identification
	\[
	\Lat_\W(S)=\Lat_\W(S_{\alpha_S})=\Lat_\W(S/{\alpha_S}).
	\]
\end{lma}	
\begin{proof}
	Suppose that $I$ is a closed ideal of $S$. Let $a\in I$ and $x,z\in S$ be such that $x\leq_\prec z\prec a$. Then $z\in I$ since $I$ is an ideal of $S$, whence $z^\prec\subset I$. This implies that $x^\prec\subset I$ and, since $I$ is closed, we obtain $x\in I$. This shows that $I$ is also an ideal of $S_{\alpha_S}$. That $I$ is also closed follows from the observation that, for any $x\in S$, we have $x^\prec\subset x^{\mathrel{{\leq_\prec}\circ{\prec}}}$.
	
	Conversely, suppose that $I$ is a closed ideal of $S_{\alpha_S}$. Let $a\in I$ and $x\in S$ be such that $x\prec a$. Then $x\leq_\prec x\leq a$, and thus $x\in I$, hence $I$ is an ideal of $S$. To verify it is also a closed ideal, suppose that $a^\prec\subset I$ for some $a\in S$. Then $a\in I$ once we show that also $a^{\mathrel{{\leq_\prec}\circ{\prec}}}\subset I$ and use that $I$ is closed in $(S,\mathrel{{\leq_\prec}\circ{\prec}})$. To see this, let $x,y\in S$ be such that $x\leq_\prec y\prec a$, and find $z\in S$ such that $y\prec z\prec a$. Then $z\in I$ since $a^\prec\subset I$ and, as $x\leq_\prec y\prec z$, we have $x\in I$.
\end{proof}

We now recall the axioms used to define the category $\Cu$ of abstract Cuntz semigroups as introduced in \cite{CowEllIva08CuInv}. These model the completion $\Cu(A)$ of the classical Cuntz semigroup $\W(A)$ for a (local) C*-algebra $A$; see \cite[Theorem 3.2.8]{AntPerThi14arX:TensorProdCu}.

\begin{pgr}[The category $\Cu$]
	\label{dfn:Cuntz-semigroup}
	An \emph{abstract Cuntz semigroup}, or also a $\CatCu$-semigroup, is a positively ordered monoid $S$ satisfying the following axioms, where $\ll$ is the compact containment relation, as defined in \autoref{dfn:preorder-way-below}:
	\begin{itemize}
		\item[\axiomO{1}] Every increasing sequence $(a_k)_k$ in $S$ has a supremum $\sup_k a_k$ in $S$.
		\item[\axiomO{2}] Every element $a \in S$ is the supremum of a sequence $(a_k)_k$ such that $a_k \ll a_{k+1}$ for all $k$.
		\item[\axiomO{3}] If $a', a, b', b \in S$ satisfy $a' \ll a$ and $b' \ll b$, then $a' + b' \ll a + b$.
		\item[\axiomO{4}] If $(a_k)_k$ and $(b_k)_k$ are increasing sequences in $S$, then $\sup_k (a_k + b_k) = \sup_k a_k + \sup_k b_k)$.
	\end{itemize}
	Given $\Cu$-semigroups $S$ and $T$, a $\CatCu$-\emph{morphism} $f \colon S \to T$ is a map that preserves addition, order, compact containment and suprema of increasing sequences. We let $\operatorname{Cu}$ be the category whose objects are $\CatCu$-semigroups and whose morphisms are $\CatCu$-morphisms. Note that, for a $\Cu$-semigroup $S$ and $a,b\in S$, we have $a\leq b$ if and only if $a^\ll\subset b^\ll$.
	
	A $\CatCu$-semigroup is a $\CatW$-semigroup, and a $\CatCu$-morphism is a $\CatW$-morphism (\cite[Lemma~3.1.4]{AntPerThi14arX:TensorProdCu}). In fact, the category $\operatorname{Cu}$ is a full reflective subcategory of $\operatorname{W}$ (\cite[3.1.5]{AntPerThi14arX:TensorProdCu}). 
	
	An ideal of a $\Cu$-semigroup $S$ is by definition an order-hereditary subsemigroup $I$ of $S$ which is also closed under suprema of increasing sequences (see \cite[5.1.1]{AntPerThi14arX:TensorProdCu}). We shall denote by $\Lat_{\Cu}(S)$ the lattice of ideals of $S$. The relation between these categories, also in the light of group actions, is revisited in \autoref{sec:categorical}, where we will relate the lattice of ideals of a $\W$-semigroup $S$ with that of its completion.
	
	We say that a $\W$-semigroup $S$ satisfies \axiomO{1} if $S$ is closed under suprema of increasing sequences with respect to the preorder $\leq_\prec$ induced by $\prec$.
	
	Notice that, if $S$ satisfies \axiomO{1} and the relation $\prec$ is stronger than the way-below relation, then any closed ideal is also closed under suprema of increasing sequences. Indeed, if $I$ is such a closed ideal and $(a_n)$ is an increasing sequence in $I$, let $a=\sup a_n$. For $c\prec a$, find using \axiomW{1} an element $c'\in S$ such that $c\prec c'\prec a$. By our assumption on $\prec$ there is $n$ such that $c'\leq_\prec a_n$, and thus $c\prec a_n$. This implies that $c\in I$ and, since $c$ is arbitrary in $a^\prec$ and $I$ is closed, we obtain $a\in I$.
	
	The converse of the above statement, in other words, that an ideal closed under suprema of increasing sequences must be closed, holds if $S$ satisfies \axiomW{2}. Thus, in particular these concepts are equivalent for $\Cu$-semigroups.
\end{pgr}


\begin{pgr}[Ideals and normal pairs]
	\label{pgr:idealspairs}
	Let $(S,\prec)$ be a $\W$-semigroup, and let $I$ be an ideal of $S$. Given $a,b\in I$, we write $a\leq_I b$ provided that, for any $x\prec a$, there is $y\in I$ such that $x\prec b+y$. We write $\alpha_I=(\prec,\leq_I)$, $S_I=S_{\alpha_I}$, and $S/I=S/{\alpha_I}$.
	
	In the converse direction, given an admissible pair $\alpha=(\prec,\leq)$ on a $\W$-semigroup $(S,\prec)$, write $I_\alpha=\{a\in S\colon a\leq 0\}$. 
\end{pgr}
We now explore the relationship between these two notions.

\begin{lma}
	\label{prp:idealsnormal} 
	Let $(S,\prec)$ be a $\W$-semigroup.
	Then:
	\beginEnumConditions
	\item If $\alpha=(\prec,\leq)$ is an admissible pair  on $S$, then $I_\alpha$ is an ideal of $S$, which is closed if $\alpha$ is left $\prec$-closed.
	\item If $I$ is an ideal, then $\leq_I$ is a preorder and $\alpha_I$ is an admissible $\prec$-normal pair. 
	\item If $I$ is an ideal and $\alpha=(\prec,\leq)$ an admissible pair such that $\mathrel{{\prec}\circ{\leq}\relsubseteq{\prec}}$, then $I\subseteq I_\alpha$ if and only if $\alpha_I\leq\alpha$. In particular, $I\subseteq I_{\alpha_I}$ and $\alpha_{I_\alpha}\leq \alpha$. Finally, If $I$ is a closed ideal, then $I=I_{\alpha_I}$.	
\end{enumerate}
\end{lma}
\begin{proof}
(i): It is clear that $I_\alpha$ is a subsemigroup since $0$ is an idempotent element. If now $a\prec b $ and $b\in I_\alpha$, we have $b\leq 0$ and as $\alpha$ is admissible we have $a\leq b\leq 0$, that is, $a\leq 0$. 

Now assume that $\alpha$ is left $\prec$-closed and that $a^\prec\subseteq I_\alpha$. If $c\prec a$, choose $c'$ such that $c\prec c'\prec a$. Since by assumption $c'\leq 0$, we have $(c,0)\in\mathrel{{\prec}\circ{\leq}}$, and since $\alpha$ is closed and $c\prec a$ is arbitrary, this implies that $(a,0)\in\mathrel{{\leq_\prec}\circ{\leq}}
$. Using again that $\alpha$ is admissible, we have $\mathrel{{\leq_\prec}\relsubseteq{\leq}}$, and thus $a\leq 0$, that is, $a\in I_\alpha$.

(ii):  The fact that $\leq_I$ is a preorder and additively closed follows from the fact that $I$ is a subsemigroup. To show that $\alpha_I$ is an admissible pair, suppose that $a^{\leq_I\circ\prec}\subseteq b^{\leq_I\circ\prec}$. We must show that $a\leq_I b$. To this end, let $x\prec a$ and choose $z$ such that $x\prec z\prec a$. Since $z\leq_Iz\prec a$, there is by assumption an element $w$ with $z\leq_I w\prec b$. Now, using that $x\prec z\leq_I w$ we find $y'\prec y\in I$ with $x\prec w+y'\prec b+y$.

Finally, to see that $\alpha_I$ is $\prec$-normal, we need to check that $\leq_I$ is $\prec$-continuous. Suppose that $a\prec b\leq_Ic$. Choose $b'$ such that $a\prec b'\prec b$. Then there is $y\in I$ such that $b'\prec c+y$. Apply \axiomW{4} to obtain $c'\prec c$ and $y'\prec y$ (hence $y'\in I$) such that $b'\prec c'+y'$. Now, we have that $a\prec b'\leq_I c'\prec c$.

(iii): Suppose first that $I\subseteq I_\alpha$. To prove that $\alpha_I\leq\alpha$ it suffices to show that $\mathrel{{\leq_I}\relsubseteq{\leq}}$; see \autoref{lma:auxiliary} (i). Assume that $a\leq_I b$. Since $\alpha$ is admissible, it is enough to show that $a^\prec\subseteq b^\prec$. Let $x\prec a$. Then there is $y\in I$ such that $x\prec b+y$. Using \axiomW{4}, we find $b'\prec b$ and $y'\prec y$ such that $x\prec b'+y'$. Since $y\in I\subseteq I_\alpha$ and $y'\prec y$, we conclude from the assumptions that $y'\prec 0$. Thus $x\prec b$.

Conversely, suppose that $\alpha_I\leq \alpha$, and let $a\in I$. If $x\prec a$, clearly $x\prec 0+a$, that is, $a\leq_I 0$. Thus $a\leq 0$, hence $a\in I_\alpha$.

In particular, since $\alpha_I=\alpha_I$ we get $I\subseteq I_{\alpha_I}$; and since $I_\alpha=I_\alpha$, we get $\alpha_{I_\alpha}\leq \alpha$.

Assume finally that $I$ is closed. We already have that $I\subseteq I_{\alpha_I}$. Now let $a\in I_{\alpha_I}$ and let $x\prec a$. Since $a\leq_I 0$, there is $b\in I$ such that $x\prec 0+b=b$, and since $I$ is an ideal we have $x\in I$. Thus $a^\prec\subseteq I$ and as $I$ is closed, this implies $a\in I$.
\end{proof}
\begin{pgr}[Galois connections]
Recall that, given posets $S$ and $T$, a \emph{Galois connection} between $S$ and $T$ consists of a pair of order-preserving functions $f\colon S\to T$ and $g\colon T\to S$ such that $g(f(a))\leq a$ and 	$f(g(b))\leq b$ for any $a\in S$ and $b\in T$. 

Let $(S,\prec)$ be a $\W$-semigroup. Let us denote by $\mathrm{Ad}(S)$ the poset of admissible pairs, by $\mathrm{Ad}_{n}(S)$ the poset of $\prec$-normal admissible pairs, and by $\mathrm{Ad}_c(S)$ the poset of $\prec$-closed admissible pairs. We write $\mathrm{Ad}_{nc}(S)=\mathrm{Ad}_n(S)\cap\mathrm{Ad}_c(S)$.
\end{pgr}	

\begin{prp}\label{cor:closed-ideal-vs-normal-preorder-subcategory}
Let $(S,\prec)$ be a $\W$-semigroup. Consider the assignments 
\[\varphi\colon \mathrm{IdLat}_\W(S)\to \mathrm{Ad}(S) \text{ and }\psi\colon \mathrm{Ad}(S)\to \mathrm{IdLat}_\W(S)
\] 
given by $I \stackrel{\varphi}{\longmapsto} {\alpha_I}$ and ${\alpha} \stackrel{\psi}{\longmapsto} I_\alpha$. If $S$ satisfies \axiomO{1}, $\mathrel{{\prec}\circ{\leq}\relsubseteq{\prec}}$, and $\mathrel{{\prec}\relsubseteq{\ll}}$, then $\varphi$ and $\psi$ yield a Galois connection between $\Lat_\W(S)$ and $\mathrm{Ad}_{nc}(S)$.
\end{prp}
\begin{proof}
In general, the assignments in the statement are given by \autoref{prp:idealsnormal}.
By (ii) and (iii) in said lemma, if $I$ is a closed ideal, then $\alpha_I$ is an admissible $\prec$-normal pair and $\psi(\varphi(I))=I_{\alpha_I}= I$. Again by \autoref{prp:idealsnormal} (i), (iii), if $\alpha$ is a $\prec$-normal closed pair, then $I_\alpha$ is a closed ideal and $\varphi(\psi(\alpha))=\alpha_{I_\alpha}\leq \alpha$.	

We need to prove that if $I$ satisfies \axiomO{1}, then $\alpha_I$ is left $\prec$-closed. This will be  the case, in particular, if $I$ is closed, $S$ satisfies \axiomO{1}, and $\mathrel{{\prec}\relsubseteq{\ll}}$ as in the assumption; see the comments in \autoref{pgr:normalideals}.

Suppose now that $I$ is closed under suprema of increasing sequences with respect to the preorder $\leq_\prec$. Fix any $a,b \in S$ such that $a'  \mathrel{{\prec}  \circ  {\leq_I} }b$ for any $a' \in a^{\prec}$. Let be $(a_k)_k$ be an upward directed and cofinal sequence in $a^{\prec}$ as given by \axiomW{1}. By our assumption, there is a sequence $(c_k)_k$ in $I$ such that $a_k \prec b + c_k$ for every $k$. Since the sequence $(c_1 + \dots + c_k)_k$ is increasing in $I$ with respect to $\leq_\prec$, we may find a supremum $c$ in $I$. It follows that, if $a'\prec a$, then $a'\prec b+c$, and thus $a\leq_I b$, as desired.
\end{proof}

\begin{cor}
\label{cor:refl} Let $S$ be a $\Cu$-semigroup. Then the poset category of closed ideals of $S$ embeds as a full coreflective subcategory of the poset category of normal closed pairs on $S$, via the assignment $I \mapsto \alpha_I=(\ll,\leq_I)$.   
\end{cor}
\begin{proof}
It follows from \autoref{prp:idealsnormal} that the assignment $\varphi\colon I\to \alpha_I$ is an injective order-embedding (by using $\psi\colon \alpha\to I_\alpha$)) so one can view the poset category of closed ideals as a full subcategory of the poset category of normal closed pairs. To see that it is a correflective subcategory, we need to check that, if $I$ is a closed ideal and $\alpha$ is a closed normal pair, we have $\varphi(I)\leq \alpha$ if and only if $I\leq \psi(\alpha)$, which is immediate from the arguments in \autoref{cor:closed-ideal-vs-normal-preorder-subcategory}.
\end{proof}

\begin{cor}\label{cor:closed-ideal-quotients-generalized}
If $S$ is a $\Cu$-semigroup, and $I$ is a closed ideal of $S$, then the quotient $S/I$ as defined in \cite[Section~5.1]{AntPerThi14arX:TensorProdCu} agrees with the quotient $S/{\alpha_I}$.
\end{cor}

We conclude this section with a discussion on how taking quotients by pairs or by ideals are related in the context of $\Cu$-morphisms, in particular those induced by *-homomorphisms between C*-algebras.

\begin{pgr}[Kernels of $\Cu$-morphisms]
\label{lma:closed-ideal-Cu-morphism}	
Let $S,T$ be $\Cu$-semigroups and let $f\colon S\to T$ be a $\Cu$-morphism. We defined in \autoref{pgr:kernel} the kernel of $f$ as $\leq_f=\{(a,b)\colon f(a)^\ll\subseteq f(b)^\ll\}$ which, by virtue of \axiomO{2}, agrees with $\{(a,b)\colon f(a)\leq f(b)\}=f^*(\leq)$. Recall that the kernel pair of $f$ is, by definition, $\ker(f)=(\ll, \leq_f)$, which is a closed pair by \autoref{lma:kernelnormal}.

Let $I_{\ker(f)}$ be the closed ideal corresponding to $\ker(f)$, and notice that $I_{\ker(f)}=f^{-1}(\{0\})$. Indeed, we have  $I_{\ker(f)} = \left\{ a \in S \colon a\leq_f 0 \right\} = \left\{ a \in S \colon f(a) \leq 0 \right\} = f^{-1}(\{0\})$. 
\end{pgr}

As far as *-homomorphisms between C*-algebras are concerned, the case of surjective $*$-homomorphisms is well understood. In this case it suffices to consider ideals of $\CatCu$-semigroups, as opposed to more general normal pairs. We shall revisit the following, in the presence of group actions; see \autoref{thm:dynquotients}.

\begin{thm}[{\cite[Theorem~4.1]{CiuRobSan10CuIdealsQuot}}] \label{thm:Cu-short-exact-sequence}
Let $I$ be a $\sigma$-unital ideal of a C*-algebra $A$. Then the short exact sequence 
\[
0 \to I \xrightarrow{\iota} A \xrightarrow{\varphi} A/I \to 0
\]
induces a short exact sequence of $\Cu$-semigroups and $\Cu$-morphisms
\[
0 \to \operatorname{Cu}(I) \xrightarrow{\operatorname{Cu}(\iota)} \operatorname{Cu}(A) \xrightarrow{\operatorname{Cu}(\varphi)} \operatorname{Cu}(A/I) \to 0 \; ,
\]
i.e., the $\Cu$-morphism $\operatorname{Cu}(\iota)$ is injective, the $\Cu$-morphism $\operatorname{Cu}(\varphi)$ is surjective, and $\operatorname{Cu}(\iota) \big(\operatorname{Cu}(I) \big) = \operatorname{Cu}(\varphi)^{-1} (0)$. \qed

Moreover, one has that $\CatCu(A/I)\cong \CatCu(A)/\CatCu(I)$.
\end{thm}

On the other hand, for injective *-homomorphisms, the case when the image is a hereditary C*-subalgebra (e.g., an ideal), is also well understood. It does not involve the use of quotients. 

\begin{prp}[c.f., e.g., \cite{AntPerThi20:Absbivariant}] \label{prp:Cu-hereditary-subalg}
Let $\varphi \colon A \to B$ be a *-homomorphism between C*-algebras. If $\varphi$ is injective and $\varphi(A)$ is a hereditary C*-subalgebra of $B$, then the induced $\Cu$-morphism $\operatorname{Cu}(\varphi) \colon \operatorname{Cu}(A) \to \operatorname{Cu}(B)$ is injective.  \qed
\end{prp}

It is general *-homomorphisms (whose images may not be hereditary C*-subalgebras) that call for the use of normal pairs. 

\begin{thm} \label{thm:Cu-quotients-from-star-homomorphism}
Let $\varphi \colon A \to B$ be a *-homomorphism between C*-algebras. Assume the kernel $\ker(\varphi)$ is a $\sigma$-unital ideal of $A$. We decompose $\varphi$ as 
\[
\xymatrix{
	A \ar[rr]^{\varphi} \ar@{->>}[rd]_{\pi_{\ker(\varphi)}} & & B \\
	& A / \ker(\varphi) \ar@{^{(}->}[ru]_{\overline{\varphi}} &
}
\]
where the arrows $\twoheadrightarrow$ and $\hookrightarrow$ indicate surjectivity and injectivity, respectively. Then, after applying the functor $\operatorname{Cu}$, we obtain following commutative diagram of $\Cu$-morphisms: 
\[
\xymatrix{
	\operatorname{Cu}(A) \ar[rrr]^{\operatorname{Cu}(\varphi)} \ar@{->>}[rd]^{\operatorname{Cu} \left( \pi_{\ker(\varphi)} \right)} \ar@{->>}@/_2pc/[rdd]_{\pi_{\operatorname{Cu} (\ker (\varphi))}} & & & \operatorname{Cu}(B) \\
	& \operatorname{Cu} (A / \ker(\varphi)) \ar[rru]^{\operatorname{Cu} \left( \overline{\varphi} \right)} &   \left\{ \left[ \varphi(a) \right] \in \operatorname{Cu}(B) \colon a \in A \otimes K \right\}  \ar@{^{(}->}[ru] & \\
	& \operatorname{Cu} (A) / \operatorname{Cu}(\ker(\varphi)) \ar[u]_{\cong} \ar[r] &  \operatorname{Cu} (A) / \ker(\operatorname{Cu}(\varphi))  \ar[u]_{\cong} \ar@{^{(}->}@/_2pc/[ruu] & \\
	& \operatorname{Cu} (A) / I_{\ker(\operatorname{Cu}(\varphi))} \ar@{=}[u] & &
}
\]
where again the arrows $\twoheadrightarrow$ and $\hookrightarrow$ indicate surjectivity and injectivity, respectively. 
\end{thm}

\begin{proof}
On the left half of the diagram, the isomorphism 
\[
\operatorname{Cu} (A) / \operatorname{Cu}(\ker(\varphi))  \xrightarrow{\cong} \operatorname{Cu} (A / \ker(\varphi)) \; ,
\]
as well as the commutative triangle involving it and $\operatorname{Cu}(A)$, comes from applying Theorem~\ref{thm:Cu-short-exact-sequence} to the short exact sequence $0 \to \ker(\varphi) \to A \to A / \ker(\varphi) \to 0$ of C*-algebras. This theorem also gives $\operatorname{Cu}(\ker(\varphi)) = \operatorname{Cu}(\varphi)^{-1} (0)$, which is equal to $I_{\ker(\operatorname{Cu}(\varphi))}$ by the observations in \autoref{lma:closed-ideal-Cu-morphism}; hence the equality 
\[
\operatorname{Cu} (A) / \operatorname{Cu}(\ker(\varphi))  = \operatorname{Cu} (A) / 0_{\ker(\operatorname{Cu}(\varphi))}  \; . 
\]
By (iii) in \autoref{prp:idealsnormal}, it follows that $\alpha_{I_{\ker(Cu(\varphi))}}\leq \ker(\Cu(\varphi))$ and note that 
\[
\Cu(A)/{\alpha_{I_{\ker(Cu(\varphi))}}}=\Cu(A)/{\Cu(\ker(\varphi))},
\]	
hence there exists the map $\operatorname{Cu} (A) / \operatorname{Cu}(\ker(\varphi))\to \operatorname{Cu} (A) / \ker(\operatorname{Cu}(\varphi))$ as appears in the diagram. The left-hand part of diagram follows due to the universal property displayed in \autoref{thm:fundamental-theorem0} and \autoref{rmk:fundamental-theorem}.
\end{proof}

\section{Generating prenormal pairs}\label{sec:GenNormalPreorders}

In this section and the next, we discuss how to generate normal pairs from a given relation, not necessarily a preorder, on a $\W$-semigroup, beyond the ones given by ideals. In preparation for this, we first discuss how to extend a $\W$-preorder on a pointed set $(X,0)$ by a given relation $R$ on $X$, provided that $R$ satisfies the appropriate continuity requirement. 
\begin{pgr}[The prenormal pair generated by a relation]
	\label{dfn:axiomatic-merging}
	Let $(X, 0 , \prec_{\phantomtoggle{P}})$ be a $\W$-set and let $R$ be a left $\prec$-continuous relation on $X$, that is, a relation satisfying $\mathrel{{\prec}\circ {R}\relsubseteq{\prec}\circ{R}\circ{\prec}}$. We write $\mathcal{E}_X(\prec,R)$ for the collection of all preorders $\leq$ on $X$ satisfying that, for all $a, b \in X$, any of the following conditions implies $a \leq b$:
	\beginEnumConditions
	\item\label{dfn:axiomatic-merging::P} $a^\prec\subseteq b^\prec$;
	\item\label{dfn:axiomatic-merging::R} $(a,b) \in R$; 
	\item\label{dfn:axiomatic-merging::closure} for any $c \in a^{\prec_{\phantomtoggle{P}}}$, we have $c \leq b$. 
\end{enumerate}
Observing that $\mathcal{E}_X(\prec,R)$ is closed under taking arbitrary intersections, we define $\lequpr{R}$ to be the smallest element in $\mathcal{E}_X(\prec,R)$ under set containment, i.e., the strongest preorder satisfying the above conditions. We also define $\precupr{R}$  to be the composition $ {\lequpr{R}} \circ {\prec_{\phantomtoggle{P}}} \circ  {\lequpr{R}}$. 

Let us point out some simple consequences of our definition above. Notice first that, for each preorder ${\leq}\in\mathcal{E}_X(\prec, R)$, the corresponding pair $(\prec,\leq)$ is admissible. Indeed, suppose that $a^{\leq\circ\prec}\subseteq b^{\leq\circ\prec}$ and let $c\prec a$. Then $c\leq c\prec a$ and thus there is $d$ such that $c\leq d\prec b$. By (i) above $d\leq b$, hence $c\leq b$. Thus, (iii) implies that $a\leq b$. 

Next, observe that, if $R$ and $R'$ are two left $\prec$-continuous relation on $X$ and $R \relsubseteq R'$, then we have $\mathcal{E}_X(\prec,R) \supseteq \mathcal{E}_X(\prec,R')$, and thus ${\lequpr{R}} \relsubseteq {\lequpr{R'}}$ and $ {\precupr{R}}  \relsubseteq   {\precupr{R'}}$. 

The same conclusions as in the above paragraph hold if we weaken the hypothesis $R \relsubseteq R'$ to $R \relsubseteq (\operatorname{id} \relcup R') \circ  {{\leq_\prec}_{\phantomtoggle{P}}}$, where recall that $\leq_\prec$ is the preorder defined by $\prec$, that is, $a\leq_\prec b$ whenever $a^\prec\subseteq b^\prec$. Here we exploit the fact that $\mathcal{E}_X(\prec,R')$ consists of preorders.

Finally, the requirement that condition (iii) above 
implies $a \leq b$ can be replaced by the requirement that $(\prec,\leq)$ is left $\prec$-closed (see \autoref{pgr:prenormal_closed_preorder}).  This is because the other conditions already imply that $ {\leq}  \releq   {\leq_\prec}  \circ  {\leq}$. 

Write $\alphaupr{R}=(\prec,\lequpr{R})$ as well as $\talphaupr{R}=(\precupr{R},\lequpr{R})$. We refer to $\alphaupr{R}$ as the $\prec$-\emph{prenormal closed pair generated by} $R$ and to $\talphaupr{R}$ as the \emph{prenormal extension from $\prec$ by} $R$.  Our terminology is justified by \autoref{cor:merging-continuous} below.
\end{pgr}

We will show that both $\alphaupr{R}$ and $\talphaupr{R}$ are $\prec$-prenormal admissible pairs, but in order to do that, we need to give a more concrete description of $\lequpr{R}$. 

\begin{thm}\label{thm:concrete-merging}
Let $(X, 0 , \prec_{\phantomtoggle{P}})$ be a $\W$-set and let $R$ be a left $\prec$-continuous relation on $X$. For any $a, b \in X$, the following statements are equivalent: 
\beginEnumConditions
\item\label{thm:concrete-merging::axiomatic} $a \lequpr{R} b$; 
\item\label{thm:concrete-merging::prec} for any $c \in a^{\prec_{\phantomtoggle{P}}}$, we have either $c \prec_{\phantomtoggle{P}} b$ or else there are $d_1, \ldots, d_n$ and $e_1, \ldots, e_n$ in $X$ such that 
\[
c \prec_{\phantomtoggle{P}} d_1 R e_1 \prec_{\phantomtoggle{P}} d_2 R e_2 \prec_{\phantomtoggle{P}} \ldots  \prec_{\phantomtoggle{P}} d_n R e_n \prec_{\phantomtoggle{P}} b \; . 
\]
\end{enumerate}
\end{thm}

\begin{proof}
We begin by observing that the condition in \autoref{thm:concrete-merging::prec} may be conveniently written in a more compact form as: for any $c\in a^\prec$, we have $(c,b) \in  {\prec_{\phantomtoggle{P}}}  \circ  (R  \circ  {\prec_{\phantomtoggle{P}}})^{\circ \mathbb{N}}$.

Let us now temporarily write $a \lesssimupr{R} b$ for the relation defined by \autoref{thm:concrete-merging::prec}. We first claim that $\lesssimupr{R}$ is a preorder. Indeed, it is clear that $\lesssimupr{R}$ is reflexive. In order to prove transitivity, it suffices to show that for any $a$, $b$, and $b'$ in $X$, if $a \lesssimupr{R} b$ and $b \lesssimupr{R} b'$, then $a \lesssimupr{R} b'$. To this end, the hypothesis $a \lesssimupr{R} b$ implies that, for any $c \in a^{\prec_{\phantomtoggle{P}}}$, we have $(c,b) \in  {\prec_{\phantomtoggle{P}}}  \circ  (R  \circ {\prec_{\phantomtoggle{P}}} )^{\circ \mathbb{N}}$. Since this composite relation is equal to $( {\prec_{\phantomtoggle{P}}}  \circ  R)^{\circ \mathbb{N}}  \circ  {\prec_{\phantomtoggle{P}}}$, there is $c' \in b^{\prec_{\phantomtoggle{P}}}$ such that $(c,c') \in ( {\prec_{\phantomtoggle{P}}}  \circ  R)^{\circ \mathbb{N}}$ and $c' \prec_{\phantomtoggle{P}} b$. The hypothesis $b \lesssimupr{R} b'$ then implies that $(c',b') \in  {\prec_{\phantomtoggle{P}}}  \circ (R  \circ  {\prec_{\phantomtoggle{P}}})^{\circ \mathbb{N}}$. Merging these two, we have $(c, b') \in ( {\prec_{\phantomtoggle{P}}}  \circ   R)^{\circ \mathbb{N}}  \circ  {\prec_{\phantomtoggle{P}}}  \circ   (R  \circ  {\prec_{\phantomtoggle{P}}})^{\circ \mathbb{N}} \releq  {\prec_{\phantomtoggle{P}}}  \circ   (R  \circ  {\prec_{\phantomtoggle{P}}})^{\circ \mathbb{N}}$. Since $c$ is arbitrary in $a^{\prec_{\phantomtoggle{P}}}$, we have $a \lesssimupr{R} b'$, as desired. 

Thus, in order to prove the theorem it suffices to show that $\lequpr{R}\releq\lesssimupr{R}$.

($\lequpr{R}\relsubseteq\lesssimupr{R}$): By definition, $\lequpr{R}$ is the strongest preorder satisfying the three conditions in \autoref{dfn:axiomatic-merging} to define $\mathcal{E}_X(\prec, R)$, and thus we need to verify that $\lesssimupr{R}$ also satisfies them. To this end, note first that obviously $a^\prec\subseteq b^\prec$ implies $a\lesssimupr{R}b$. Next assume that $a, b \in X$ satisfy $(a,b) \in R$. If $c\prec a$, then we have $(c,b)\in\prec\circ R$ and, using that $R$ is $\prec$-continuous by assumption, this yields $(c,b)\in\mathrel{{\prec}\circ {R}\circ{\prec}}$, and thus $a\lesssimupr{R}b$. Finally, suppose that $c\lesssimupr{R}b$ for any $c\in a^\prec$, and we have to show that $a\lesssimupr{R}b$. Take $c\in a^\prec$, and find by \axiomW{1} an element $c'\in X$ such that $c\prec c'\prec a$. Now by assumption $c'\lesssimupr{R}b$ and since $c\prec c'$, we get $(c,b)\in {\prec_{\phantomtoggle{P}}}  \circ  (R  \circ  {\prec_{\phantomtoggle{P}}})^{\circ \mathbb{N}}$, as desired.

($\lesssimupr{R}\relsubseteq\lequpr{R}$): It suffices to show from \autoref{dfn:axiomatic-merging::closure} in \autoref{dfn:axiomatic-merging} that $c\lequpr{R}b$ for any $c\in a^\prec$. For any such $c$, we have $c\lesssimupr{R}b$, that is,  $(c,b) \in  {\prec_{\phantomtoggle{P}}}  \circ  (R  \circ  {\prec_{\phantomtoggle{P}}})^{\circ \mathbb{N}}$. But $\lequpr{R}$ contains both $\prec$ and $R$, again by construction. Therefore $c\lequpr{R}b$, as required.
\end{proof}

We recall here that a pair $\alpha=(\preccurlyeq,\leq)$ is admissible if $a\leq b$ whenever $a^{\leq\circ\preccurlyeq}\subseteq b^{\leq\circ\preccurlyeq}$, and that $\alpha$ is $\prec$-prenormal provided $\preccurlyeq$ is $\prec$-prenormal and $\leq$ is left $\prec$-continuous; see Paragraphs~\ref{pgr:prenormal_closed_preorder} and \ref{pgr:admissible}.
\begin{cor}\label{cor:merging-continuous}
Let $(X, 0 , \prec_{\phantomtoggle{P}})$ be a $\W$-set and let $R$ be a left $\prec$-continuous relation on $X$. Then 
\beginEnumConditions
\item $\alphaupr{R}$ and $\talphaupr{R}$ are $\prec$-prenormal admissible pairs.
\item $\alphaupr{R}\leq\talphaupr{R}$.
\end{enumerate}
\end{cor}

\begin{proof}
We first establish the following somewhat technical statement (*) that will be helpful to prove the rest of the conditions: $ {\prec_{\phantomtoggle{P}}}  \circ   {\lequpr{R}}  \relsubseteq   {\prec_{\phantomtoggle{P}}}  \circ   (R  \circ  {\prec_{\phantomtoggle{P}}})^{\circ \mathbb{N}} \circ  {\prec_{\phantomtoggle{P}}}$.

To see this, let $a, b, c \in X$ be satisfying $a \prec_{\phantomtoggle{P}} b \lequpr{R} c$. We use the formulation of \autoref{thm:concrete-merging::prec} at the beginning of the proof of \autoref{thm:concrete-merging} to obtain $(a, c) \in  {\prec_{\phantomtoggle{P}}}  \circ   (R  \circ  {\prec_{\phantomtoggle{P}}})^{\circ \mathbb{N}}$. That is to say, $(a,c)\in ( {\prec_{\phantomtoggle{P}}}  \circ   R)^{\circ \mathbb{N}}  \circ  {\prec_{\phantomtoggle{P}}}$. By \axiomW{1}, we have $ {\prec_{\phantomtoggle{P}}}  \relsubseteq   {\prec_{\phantomtoggle{P}}}  \circ   {\prec_{\phantomtoggle{P}}}$, which implies that $(a, c) \in ( {\prec_{\phantomtoggle{P}}}  \circ   R)^{\circ \mathbb{N}}  \circ  {\prec_{\phantomtoggle{P}}}  \circ   {\prec_{\phantomtoggle{P}}}$, i.e., there is $c' \in X$ such that $a \left(  {\prec_{\phantomtoggle{P}}}  \circ   (R  \circ  {\prec_{\phantomtoggle{P}}})^{\circ \mathbb{N}} \right) c' \prec_{\phantomtoggle{P}} c$, as desired. 

Next, to show that $\alphaupr{R}$ and $\talphaupr{R}$ are $\prec$-prenormal, we must prove that ${\precupr{R}}={\precupr{R}}\circ{\prec}$  
and that $\lequpr{R}$ is left $\prec$-continuous. We start with the latter: by definition (see \autoref{dfn:axiomatic-merging}), the relation $ {\prec_{\phantomtoggle{P}}}  \circ   (R  \circ  {\prec_{\phantomtoggle{P}}})^{\circ \mathbb{N}}$ is stronger than $\precupr{R}$, which at the same time is stronger than $\lequpr{R}$. Using (*), the conclusion follows.

We now show that ${\precupr{R}}={\precupr{R}}\circ{\prec}={\lequpr{R}}\circ{\prec}$. We know that, by definition, $ {\precupr{R}}  \releq   {\lequpr{R}}  \circ   {\prec_{\phantomtoggle{P}}}  \circ   {\lequpr{R}}$. Thus, using (*) and the fact that $ {\prec_{\phantomtoggle{P}}}  \circ   (R  \circ  {\prec_{\phantomtoggle{P}}})^{\circ \mathbb{N}}$ is stronger than $\lequpr{R}$, we obtain $\precupr{R}\relsubseteq \lequpr{R}\circ\prec$. That $\lequpr{R}\circ\prec\relsubseteq \precupr{R} $ follows from the fact that $\prec$ is idempotent, by \axiomW{1}, and stronger than $\lequpr{R}$. The implication $\precupr{R}\circ\prec\relsubseteq \lequpr{R}\circ \prec$ holds because $\precupr{R}$ is stronger than $\lequpr{R}$. The converse implication is a direct consequence of \axiomW{1}.

We verify that $a^{\lequpr{R}\circ\prec}\subseteq b^{\lequpr{R}\circ\prec}$ implies $a\lequpr{R}b$. From this it will follow that $\alphaupr{R}$ is an admissible pair. We use the description of $\lequpr{R}$ in \autoref{thm:concrete-merging}. Thus, take $c\prec a$ and choose by \axiomW{1} $c'$ with $c\prec c'\prec a$. Then $c'\lequpr{R}c'\prec a$ and by assumption there is $d$ such that $c\lequpr{R} d\prec b$. Therefore, either $c\prec d$, or else there are elements $d_1,e_1,\dots, d_n,e_n$ such that $c \prec_{\phantomtoggle{P}} d_1 R e_1 \prec_{\phantomtoggle{P}} d_2 R e_2 \prec_{\phantomtoggle{P}} \ldots  \prec_{\phantomtoggle{P}} d_n R e_n \prec_{\phantomtoggle{P}} d\prec b$. Thus $a\lequpr{R}b$.

To show that $\talphaupr{R}$ is also admissible, it suffices to prove that $a^{\lequpr{R}\circ\precupr{R}}\subseteq b^{\lequpr{R}\circ\precupr{R}}$ implies $a^{\lequpr{R}\circ\prec}\subseteq b^{\lequpr{R}\circ\prec}$ and apply the conclusion of the paragraph above. Thus, suppose that $c,d\in X$ satisfy $c\lequpr{R}d\prec a$. By definition of $\precupr{R}$ this implies that $c\lequpr{R}d\precupr{R}a$, hence by assumption there is $e\in X$ such that $c\lequpr{R}e\precupr{R}b$. Again by definition of $\lequpr{R}$, we may find $e',b'\in X$ such that $c\lequpr{R}e\lequpr{R}e'\prec b'\lequpr{R}b$. By continuity of $\lequpr{R}$, there are elements $e'', b''\in X$ such that $e'\prec e''\lequpr{R}b''\prec b$. Putting all this together we obtain $c\lequpr{R}b''\prec b$, as desired. This concludes the proof of (i).

(ii) follows from \autoref{lma:auxiliary} (i) and the fact that, by definition, we have $\prec \relsubseteq\precupr{R}$.
\end{proof}
We say that an admissible pair $\alpha=(\prec,\leq)$ \emph{contains} a relation $R$ provided ${R}\relsubseteq{\leq}$.

\begin{thm}\label{thm:merging-W-preorder}
Let $(X, 0 , \prec_{\phantomtoggle{P}})$ be a $\W$-set and let $R$ be a left $\prec$-continuous relation on $X$. Then $\alphaupr{R}$ is the smallest $\prec$-prenormal closed admissible pair containing $R$.~Furthermore, $(X,0,\precupr{R},\lequpr{R})$ satisfies \axiomW{2}.
\end{thm}

\begin{proof}
We have already proved in \autoref{cor:merging-continuous}  that  $\alphaupr{R}$ is $\prec$-prenormal. To see that it is also closed, we apply \autoref{prp:pre_dense_implies_pre_order}. Thus we have to check that $\talphaupr{R}=(\precupr{R},\lequpr{R})$ is $\prec$-prenormal, minimal admissible, and auxiliary.

That $\talphaupr{R}$ is $\prec$-prenormal follows again from  \autoref{cor:merging-continuous}, and is clearly auxiliary. Hence, we only need to show that $\precupr{R}$ induces $\lequpr{R}$, that is,
for $a,b\in X$, we have $a\lequpr{R}b$ if and only if $a^{\precupr{R}}\subseteq b^{\precupr{R}}$. The forward implication is clear. Suppose, conversely, that $c\prec a$. Then  $c\precupr{R}a$ which by assumption implies $c\precupr{R}b$. Thus $c\lequpr{R}b$. We have then verified condition (iii) in \autoref{dfn:axiomatic-merging}, and so $a\lequpr{R}b$. 

By \autoref{prp:pre_dense_implies_pre_order} and its proof, $(X,0,\precupr{R},\lequpr{R})$ satisfies \axiomW{2}.

Finally, let $\alpha=(\preccurlyeq,\leq)$ be a $\prec$-prenormal closed admissible pair containing $R$. We first verify that $\leq$ belongs to $\mathcal{E}_X(\prec,R)$ and thus ${\lequpr{R}}\relsubseteq{\leq}$. To see this, we have to check that $\leq$ satisfies the conditions in \autoref{dfn:axiomatic-merging}, of which (ii) is automatic. Now, suppose that $a^\prec\subseteq b^\prec$. Then, by prenormality we have ${\preccurlyeq}={\preccurlyeq}\circ{\prec}$, which implies $a^\preccurlyeq\subseteq b^\preccurlyeq$ and, since $\alpha$ is admissible, we obtain $a\leq b$. Thus (i) is also satisfied. Next, if $c\leq b$ for any $c\prec a$, then for any such $c$ choose $c'\in X$ such that $c\prec c'\prec a$ and we have $c\prec c'\leq b$. Since $\alpha$ is closed, there is $d\in X$ such that $a^\prec\subseteq d^\prec$ and $d\leq b$. By what we just proved, this implies that $a\leq d\leq b$. By \autoref{lma:auxiliary} (i), we have $\alphaupr{R}\leq \alpha$, as desired. 
\end{proof}

The following corollary shows that repeated generations of prenormal pairs can be contracted into a single one. 

\begin{cor}\label{cor:merging-associative}
Let $(X, 0 , \prec_{\phantomtoggle{P}})$ be a $\W$-set and let $R$ and $R'$ be left $\prec$-continuous relations on $X$. Let $\alphauprr{R}{R'}=(\prec,\lequprr{R}{R'})$ be the $\prec$-prenormal pair generated by $R'$ and write $\precuprr{R}{R'}$ to mean ${\lequprr{R}{R'}} \circ  {\precupr{R}}  \circ {\lequprr{R}{R'}}$. Then $\alphauprr{R}{R'}=\alphaupr{R\cup R'}$ and $(\widetilde{\alphaupr{R}})^\text{\miniscule $R'$}=\talphaupr{R\cup R'}$.
\end{cor}
\begin{proof}
We must prove that 	${\lequprr{R}{R'}}  \releq {\lequpr{R \relcup R'}}$ and ${\precuprr{R}{R'}} \releq  {\precupr{R \relcup R'}}$.	

We first show that $\lequprr{R}{R'}$ and $\lequpr{R\cup R'}$ agree. To show the former is stronger than the latter, by definition of $\lequprr{R}{R'}$ we need to verify that $\lequpr{R \relcup R'}$ satisfies conditions (i)-(iii) defining $\mathcal{E}_X(\precupr{R}, R')$ as in~\autoref{dfn:axiomatic-merging}. 

To verify condition (i), assume $a^{\precupr{R}}\subseteq b^{\precupr{R}}$. Then $a\lequpr{R}b$. Since we clearly have $\mathcal{E}_X(\prec , R) \supseteq \mathcal{E}_X(\prec , R \relcup R')$, it follows that $a\lequpr{R\cup R'}b$. Condition (ii) is obvious since $R' \relsubseteq R \relcup R'$. For condition (iii), assume that $c\lequpr{R\cup R'}$ for any $c\in a^{\precupr{R}}$. If $c\prec a$, then in particular $c\precupr{R}a$, and thus by definition of $\lequpr{R\cup R'}$ we obtain $a\lequpr{R\cup R'}b$.

Next, to show that $\lequpr{R\cup R'}$ is stronger than $\lequprr{R}{R'}$, we proceed similarly to verify that $\lequprr{R}{R'}$ satisfies conditions (i)-(iii) defining $\mathcal{E}_X(\prec, R\cup R')$ as in~\autoref{dfn:axiomatic-merging}. 

To check condition (i), observe that if $a^\prec\subseteq b^\prec$, then $a\lequpr{R}b$ which, as we have proved in \autoref{thm:merging-W-preorder}, means that $a^{\precupr{R}}\subseteq b^{\precupr{R}}$. This implies that $a\lequprr{R}{R'}b$. Condition (ii) follows from the clear containment $R'\relsubseteq \lequprr{R}{R'}$, and also from $R\relsubseteq \lequpr{R}\relsubseteq \lequprr{R}{R'}$. Finally, to verify condition (iii), assume that $c\lequprr{R}{R'}b$ for any $c\prec a$, and we need to show that $a\lequprr{R}{R'}b$. To this end, let $d\precupr{R}a$. We know from the proof of \autoref{cor:merging-continuous} that ${\precupr{R}}={\lequpr{R}\circ{\prec}}$. Hence, there is $c\in X$ such that $d\lequpr{R}c\prec a$. By assumption this implies that $d\lequpr{R}c\,\,\lequprr{R}{R'}\,\,b$, and thus $d\,\,\lequprr{R}{R'}\,\,b$. Since $d$ is arbitrary in $a^{\precupr{R}}$, we have $a\lequprr{R}{R'}b$, as desired.

Using the definitions of $\precuprr{R}{R'}$, $\precupr{R}$, and $\precupr{R\cup R'}$ at the first, second, and last steps respectively, the simple fact that $\lequprr{R}{R'}\circ \lequpr{R}\releq \lequprr{R}{R'}$ at the third step, and the equivalence $\lequprr{R}{R'}\releq\lequpr{R\relcup R'}$ at the fourth step, we obtain
\begin{align*}
{\precuprr{R}{R'}} & \releq  {\lequprr{R}{R'}} \circ  {\precupr{R}}  \circ  {\lequprr{R}{R'}} 
\releq  {\lequprr{R}{R'}} \circ  {\lequpr{R}}  \circ   {\prec_{\phantomtoggle{P}}}  \circ   {\lequpr{R}}   \circ {\lequprr{R}{R'}} \\
& \releq  {\lequprr{R}{R'}} \circ  {\prec_{\phantomtoggle{P}}}  \circ  {\lequprr{R}{R'}}
\releq   {\lequpr{R \relcup R'}} \circ  {\prec_{\phantomtoggle{P}}}  \circ   {\lequpr{R \relcup R'}} \\
& \releq   {\precupr{R \relcup R'}} \; .\qedhere
\end{align*}	
\end{proof}


Since it is cumbersome to deal with compositions of indefinite lengths, we discuss when this can be avoided. 

\begin{pgr}[Almost transitive relations]\label{dfn:almost-transitive}
Let $(X, 0 , \prec_{\phantomtoggle{P}})$ be a $\W$-set. A relation $R$ on $X$ is \emph{almost transitive} (or more precisely, \emph{$\prec_{\phantomtoggle{P}}$-almost transitive}) if 
\[
{\prec_{\phantomtoggle{P}}}	\circ R  \circ  {\prec_{\phantomtoggle{P}}}  \circ   R \circ {\prec_{\phantomtoggle{P}}} \relsubseteq  {\prec_{\phantomtoggle{P}}}  \circ   R  \circ  {\prec_{\phantomtoggle{P}}} \; .
\]
It is clear that a transitive relation $R$ is $\prec_{\phantomtoggle{P}}$-almost transitive provided that it satisfies 
${\prec_{\phantomtoggle{P}}}  \circ   R \relsubseteq     R  \circ  {\prec_{\phantomtoggle{P}}}$ or $R \circ {\prec_{\phantomtoggle{P}}}   \relsubseteq     {\prec_{\phantomtoggle{P}}} \circ   R $. 
\end{pgr}

\begin{prp}\label{prp:one-step}
Let $(X, 0 , \prec_{\phantomtoggle{P}})$ be a $\W$-set and let $R$ be a left $\prec$-continuous relation on $X$. If $R$ is $\prec$-almost transitive, then for any $a, b \in X$, the following statements are equivalent: 
\beginEnumConditions
\item\label{prp:one-step::axiomatic} $a \lequpr{R} b$; 
\item\label{prp:one-step::prec} for any $c \in a^{\prec_{\phantomtoggle{P}}}$, we have either that $c\prec b$ or else $(c,b) \in  {\prec_{\phantomtoggle{P}}}  \circ   R  \circ  {\prec_{\phantomtoggle{P}}}$. 
\end{enumerate}
Moreover, if, in addition, we have $\operatorname{id} \relsubseteq R$, then we may replace 
${\prec_{\phantomtoggle{P}}} \relcup ( {\prec_{\phantomtoggle{P}}}  \circ   R  \circ  {\prec_{\phantomtoggle{P}}})$ by $ {\prec_{\phantomtoggle{P}}}  \circ   R  \circ  {\prec_{\phantomtoggle{P}}}$ in the above. 
\end{prp}
\begin{proof}
Since $R$ is almost transitive, we have $ {\prec_{\phantomtoggle{P}}}  \circ   (R  \circ  {\prec_{\phantomtoggle{P}}})^{\circ 2} \relsubseteq  {\prec_{\phantomtoggle{P}}}  \circ   (R  \circ  {\prec_{\phantomtoggle{P}}})^{\circ 1}$. Applying this inductively, we see that 
\[	
{\prec_{\phantomtoggle{P}}}  \circ   (R  \circ  {\prec_{\phantomtoggle{P}}})^{\circ \mathbb{N}} \releq  {\prec_{\phantomtoggle{P}}} \relcup ( {\prec_{\phantomtoggle{P}}}  \circ  R  \circ  {\prec_{\phantomtoggle{P}}}).
\]	 
In light of Theorem~\ref{thm:concrete-merging}, this yields the desired equivalence.	

For the statement after ``moreover'', we simply use the equivalence
$ {\prec_{\phantomtoggle{P}}}  \releq   {\prec_{\phantomtoggle{P}}}  \circ   {\prec_{\phantomtoggle{P}}}  \releq   {\prec_{\phantomtoggle{P}}}  \circ  \operatorname{id} \circ  {\prec_{\phantomtoggle{P}}}$. 
\end{proof}


The following corollary shows how the usual Cuntz subequivalence on positive elements in a C*-algebra arises naturally from extending the $\W$-preorder in Example~\ref{exa:precsimher} by a kind of generalized Murray-von Neumann (sub-)equivalence. This answers a question raised by G.~A.~Elliott.

\begin{cor}\label{cor:precsimher-axiomatic}
Let $A$ be a C*-algebra. Let $\alpha = (\llher, \leqher)$ be the admissible pair on $(A_+, 0)$ as in Example~\ref{exa:precsimher}. Then the Cuntz subequivalence $\precsim_{\operatorname{Cu}}$ is the smallest preorder $\leq$ on $A_+$ such that for all $a, b \in A_+$, any of the following conditions implies $a \leq b$:
\beginEnumConditions
\item\label{cor:precsimher-axiomatic::P} $a \leqher b$; 
\item\label{cor:precsimher-axiomatic::R} there exists $x \in A$ such that $a = x b x^*$; 
\item\label{cor:precsimher-axiomatic::closure} for any $c \in a^{\llher}$, we have $c \leq b$. 
\end{enumerate}
\end{cor}

\begin{proof}
Let $R$ be the relation on $A_+$ defined so that $(a,b) \in R$ if and only if there exists $x \in A$ such that $a = x b x^*$. Observe that $R$ is left $\llher$-continuous. Indeed, for any $a, b, c \in A_+$ satisfying $a \llher b R c$, we find $\varepsilon > 0$ and $x \in A$ such that $a \llher (b - \varepsilon)_+$ and $b = x^* c x$. Set $L = \max\{1, \|x^* x\| \}$. Then we have 
\[
x^* \left( c - \left(c - \frac{\varepsilon}{L} \right)_+ \right) x \leq x^* \left( c - \left(c - \frac{\varepsilon}{L} \right) \right) x = \frac{\varepsilon}{L} x^* x \leq \varepsilon \; ,
\]
which implies
\begin{align*}
(x^* c x - \varepsilon)_+  & \leq \left(x^* c x - x^* \left( c - \left(c - \frac{\varepsilon}{L} \right)_+ \right) x \right)_+ \\ &  = \left(x^* \left(c - \frac{\varepsilon}{L} \right)_+ x \right)_+ \leq x^* \left(c - \frac{\varepsilon}{L} \right)_+ x \; .
\end{align*}
Hence, setting $b' = x^* \left(c - \frac{\varepsilon}{L} \right)_+ x$ and $c' = \left(c - \frac{\varepsilon}{L} \right)_+$, we have $a \llher b'$, $(b',c') \in R$, and $c' \llher c$, which shows $R$ is left $\llher$-continuous. 

Since the ``smallest preorder'' defined in the statement of the corollary is $\lequpr{R}$ applied to the $\W$-set $(A_+,0,\llher)$
as defined in \autoref{dfn:axiomatic-merging}, we need to show that $\precsimr{\operatorname{Cu}}$ agrees with $\lequpr{R}$. (Notice also that $a\leqher$ if and only if $a^{\llher}\subseteq b^{\llher}$.)

To this end, we observe that $R$ is $\llher$-almost transitive. Suppose first that $a,b,c,d\in A_+$ satisfy $aRb$, $b \llher c$, and $cRd$. We claim that $(a,d)\in {\leqher}\circ {R}\circ {\leqher}$. Choose $x, y\in A$ such that $a=xbx^*$, $c=ydy^*$, and let $\varepsilon>0$ such that $b$ belongs to the hereditary algebra generated by $(c-\varepsilon)_+$. Let $f$ be a continuous function such that $f(c)$ is a unit for $(c-\varepsilon)_+$. Then $f(c)bf(c)=bf(c)=b$ and thus $b\leq Nc$ for some $N$. Therefore $xbx^*\leq Nxcx^*$ and thus it belongs to the hereditary algebra generated by $xcx^*$. This implies that $a\leqher xcx^*$, with $xcx^*=xyd(xy)^*$, that is, $xcx^*Rd$. Thus $(a,d)\in {\leqher}\circ {R}\circ {\leqher}$, as claimed.

Using the above claim at the second step and that $\llher$ is auxiliary for $\leqher$ at the last step, we obtain 
\[
\llher\circ\mathrel{R}\circ\llher\circ\mathrel{R}\circ\llher\relsubseteq\llher\circ {\leqher}\circ \mathrel{R}\circ {\leqher}\circ\llher\relsubseteq\llher\circ\mathrel{R}\circ\llher,
\]	
which shows that $R$ is $\llher$-almost transitive.

Now we verify that 
$\precsimr{\mathrm{Cu}}$ is equivalent to condition (ii) in \autoref{prp:one-step}. To see this, assume that $a\precsimr{\mathrm{Cu}}b$ and let $c\llher a$. Then $c\in A_{(a-\frac{\varepsilon}{2})_+}$ for some $\varepsilon>0$. Choose $x\in A$ and $\delta>0$ such that $\Vert (a-\varepsilon/2)_+-x(b-\delta)_+x^*\Vert<\varepsilon/2$. This implies, using \cite[Lemma 2.2]{KirRor02InfNonSimpleCalgAbsOInfty}, that there is $y\in A$ such that $(a-\varepsilon)_+=y(b-\delta)_+y^*$. Thus $c\llher (a-\varepsilon)_+R(b-\delta)_+\llher b$. Conversely, suppose that $a,b$ satisfy condition (ii) in \autoref{prp:one-step}, and let $\varepsilon>0$. Then $(a-\varepsilon)_+\llher a$ and thus there exist $c, d\in A_+$, $\gamma,\delta>0$, and $x\in A$ such that $a-\varepsilon)_+\in A_{(c-\gamma)_+}$, $c=xdx^*$, and $d\in A_{(b-\delta)_+}$, which clearly implies that $a\precsimr{\mathrm{Cu}}b$.
\end{proof}
\begin{rmk}\label{rmk:precsimher-axiomatic}
In \autoref{cor:precsimher-axiomatic}, we may obtain equivalent formulations by making either or both of the following replacements:
\begin{itemize}
\item Condition (iii)
may be replaced by the following: 
\begin{itemize}
\item[(iii')] there is an increasing sequence $(a_n)_{n \in \mathbb{N}}$ in $A_+$ whose supremum $\sup_n^{\leqher} a_n$ with regard to $\leqher$ exists and is equal to $a$, and $a_n \leq b$ for every $n \in \mathbb{N}$. (Here, $\leq$ is not to be confused with the natural order in $A_+$.)
\end{itemize}
\item Condition (ii)
may be replaced by the following: 
\begin{itemize}
\item[(ii')] there exists $x \in A$ such that $a = x x^*$ and $b = x^* x$.
\end{itemize}
\end{itemize}
\end{rmk}

\section{Generating normal pairs}
\label{sec:extendingnormal}
In this secion we build upon the work in \autoref{sec:GenNormalPreorders} in order to extend pairs to normal pairs on $\W$-semigroups.
\begin{pgr}[The normal pair generated by a relation]\label{dfn:axiomatic-merging-additive}
	Let $(S, \prec)$ be a $\W$-semigroup and let $R$ be a left $\prec$-continuous relation on $S$. We write $\mathcal{E}_S^+ (R)$ for the collection of all preorders $\leq$ on $S$ satisfying that, for all $a, b \in S$, any of the following conditions implies $a \leq b$:
	\beginEnumConditions
	\item\label{dfn:axiomatic-merging-additive::P} $a^\prec \subseteq b^\prec$; 
	\item\label{dfn:axiomatic-merging-additive::R} $(a,b) \in R$; 
	\item\label{dfn:axiomatic-merging-additive::closure} for any $c \in a^{\prec}$, we have $c \leq b$;
	\item\label{dfn:axiomatic-merging-additive::additive} there are $a_1, a_2, b_1, b_2$ in $S$ such that $a = a_1 + a_2$, $b = b_1 + b_2$, and $a_k \leq b_k$ for $k = 1, 2$.
\end{enumerate}
Observing that $\mathcal{E}_S^+ (R)$ is closed under taking arbitrary intersections, we define $\leqr{R}$ to be the smallest element in $\mathcal{E}_S^+ (R)$ under set containment, i.e., the strongest pre-order satisfying the above conditions. We also define $\precr{R}$ to be the composition $ {\leqr{R}}  \circ   {\precr{R}}  \circ  {\leqr{R}}  $.

Write $\alphar{R}=(\prec,\leqr{R})$ as well as $\talphar{R}=(\precr{R},\leqr{R})$. We refer to $\alphar{R}$ as the $\prec$-\emph{normal closed pair generated by} $R$ and to $\talphar{R}$ as the \emph{normal extension from $\prec$ by} $R$. Again our terminology is justified below.

We write $S \doubleslash R$ for the quotient $\W$-semigroup $S / {\alphar{R}}$ (that is, the antisymmetrization of $S_{\alphar{R}}$; see \autoref{pgr:quotient}).  This generalizes the notion of the quotient of a $\Cu$-semigroup by a closed ideal (see \cite[Section 5.1]{AntPerThi14arX:TensorProdCu}).
\end{pgr}
The following lemmas show that the preorder $\leqr{R}$ can equivalently be induced by first taking the additive closure of $R$ (plus the identity relation $\operatorname{id}$) and then applying the construction in \autoref{dfn:axiomatic-merging} without any concern about the additive structure. 

\begin{lma}\label{lma:additive-closure-continuity}
Let $(S, \prec)$ be a $\W$-semigroup and let $R$ and $R'$ be left $\prec$-continuous relations on $S$. Then the sum $R + R'$ and the additive closure $R_+$ are also left $\prec$-continuous. 
\end{lma}

\begin{proof}
It suffices to show that $R+R'$ is left $\prec$-continuous. To this end, let $a,b,b',c,c'\in S$ be such that $a\prec (b+b')(R+R')(c+c')$. This means in particular that $bRc$ and $b'Rc'$. Using \axiomW{4}, find $b_1, b_1'\in S$ such that $b_1\prec b$, $b_1'\prec b'$, and $a\prec b_1+b_1'$. Since both $R$ and $R'$ are left $\prec$-continuous, there are $b_2, b_2', c_1,c_1'\in S$ with $b_1\prec b_2Rc_1\prec c$ and $b_1'\prec b_2'Rc_1'\prec c'$. This is now easily seen to imply that $a\prec (b_2+b_2')R(c_1+c_1')\prec c+c'$, as desired.
\end{proof}

\begin{lma}\label{lma:axiomatic-merging-additive-closure}
Let $(S, \prec)$ be a $\W$-semigroup and let $R$ be a left $\prec$-continuous relation on $S$. Then the preorders $\leqr{R}$ and $\lequpr{(R + \operatorname{id})_+}$ (see \autoref{dfn:axiomatic-merging} and~\autoref{dfn:additive-closure}) coincide. Therefore $\alphar{R}=\alphaupr{(R+\operatorname{id})_+}$.
\end{lma}

\begin{proof}
We first show from the definitions that 
\[	
\mathcal{E}_S^+ (R) \subseteq \mathcal{E}_S^+ ((R + \operatorname{id})_+) \subseteq \mathcal{E}_S(\prec, (R + \operatorname{id})_+),
\]
which implies that $\leqr{R}$ is weaker than $\lequpr{(R + \operatorname{id})_+}$. For the first inclusion, one only needs to show that, if ${\leq}\in{\mathcal{E}_S^+ (R)} $ and $a,b\in S$ satisfy $a((R + \operatorname{id})_+)b$, then $a\leq b$. By definition, there is $n$ and $a_{ij}, b_{ij}$, for $i=1,\dots,n$, $j=1,2$ such that $a_{i1}Rb_{i1}$, and $a_{i2}=b_{i2}$, and $a=\sum_i\sum_j a_{ij}$, $b=\sum_i\sum_jb_{ij}$. In particular, we get $a_{ij}\leq b_{ij}$ for all $i, j$ and thus $a\leq b$. The second inclusion is trivial.

We next use the concrete realization of $\lequpr{(R + \operatorname{id})_+}$ (see \autoref{thm:concrete-merging}) to show that $\lequpr{(R + \operatorname{id})_+}$ is additively closed, i.e., satisfies condition~\ref{dfn:axiomatic-merging-additive::additive} in \autoref{dfn:axiomatic-merging-additive}, which implies that $\leqr{R}$ is equal to $\lequpr{(R + \operatorname{id})_+}$. In order to check this, assume that $a_i\lequpr{(R + \operatorname{id})_+}b_i$ for $i=1,2$. If $c\prec a_1+a_2$, then using \axiomW{4} find $c_1,c_2\in S$ such that $c_i\prec a_i$ and $c\prec c_1+c_2$. The only case of interest is when there are elements $d_i^{(j)},e_i^{(j)}\in S$, for $j=1,2$, and $i=1,\dots,n$ such that
\[
c_1 \prec_{\phantomtoggle{P}} d_1^{1}{(R + \operatorname{id})_+}e_1^{(1)} \prec_{\phantomtoggle{P}} d_2^{1}{(R + \operatorname{id})_+}e_2^{(1)} \prec_{\phantomtoggle{P}} \dots \prec_{\phantomtoggle{P}}  d_n^{(1)}{(R + \operatorname{id})_+}e_n^{(1)} \prec_{\phantomtoggle{P}} b_1,
\]
and similarly for $c_2$ (notice that we can take the same length $n$ for both $c_1$ and $c_2$). It is then clear that
\begin{align*}
	c_1+c_2 & \prec_{\phantomtoggle{P}} (d_1^{1}+d_1^{(2)}){(R + \operatorname{id})_+}(e_1^{(1)}+e_1^{(2)}) \\ &  \prec_{\phantomtoggle{P}} (d_2^{1}+d_2^{2}){(R + \operatorname{id})_+}(e_2^{(1)}+e_2^{(2)}) \prec_{\phantomtoggle{P}} \dots \\ & \prec_{\phantomtoggle{P}}  (d_n^{(1)}+d_n^{(2)}){(R + \operatorname{id})_+}(e_n^{(1)}+e_n^{(2)}) \prec_{\phantomtoggle{P}} b_1+b_2,
\end{align*}
as desired.
\end{proof}
\begin{rmk}\label{rmk:plus_identity_relation}
We have $(R + \operatorname{id})_+ \releq R_+ + \operatorname{id}$, since $\operatorname{id}$ is clearly additively closed. Hence if $\operatorname{id} \relsubseteq R$, then $(R + \operatorname{id})_+ \releq R_+$. 
\end{rmk}
\begin{cor}
\label{cor:alphaupr} Let $(S,\prec)$ be a $\W$-semigroup and let $R$ be a left $\prec$-continuous relation on $S$. Then, the pairs $\alphar{R}$ and $\talphar{R}$ are $\prec$-normal admissible and $\alphar{R}\leq\talphar{R}$.
\end{cor}	

\begin{proof}
From \autoref{lma:axiomatic-merging-additive-closure} we have \[\alphar{R}=(\prec,\leqr{R})=(\prec, \lequpr{(R + \operatorname{id})_+})
\] 
and 
\[
\talphar{R}=(\precr{R},\leqr{R})=(\leqr{R}\circ\prec\circ\leqr{R},\leqr{R})=(\precupr{(R+\operatorname{id})_+},\lequpr{(R+\operatorname{id})_+}).
\]
Thus the result follows from \autoref{lma:additive-closure-continuity} and \autoref{cor:merging-continuous}.
\end{proof}

\begin{thm}\label{thm:generating-normal-preorders}
Let $(S, \prec)$ be a $\W$-semigroup and let $R$ be a left $\prec$-continuous relation on $S$. Then $\alphar{R}$ is the smallest $\prec$-normal closed admissible pair that contains $R$.
\end{thm}

\begin{proof}
It follows from \autoref{lma:axiomatic-merging-additive-closure} and \autoref{thm:merging-W-preorder} that $\alphar{R}$ is the smallest $\prec$-prenormal closed admissible pair that contains $(R+\operatorname{id})_+$. Since $\alphar{R}$ is additively closed by definition, it is also a normal closed pair. Now let $\alpha=(\preccurlyeq,\leq)$ be another $\prec$-normal closed admissible pair that contains $R$. Since $\leq$ is additively closed we have that it must also contains $(R+\operatorname{id})_+$ and thus $\alphar{R}=\alphaupr{(R+\operatorname{id})_+}\leq \alpha$.
\end{proof}


The quotient $\W$-semigroup $S  \doubleslash  R$ enjoys the following universal property. 

\begin{thm}\label{thm:quotient-universal-property}
Let $(S, \prec)$ be a $\W$-semigroup and let $R$ be a left $\prec$-continuous relation on $S$. Write $\pi_R \colon S \to S  \doubleslash  R$ for the canonical surjection. Then for any $\W$-morphism $f \colon S \to T$ such that $\ker(f)$ contains $R$, there is a unique $\W$-morphism $\overline{f} \colon S  \doubleslash  R \to T/{\alpha_T}$ such that $\pi_{\alpha_T}\circ f = \overline{f} \circ \pi_R$. 
\[
\xymatrix{
	S \ar[d]_{\pi_R} \ar[r]^{f} & T\ar[d]^{\pi_{\alpha_T}} \\
	S  \doubleslash  R \ar@{-->}[r]_{\overline{f}} & T/{\alpha_T}
}
\]
Moreover, the pair $(S  \doubleslash  R, \pi_R)$ is the unique pair of a $\CatW$-semigroup and a $\CatW$-morphism that satisfies the above universal property. 
\end{thm}

\begin{proof}
Recall that $\ker(f)=(\prec,\leq_f)$, which is an admissible $\prec$-normal closed pair (see \autoref{lma:kernelnormal}). We need to show that $\alphar{R}\leq\ker(f)$ and then apply the fundamental theorem for $\W$-semigroups (see \autoref{thm:fundamental-theorem0} and \autoref{rmk:fundamental-theorem}). Since both $\alphar{R}$ and $\ker(f)$ are $\prec$-normal, by \autoref{lma:auxiliary}(i) we only need to verify that $\leqr{R}\relsubseteq\ker(f)$, which follows from \autoref{thm:generating-normal-preorders} using the assumption that $\ker(f)$ contains $R$.
\end{proof}

\begin{pgr}[Almost refinement] \label{pgr:almost-refinement}
Recall from \autoref{dfn:almost-transitive} that a relation $R$ on a semigroup $S$ is almost transitive if ${\prec_{\phantomtoggle{P}}}	\circ R  \circ  {\prec_{\phantomtoggle{P}}}  \circ   R \circ {\prec_{\phantomtoggle{P}}} \relsubseteq  {\prec_{\phantomtoggle{P}}}  \circ   R  \circ  {\prec_{\phantomtoggle{P}}}$. This was used successfully in \autoref{prp:one-step} to simplify the expression of the relation $\lequpr{R}$.

It is not clear whether almost transitivity passes to additive closures. However, below we give a sufficient condition for this to happen (see Proposition~\ref{prp:almost-refinement-transitivity}). This condition is also present in a number of situations.

We say a $\W$-semigroup $(S,\prec)$ has \emph{almost refinement} if, given $a_1,a_2,a_1',a_2',b_1, b_2 \in S$ such that $a_i' \prec a_i$ for $i = 1,2$ and $a_1 + a_2 \prec b_1 + b_2$, then there are elements $x_{ij}\in S$ for $i,j=1,2$ such that $a_i' \prec x_{i1}+x_{i2}$ for $i = 1,2$ and $x_{1j}+x_{2j}\prec b_j$ for $j = 1,2$. 

It follows by induction that $S$ has almost refinement if and only if for any integers ${m}, {n} \geq 2$ and any $a_1, \ldots, a_{m}, a_1',\ldots, a_{m}', b_1, \ldots,  b_{n} \in S$ satisfying $a_i' \prec a_i$ for $i = 1, \ldots, {m}$ and $a_1 + \ldots + a_{m} \prec b_1 + \ldots + b_n$, there are elements $x_{ij}\in S$ for $i = 1, \ldots, {m}$ and $j = 1, \ldots, {n}$ such that $a_i' \prec x_{i1} + \ldots + x_{in}$ for $i = 1, \ldots, {m}$ and $x_{1j} + \ldots + x_{{m}j}\prec b_j$ for $j = 1, \ldots, {n}$. 

More generally, a relation $R$ on $S$ has \emph{$\prec$-almost $({m},{n})$-refinement} for some positive integers ${m}$ and ${n}$ if for any $a_1, \ldots, a_{m}, a_1',\ldots, a_{m}', b_1, \ldots,  b_{n} \in S$ satisfying $a_i' \prec a_i$ for $i = 1, \ldots, {m}$ and $\left(a_1 + \ldots + a_{m}\right) \mathrel{\prec \circ \mathrel{R} \circ \prec} \left(b_1 + \ldots + b_n\right)$, there are elements $x_{ij}, y_{ij} \in S$ for $i = 1, \ldots, {m}$ and $j = 1, \ldots, {n}$ such that $a_i' \prec x_{i1} + \ldots + x_{in}$, $x_{ij} \mathrel{\prec \circ \mathrel{R} \circ \prec} y_{ij}$, and $y_{1j} + \ldots + y_{{m}j}\prec b_j$   for $i = 1, \ldots, {m}$ and $j = 1, \ldots, {n}$. We also say $R$ has \emph{$\prec$-almost refinement} if it has $\prec$-almost $({m},{n})$-refinement for any positive integers ${m}$ and ${n}$. 

It follows from the definition of a $\W$-semigroup that the relation $\prec$ has $\prec$-almost refinement if and only if $S$ has almost refinement. 

It is also an easy exercise to verify that, if $S$ is a semigroup that has the Riesz decomposition property and the Riesz refinement property, then $S$ has the almost refinement property. In particular, this is the case for $S=\W(A)$, where $A$ is a \ca{} of real rank zero and stable rank one (see, e.g. \cite{Per97StructurePositive}).
\end{pgr}
The proof of the following lemma is routine.

\begin{lma} \label{lma:almost-refinement-basic}
Let $R$ be a relation on a $\W$-semigroup $(S, \prec)$. 
\beginEnumConditions
\item \label{lma:almost-refinement-basic:inherit} If $R$ has $\prec$-almost $({m},{n})$-refinement for some positive integers ${m}$ and ${n}$, then it has $\prec$-almost $(p,q)$-refinement for any positive integers $p \leq {m}$ and $q \leq {n}$. 
\item \label{lma:almost-refinement-basic:product} If $S$ has almost refinement and $R$ has $\prec$-almost $({m},{n})$-refinement and $\prec$-almost $(p,q)$-refinement for some positive integers ${m}$, ${n}$, $p$, and $q$, then $R$ has $\prec$-almost $({m} p,{n} q)$-refinement.  
\item \label{lma:almost-refinement-basic:generate} If $S$ has almost refinement and $R$ has $\prec$-almost $(2,1)$-refinement and $\prec$-almost $(1,2)$-refinement, then $R$ has $\prec$-almost refinement.  
\end{enumerate}
\end{lma}

\begin{prp}\label{prp:almost-refinement-transitivity}
Let $(S, \prec)$ be a $\W$-semigroup having almost refinement. Let $R$ be a left $\prec$-continuous relation on $S$ that has $\prec$-almost refinement. If $R$ is $\prec$-almost transitive in the sense of \autoref{dfn:almost-transitive}, so are the additive closures $R_+$ and $(R + \operatorname{id})_+$. 
\end{prp}

\begin{proof}
We prove the case for $R_+$, that is, $\prec \circ \mathrel{R_+}  \circ  \prec  \circ  \mathrel{R_+} \circ \prec \relsubseteq  \prec \circ  \mathrel{R_+}  \circ  \prec$, the case for $(R + \operatorname{id})_+$ being similar. 
To this end, fix $a,b,c_i, d_i, e_j, f_j \in S$ for $i = 1, \ldots, m$ and $j = 1, \ldots, n$ satisfying $a \prec \sum_{i'=1}^{m} c_{i'}$, $c_i \mathrel{R} d_i$, $\sum_{i'=1}^{m} d_{i'} \prec \sum_{j'=1}^{n} e_{j'}$, $e_j \mathrel{R} f_j$, and $\sum_{j'=1}^{n} f_{j'} \prec b$ for all $i$ and $j$. 
Applying \axiomW{4} and the left $\prec$-continuity of $R$ (twice), we obtain $c_i', c_i'', d_i', e_i', e_i'', f_i' \in S$ for $i = 1, \ldots, m$ such that $a \prec \sum_{i'=1}^{m} c_{i'}'$, $c_i' \prec c_i''$, $c_i'' \mathrel{R} d_i'$, $d_i' \prec d_i$, $\sum_{i'=1}^{m} d_{i'} \prec \sum_{j'=1}^{n} e_{j'}'$, $e_j' \prec e_j''$, $e_j'' \mathrel{R} f_j'$, and $f_j' \prec f_j$ for all $i$ and $j$. 

Combining that $S$ has almost refinement with a second usage of \axiomW{4}, we find $x_{ij}$ and $x_{ij}'$ for $i = 1, \ldots, m$ and $j = 1, \ldots, n$ such that $d_i' \prec \sum_{j' = 1}^{n} x_{ij'}'$, $x_{ij}' \prec x_{ij}$, and $\sum_{i' = 1}^{m} x_{i'j} \prec e_j'$. 

For $i = 1, \ldots, m$, since $c_i' \prec \circ \mathrel{R} \circ \prec \sum_{j' = 1}^{n} x_{ij'}'$ and $R$ has $\prec$-almost refinement, there are $y_{ij}$ and $z_{ij}$ for $i = 1, \ldots, m$ and $j = 1, \ldots, n$ such that $c_i' \prec \sum_{j' = 1}^{n} y_{ij'}$, $y_{ij} \prec \circ \mathrel{R} \circ \prec z_{ij}$, and $z_{ij} \prec x_{ij}'$. 

Next, for $j = 1, \ldots, n$, since $\sum_{i' = 1}^{m} x_{i'j} \prec \circ \mathrel{R} \circ \prec f_j$ and $R$ has $\prec$-almost refinement, there are $u_{ij}$ and $v_{ij}$ for $i = 1, \ldots, m$ and $j = 1, \ldots, n$ such that $x_{ij}' \prec u_{ij}$, $u_{ij} \prec \circ \mathrel{R} \circ \prec v_{ij}$, and $\sum_{i' = 1}^{m} v_{i'j} \prec f_j$. 

Also, for $i = 1, \ldots, m$ and $j = 1, \ldots, n$, since $y_{ij} \prec \circ\mathrel{R} \circ\prec \circ \mathrel{R} \circ \prec v_{ij}$ and $R$ is $\prec$-almost transitive, there are $w_{ij}$ and $z_{ij}$ such that $y_{ij} \prec w_{ij}$, $w_{ij} \mathrel{R} z_{ij}$, and $z_{ij} \prec v_{ij}$. 

Finally, since 
\[	
a \prec \sum_{i=1}^{m} c_{i}' \prec \sum_{i=1}^{m} \sum_{j = 1}^{n} y_{ij} \prec \sum_{i=1}^{m} \sum_{j = 1}^{n} w_{ij}
\text{ and }
\sum_{i=1}^{m} \sum_{j = 1}^{n} z_{ij} \prec \sum_{j = 1}^{n} \sum_{i=1}^{m} v_{ij} \prec \sum_{j = 1}^{n} f_j \prec b,
\]	  
we conclude that $a \prec \circ  \mathrel{R_+}  \circ  \prec b$, as desired. 
\end{proof} 
\begin{cor}
\label{cor:one-step} Let $(S, \prec)$ be a $\W$-semigroup, and let $R$ be a left $\prec$- continuous relation that has $\prec$-almost refinement and is $\prec$-almost transitive. Then, for any $a,b\in S$, we have
\[
a\leqr{R}b\iff \text{for any }c\in a^\prec, \text{ either }c\prec b\text{ or else }(c,b)\in {\prec_{\phantomtoggle{P}}}  \circ   (R+\operatorname{id})_+  \circ  {\prec_{\phantomtoggle{P}}}.
\]
\end{cor}
\begin{proof}
Use \autoref{prp:almost-refinement-transitivity}, \autoref{lma:axiomatic-merging-additive-closure} and \autoref{prp:one-step}.
\end{proof}	

\section{Dynamical $\CatW$-semigroups} \label{sec:dynamical_semigroups}
\label{sec:Dynsemigroup}

In this section, we discuss the main example of our construction of ideal-free quotient $\CatW$-semigroups, namely dynamical $\CatW$-semigroups, which are quotients associated to group actions on $\CatW$-semigroups. Throughout the section, the symbol $G$ always stands for a fixed (discrete) group. 

\begin{pgr}[$G$-actions on $\W$-semigroups]\label{dfn:actions-W-semigroups}
	Let $(S, \prec)$ be a $\W$-semigroup. A \emph{$G$-action} on $S$ is a map $\alpha \colon G\times S\to S$, denoted by $(g, a)\mapsto g a$, such that
	\beginEnumConditions
	\item If $a\prec b$ in $S$ and $g$ is an element of $G$, then $ga\prec gb$.
	\item $g(a+b)=ga+gb$, for $g\in G$ and $a,b\in S$.
	\item $ea=a$ for any $a\in S$, where $e\in G$ is the neutral element of $G$.
	\item $(gh)a=g(ha)$, for any $g,h\in G$ and any $a\in S$.
\end{enumerate}
A $\CatW$-semigroup (respectively, $\CatCu$-semigroup) $S$, together with an action $\alpha$ from a group $G$ as above will be called a $\CatGW$-semigroup (respectively, $\CatGCu$-semigroup). 
Observe that a $\CatW$-semigroup (respectively, $\CatCu$-semigroup) is just an $\{e\}\text{ \!-}\CatW$-semigroup (respectively, $\{e\}\text{ \!-}\CatCu$-semigroup) with the obvious action. 

Given $\CatGW$-semigroups $(S, \prec_S, G)$ and $(T, \prec_T, G)$, we say that a map $f \colon S \to T$ is an \emph{equivariant $\W$-morphism} if it is a $\W$-morphism and satisfies $f(g a) = g f(a)$ for any $a \in S$. 

If, in the above, the action on $T$ is trivial, then we say that $f$ is an \emph{invariant $\W$-morphism}. 
\end{pgr}
\begin{pgr}[The dynamical Cuntz semigroup]	
Let $(S,\prec, G)$ be a $\CatGW$-semigroup and let $\app$ be the \emph{orbit equivalence relation}, i.e., we write $a\app b$, for $a,b \in S$ if $a = g b$ for some $g \in G$. Notice that $\app$ is left $\prec$-continuous. Indeed, if $a,b,c\in S$ and $g\in G$ satisfy $a\prec b=gc$, then using \axiomW{1} to find $a'$ such that $a\prec a'\prec b$, we have $a\prec a'$, $a'\app g^{-1}a'$, and $g^{-1}a'\prec c$.

We write $\dyn$ for the normal closed preorder generated by $\app$ (c.f., ~\autoref{dfn:axiomatic-merging-additive}), i.e., it is the strongest preorder $\leq$ on $S$ satisfying that, for all $a, b \in S$, any of the following conditions implies $a \leq b$:
\beginEnumConditions
\item\label{dfn:actions-W-semigroups::P} $a^\prec \subseteq b^\prec$; 
\item\label{dfn:actions-W-semigroups::R} $a = g b$ for some $g \in G$; 
\item\label{dfn:actions-W-semigroups::closure} for any $c \in a^{\prec}$, we have $c \leq b$;
\item\label{dfn:actions-W-semigroups::additive} there are $a_1, a_2, b_1, b_2$ in $S$ such that $a = a_1 + a_2$, $b = b_1 + b_2$, and $a_k \leq b_k$ for $k = 1, 2$.
\end{enumerate}
Write $\alphar{G}=(\prec,\leqr{G})$, and $S/G$ for the quotient $\W$-semigroup $S  \slash  {\alphar{G}}$, which we term the \emph{dynamical $\CatW$-semigroup} associated to the action of $G$ on $S$. We write $\pi_G \colon S \to S / G$ for the canonical surjection. 

In the case of a C*-algebra $A$ together with an action of a group $G$ on $A$, we note that this action extends naturally to $\W(A)$ by setting $g\cdot[a]=[g\cdot a]$, and $\W(A)$ becomes in this fashion a $\CatGW$-semigroup, as is easy to verify. We set $\Dyn(A):=\W(A) / G$, and refer to is as the \emph{dynamical Cuntz semigroup} of $A$. 
\end{pgr}


\begin{cor}\label{cor:dynamical-subequivalence}
Let $(S, \prec, G)$ be a $\CatGW$-semigroup. 
Then for any $a, b \in S$, the following statements are equivalent: 
\beginEnumConditions
\item\label{cor:dynamical-subequivalence::axiomatic} $a \leqr{G} b$; 
\item\label{cor:dynamical-subequivalence::prec} for any $c \in a^{\prec}$,	we have $c \prec b$ or there are positive integers $m$ and $n$, elements $g_{ij}$ in $G$ and elements $d_{ij}$ in $S$, for $i = 1, \ldots, m$ and $j = 1, \ldots, n$, such that 
\[
c \prec \sum_{j=1}^{n} d_{1j} , \quad \sum_{j=1}^{n} g_{1j} d_{1j} \prec \sum_{j=1}^{n} d_{2j} , \quad \ldots \quad   \sum_{j=1}^{n} g_{m-1,j} d_{m-1,j} \prec \sum_{j=1}^{n} d_{mj} , \quad \sum_{j=1}^{n} g_{mj} d_{mj} \prec b \; . 
\]
\end{enumerate}
Moreover, if $S$ has almost refinement, we can always assume $m=1$ above, that is, the chain of relations may be replaced by
\[
c \prec \sum_{j=1}^{n} d_{1j} , \quad \sum_{j=1}^{n} g_{1j} d_{1j} \prec b \; .
\] 
\end{cor}

\begin{proof}
The equivalence between (i) and (ii) follows from~\autoref{thm:concrete-merging}, \autoref{lma:additive-closure-continuity}, \autoref{lma:axiomatic-merging-additive-closure}, and \autoref{rmk:plus_identity_relation}. 
For the last statement, assume that $S$ has almost refinement, and observe that the relation $\app$ satisfies $\prec$-almost refinement. To see this, assume that $a_i',a_i, b_j\in S$, for $i=1,\dots,n$, $j=1,\dots, m$ satisfy  $a_i'\prec a_i$,
\[
(a_1+\dots+a_n)\prec\circ\app\circ\prec (b_1+\dots+b_n).
\]
Thus, there are $c,d\in S$ and $g\in G$ such that $(a_1+\dots+a_m)\prec c$, $c=gd$, and $d\prec (b_1+\dots b_m)$. Therefore $g^{-1}a_1+\dots+g^{-1}a_m\prec d\prec (b_1+\dots+b_n)$, and the fact that $S$ has almost refinement, combined with \axiomW{4}, yields elements $y_{ij}'\prec y_{ij}$, for $i=1,\dots,m$, $j=1,\dots,n$ such that $g^{-1}a_i'\prec y_{i1}'+\dots+y_{in}'$ whereas $y_{1j}+\dots+y_{mj}\prec b_j$ for each $i, j$. Find by \axiomW{1} $y_{ij}''$ such that $y_{ij}'\prec y_{ij}''\prec y_{ij}$ and set $x_{ij}=gy_{ij}'$. Then it is easy to verify that 
\[
a_i'\prec x_{i1}+\dots+x_{in},\, x_{ij}\prec gy_{ij}''\app y_{ij}''\prec y_{ij}, \text{ and } y_{1j}+\dots y_{mj}\prec b_j,
\]
as required.

Now apply Propositions~\ref{prp:almost-refinement-transitivity},~\ref{prp:one-step}, and \autoref{lma:axiomatic-merging-additive-closure}. 
\end{proof}


\begin{cor}\label{cor:dynamical-Cuntz-subequivalence}
Let $A$ be a C*-algebra with an action by a group $G$.
Then the order on the dynamical Cuntz semigroup $\W(A)/G$ as defined in \autoref{dfn:actions-W-semigroups} is induced from the smallest preorder $\precsim$ on $M_\infty(A)_+$ such that for any $a, b \in M_\infty(A)_+$, any of the following conditions implies $a \precsim b$:
\beginEnumConditions
\item\label{cor:dynamical-Cuntz-subequivalence::Cu} $a \precsimr{\operatorname{Cu}} b$; 
\item\label{cor:dynamical-Cuntz-subequivalence::action} $a = g \cdot b$ for some $g \in G$; 
\item\label{cor:dynamical-Cuntz-subequivalence::closure} for any $c \in a^{\ll_{\operatorname{Cu}}}$, we have $c \precsim b$; 
\item\label{cor:dynamical-Cuntz-subequivalence::additive} there are $a_1, a_2, b_1, b_2$ in $M_\infty(A)_+$ such that $a = a_1 \oplus a_2$, $b = b_1 \oplus b_2$, and $a_k \precsim b_k$ for $k = 1, 2$.
\end{enumerate}
\end{cor}

\begin{proof}
This follows from the definition given in~\autoref{dfn:axiomatic-merging-additive},~\autoref{cor:precsimher-axiomatic}, and Corollary~\ref{cor:merging-associative}.  
\end{proof}

\begin{rmk}
In \autoref{cor:dynamical-Cuntz-subequivalence} (see also \autoref{rmk:precsimher-axiomatic}), we may obtain equivalent formulations of the order in $\Dyn(A)$ by the following replacements:
\begin{itemize}
\item Conditions (i) and (iii)
may be replaced by the combination of the following three conditions: 
\begin{itemize}
\item[(i')] $a \precsimher b$ as in Example~\ref{exa:precsimher};
\item[(ii')] there exists $x \in M_\infty(A)$ such that $a = x b x^*$.
\item[(iii')] for any $c \in a^{\llher}$, we have $c \precsim b$; 
\end{itemize}
\item Condition~(ii') may further be replaced by 
\begin{itemize}
\item[(ii'')] there exists $x \in M_\infty(A)$ such that $a = x x^*$ and $b = x^* x$.
\end{itemize}
\item Condition (iii)
may also be replaced by the following: 
\begin{itemize}
\item[(iii'')] there is an increasing sequence $([a_n])_{n \in \mathbb{N}}$ in $\W(A)$ whose supremum $\sup_n^{\precsimr{\operatorname{Cu}}} [a_n]$ with regard to $\precsimr{\operatorname{Cu}}$ exists and is equal to $[a]$, and $a_n \precsim b$ for every $n \in \mathbb{N}$.
\end{itemize}
\end{itemize}
\end{rmk}

The (dynamical) $\CatW$-semigroup $S / G$ enjoys the following universal property. 

\begin{thm}\label{thm:dynamical-universal-property}
Let $(S, \prec_S, G)$ be a $\CatGW$-semigroup and let $(T, \prec_T)$ be a $\W$-semigroup. Then for any invariant $\W$-morphism $f \colon S \to T$, there is a unique order-preserving $\W$-morphism $f_G \colon S / G \to T/{\alpha_T}$ such that $\pi_{\alpha_T}\circ f = f_G \circ \pi_G$, i.e., the following diagram commutes:  
\[
\xymatrix{
S \ar[d]_{\pi_G} \ar[r]^{f} & T\ar[d]^{\pi_{\alpha_T}} \\
S  /G \ar@{-->}[r]_{f_G} & T/{\alpha_T}
}
\]
Moreover, the pair $(S / G, \pi_G)$ is the unique pair of a $\W$-semigroup and a $\W$-morphism that satisfies the above universal property. 
\end{thm}

\begin{proof}
This follows from Theorem~\ref{thm:quotient-universal-property}. 
\end{proof}

\section{A categorical interlude} \label{sec:categorical}

For many purposes, it is more convenient to work with $\CatCu$-semigroups than $\CatW$-semigroups. Fortunately, \cite{AntPerThi14arX:TensorProdCu} provides a way to take a completion of a $\CatW$-semigroup and obtain a $\CatCu$-semigroup. In this section, we study how this $\CatCu$-completion behaves in relation with dynamical $\CatW$-semigroups. For this purpose, it is convenient to formulate things in a categorical language.

\begin{pgr}[Reflective subcategories]
	\label{rmk:categories-reflective}
	In the following, recall that a subcategory is \emph{full} if for any two objects in the subcategory, all morphisms between them in the ambient category are also contained in the subcategory. A subcategory is \emph{reflective} if the inclusion functor has a left adjoint, called the \emph{reflector}. 
	
	t follows from the theory of adjoint functors that 
	a subca\-tegory is \emph{reflective} if for every object $Y$ in the ambient category, there is a \emph{universal morphism} from $Y$ to the inclusion functor, 
	i.e., a unique pair $\left( \overline{Y}, \pi \right)$ of an object $\overline{Y}$ in the subcategory and a morphism $\pi \colon Y \to \overline{Y}$ in the ambient category with the following universal property: for any object $X$ in the subcategory and every morphism $f \colon Y \to X$ in the ambient category, there exists a unique morphism $\overline{f} \colon \overline{Y} \to X$ in the subcategory with $\overline{f} \circ \pi = f$, as indicated in the following commutative diagram: 
	\[
	\xymatrix{
		Y \ar[d]_{\pi} \ar[rd]^{f} &  \\
		\overline{Y} \ar@{-->}[r]_{\overline{f}} & X
	}
	\]
	where the bottom row is required to be in the subcategory. In this case, the reflector is a functor that sends every object $Y$ in the ambient category to $\overline{Y}$ in the subcategory, and every morphism $f \colon Y \to Y'$ in the ambient category to a morphism $\overline{f} \colon \overline{Y} \to \overline{Y'}$ in the subcategory determined by applying the above universal property to the composition $\pi \circ f$ as in the following commutative diagram: 
	\[
	\xymatrix{
		Y \ar[d]_{\pi} \ar[rd] \ar[r]^{f} & Y' \ar[d]_{\pi}  \\
		\overline{Y} \ar@{-->}[r]_{\overline{f}} & \overline{Y'}
	}
	\]
\end{pgr}


\begin{lma}\label{lma:categories-reflective-properties}
	Let $\mathcal{C}$ be a category, $\mathcal{D}$ a subcategory of $\mathcal{C}$, and $\mathcal{E}$ a subcategory of $\mathcal{D}$. Then the following holds: 
	\beginEnumConditions
	\item\label{lma:categories-reflective-properties::composition} If $\varphi \colon \mathcal{C} \to \mathcal{D}$ is a reflector of the inclusion $\mathcal{D} \hookrightarrow \mathcal{C}$ and $\psi \colon \mathcal{D} \to \mathcal{E}$ is a reflector of the inclusion $\mathcal{E} \hookrightarrow \mathcal{D}$, then $\varphi \circ \psi \colon \mathcal{C} \to \mathcal{E}$ is a reflector of the inclusion $\mathcal{E} \hookrightarrow \mathcal{C}$. 
	\item\label{lma:categories-reflective-properties::restriction} If $\tau \colon \mathcal{C} \to \mathcal{E}$ is a reflector of the inclusion $\mathcal{E} \hookrightarrow \mathcal{C}$, then the restriction $\tau \big|_{\mathcal{D}} \colon \mathcal{D} \to \mathcal{E}$ is a reflector  of the inclusion $\mathcal{E} \hookrightarrow \mathcal{D}$. 
	\item\label{lma:categories-reflective-properties::uniqueness} If $\theta, \zeta \colon \mathcal{C} \to \mathcal{D}$ are both reflectors of the inclusion $\mathcal{D} \hookrightarrow \mathcal{C}$, then there is a natural isomorphism between $\theta$ and $\zeta$. 
\end{enumerate}
\end{lma}

\begin{proof}
These follow directly from the comments in \autoref{rmk:categories-reflective}. 
\end{proof}


Recall that the category $\CatW$ consists of, as objects, all $\CatW$-semigroups and, as morphisms, all $\CatW$-morphisms. It has a subcategory $\CatCu$ consisting of, as objects, all $\CatCu$-semigroups and, as morphisms, all $\CatCu$-morphisms. The exact relationship between $\CatW$ and $\CatCu$ was obtained in \cite{AntPerThi14arX:TensorProdCu}. We recall it here by completeness.
\begin{thm}[{\cite[{3.1.5}]{AntPerThi14arX:TensorProdCu}}]\label{thm:categories-Cu-full-reflective-in-W}
The category $\CatCu$ is a full reflective subcategory of $\CatW$. 
\end{thm}


The reflector $\gamma \colon \CatW \to \CatCu$ was termed the \emph{$\CatCu$-completion functor} (see \cite{AntPerThi14arX:TensorProdCu}, and also \cite{AntBosPer11CompletionsCu}). Now we add group actions into the picture. 

\begin{pgr}[The categories $\CatGW$ and $\CatGCu$]
\label{rmk:categories-W-Cu-G-categorical}
Let $G$ be a group. We let $\CatGW$ be the category that consists of, as objects, all $\CatGW$-semigroups and, as morphisms, all $\rm G$-equivariant $\CatW$-morphisms. It has a subcategory $\CatGCu$ consisting of, as objects, all $\CatGCu$-semigroups and, as morphisms, all $\rm G$-equivariant $\CatCu$-morphisms. 
We naturally view $\CatW$ (respectively, $\CatCu$) as a subcategory of $\CatGW$ (respectively, $\CatGCu$) by equipping every $\CatW$-semigroup (respectively, $\CatCu$-semigroup) the trivial $G$-action. 

Note that the category $\CatGW$ can be constructed from $\CatW$ by means of general category theory. Indeed, an object in $\CatGW$ is nothing but a tuple $(S, \eta)$, where $S$ is an object in $\CatW$ and $\eta = \left( \eta_g \right)_{g \in G}$ is a homomorphism from $G$ to the group of automorphisms of $S$ in $\CatW$, while a morphism from $(S, \eta)$ to $(T, \mu)$ in $\CatGW$ is nothing but a morphism $f$ from $S$ to $T$ in $\CatW$ such that $f \circ \eta_g = \mu_g \circ f$ for any $g \in G$, that is, an equivariant $\W$-morphism.

With this point of view, we have a functor $\omega_G \colon \CatGW \to \CatW$ that sends every object $(S,\eta)$ in $\CatGW$ to $S/G$ (where the quotient is taken by the action induced by $\eta$). For a morphism $f$ in $\CatGW$,  write $\omega_G(f)=f_G$, as in \autoref{thm:dynamical-universal-property}. It follows from the universal property that for any morphism $f \colon (S, \eta) \to (T, \mu)$ in $\CatGW$, the morphism $f_G \colon S /G\to T / G$ in $\CatW$ is the unique one that makes the following diagram commute: 
\[
\xymatrix{
	S \ar[d]_{\pi_G} \ar[r]^{f} & T \ar[d]^{\pi_G} \\
	S / G \ar@{-->}[r]_{f_G} &  T/G
}
\]
\end{pgr}
\begin{thm}\label{thm:categories-Cu-G-full-reflective-in-W-G}
The category $\CatGCu$ is a full reflective subcategory of $\CatGW$, with reflector $\gamma_G\colon \CatGW\to\CatGCu$ referred as to the $\CatGCu$\emph{-completion functor}. 
\end{thm}

\begin{proof}
This follows from \autoref{thm:categories-Cu-full-reflective-in-W} and \autoref{rmk:categories-W-Cu-G-categorical}, where the reflector $\gamma_G \colon \CatGW \to \CatGCu$ sends every object $(S, \eta)$ in $\CatGW$ to $\left( \gamma(S), \left( \gamma \left(\eta_g\right) \right)_{g \in G} \right)$ in $\CatGCu$ and sends every morphism $f$ in $\CatGW$, viewed as a morphism in $\CatW$, to $\gamma(f)$ in $\CatGCu$. 
\end{proof}

\begin{lma}\label{lma:categories-W-full-reflective-in-W-G}
For any object $(S, \eta)$ in $\CatGW$, the pair $(S / G, \pi_G)$ is a universal morphism from $(S, \eta)$ to the inclusion functor $\CatW \to \CatGW$.  
\end{lma}

\begin{proof}
This follows from \autoref{thm:dynamical-universal-property} and \autoref{rmk:categories-reflective}. 
\end{proof}


\begin{thm}\label{thm:categories-W-full-reflective-in-W-G}
The category $\CatW$ is a full reflective subcategory of $\CatGW$, with the orbit reflector being a reflector. 
\end{thm}

\begin{proof}
Fullness follows from the fact that every $\CatW$-morphism is equivariant when the domain and the codomain are equipped with the trivial action. Reflectivity via the orbit functor follows from \autoref{lma:categories-W-full-reflective-in-W-G} and \autoref{rmk:categories-reflective}.  
\end{proof}


\begin{cor}\label{cor:categories-Cu-full-reflective-in-W-G}
The following statements are true: 
\beginEnumConditions
\item\label{cor:categories-Cu-full-reflective-in-W-G::composition} The category $\CatCu$ is a full reflective subcategory of $\CatGW$, with the composition $\gamma \circ \omega_G$ being a reflector. 
\item\label{cor:categories-Cu-full-reflective-in-W-G::restriction} The category $\CatCu$ is a full reflective subcategory of $\CatGCu$, with the composition $\gamma \circ \left( \omega_G \big|_{\CatGCu} \right)$ being a reflector. 
\item\label{cor:categories-Cu-full-reflective-in-W-G::uniqueness} There is a natural isomorphism between the functors $\gamma \circ \left( \omega_G \big|_{\CatGCu} \right) \circ \gamma_G$ and $\gamma \circ \omega_G$ from $\CatGW$ to $\CatCu$, i.e., the following diagram commutes up to a natural isomorphism. 
\[
\xymatrix@R-20pt{
	\CatGW \ar[dd]_{\gamma_G} \ar[r]^{\omega_G}  & \CatW \ar[dr]^{\gamma} \\
	& & \CatCu \\
	\CatGCu \ar[r]_{\omega_G} & \CatW \ar[ur]_{\gamma} \\
}
\]
\end{enumerate}
\end{cor}

\begin{proof}
Consider the commutative diagram of inclusions of categories
\[
\xymatrix{
\CatCu \ar@{^{(}->}[r] \ar@{^{(}->}[d] & \CatW  \ar@{^{(}->}[d] \\ 
\CatGCu \ar@{^{(}->}[r] & \CatGW 
}
\]
Then \ref{cor:categories-Cu-full-reflective-in-W-G::composition} follows from applying \autoref{lma:categories-reflective-properties}\ref{lma:categories-reflective-properties::composition} to the inclusions $\CatCu \hookrightarrow \CatW \hookrightarrow \CatGW$. Furthermore, by \autoref{lma:categories-reflective-properties}\ref{lma:categories-reflective-properties::restriction}, the restriction $\left( \gamma \circ \omega_G \right)\big|_{\CatGCu} = \gamma \circ \left( \omega_G \big|_{\CatGCu} \right)$ is a reflector of the inclusion $\CatCu \hookrightarrow \CatGCu$, i.e., \ref{cor:categories-Cu-full-reflective-in-W-G::restriction} holds. It then follows by  \autoref{lma:categories-reflective-properties}\ref{lma:categories-reflective-properties::composition} again that $\gamma \circ \left( \omega_G \big|_{\CatGCu} \right) \circ \gamma_G$ is also a reflector of the inclusion $\CatCu \hookrightarrow \CatGW$, which implies \ref{cor:categories-Cu-full-reflective-in-W-G::uniqueness} by \autoref{lma:categories-reflective-properties}\ref{lma:categories-reflective-properties::uniqueness}. 
\end{proof}


\begin{pgr}[The complete dynamical Cuntz semigroup]
\label{dfn:complete-dynamical-Cuntz}
Let $(S,\eta)$ be a $\CatGW$-semigroup. We refer to the $\CatCu$-semigroup  $\gamma \circ \omega_G(S,\eta)$ as the \emph{dynamical completion of $(S,\eta)$}. We have, by \autoref{cor:categories-Cu-full-reflective-in-W-G}\ref{cor:categories-Cu-full-reflective-in-W-G::uniqueness}, a natural isomorphism
\[
\gamma(S/{\eta})=\gamma \circ \omega_G(S,\eta)\cong \gamma\circ\omega_G\circ\gamma_G(S,\eta)=\gamma(\gamma_G(S,\eta)/G).
\]  
We remark that we do not have at present an example of a $\CatGCu$-semigroup $(S,\eta)$ such that $S/G$ is not a $\CatCu$-semigroup.	

If $A$ is a G-\ca{}, we shall write $\CuDyn(A):= \gamma (\CatW(A) /{G})$, and call it the \emph{complete dynamical Cuntz semigroup}.
\end{pgr}


We have the following natural identification:
\begin{cor}\label{cor:CuWidentifications}
Let $A$ be a $\mathrm{G}$-\ca. Then there is a natural isomorphism
\[
\CuDyn(A)=\gamma (\W(A)/{G})\cong\gamma(\Cu(A)/{G}).
\]
\end{cor}
\begin{proof}
We know from \autoref{dfn:complete-dynamical-Cuntz} that $\Cu_G(A)\cong \gamma (\gamma_G(\W(A))/{G})$. Now, $\gamma_G(\W(A))$ is just $\gamma(\W(A))$ with the induced action, according to \autoref{thm:categories-Cu-G-full-reflective-in-W-G}. As shown in \cite[Theorem 3.2.8]{AntPerThi14arX:TensorProdCu}, we have that $\gamma(\W(A))$ is naturally isomorphic to $\Cu(A)$.
\end{proof}

\section{Invariant closed ideals} \label{sec:invariant_ideals}
In this section we revisit the notion of ideal in the light of an action by a group $G$. We identify the $G$-invariant ideals of a $\CatGW$-semigroup with ideals those of its corresponding dynamical quotient. Using this, we conclude that, for a C*-dynamical system $(A,G)$ and a closed two-sided $G$-invariant ideal $I$, we have $\Cu_G(A)/\Cu_G(I)\cong\Cu_G(A/I)$; see \autoref{thm:dynquotients}.

\begin{pgr}[Simplicity and invariant ideals]
	Recall that, for a $\W$-semigroup $S$, we denote by $\Lat_\CatW(S)$ the lattice of closed ideals. For any element $a\in S$, put
	\[
	I(a):=\{b\in S\colon \text{for any }b'\prec b\text{ there is }n\in \N \text{ such that } b'\prec na\},
	\]
	which is the smallest closed ideal of $S$ that contains $a$.
	
	A non-zero element $a$ in a $\CatW$-semigroup $S$ is an \emph{order-unit} provided for any pair of elements $b,b'\in S$, with $b'\prec b$, there is $n\in\N$ such that $b'\prec na$. In other words, $a$ is an order-unit precisely when $I(a)=S$.
	
	If $S$ has no non-trivial closed ideals, then we say that $S$ is \emph{simple}. With this terminology and our observations above, the fact that $S$ is simple is  equivalent to the demand that every non-zero element is an order-unit. 
	
	Recall from \autoref{pgr:idealspairs} that, given a closed ideal $I$, we set $S/I=S/{\alpha_I}$, where $\alpha_I=(\prec,\leq_I)$ and $a\leq_I b$ provided that, whenever $x\prec a$, there  is $y\in I$ such that $x\prec b+y$. We also have shown that $S/I$ is a $\W$-semigroup and that the natural quotient map $\pi\colon S\to S/I$ is a $\W$-morphism. Recall that the relation $\prec_I$ in $S/I$ making it a $\W$-semigroup is given by $\pi(a)\prec_I\pi(b)$ whenever $a\mathrel{{\leq_I}\circ{\prec}}b$
	
	If $S$ is a $\CatGW$-semigroup, a closed ideal $I$ of $S$ is said to be a \emph{$\CatGW$-ideal} if, furthermore, $ga\in I$ whenever $g$ is an element of $G$ and $a$ is an element of $I$. (Equivalently, $GI=I$.) We shall denote the lattice of $\CatGW$-ideals of $S$ by $\Lat_{\CatW}^G(S)$. We will say that the action of $G$ on $S$ is \emph{minimal} in case $S$ does not have any non-trivial $\CatGW$-ideals. We show below that this notion agrees with the one given by Rainone in \cite[Definition 3.4]{Rai16PLonMathSoc}.
	Now, assume that $S$ is a $\CatGW$-semigroup $S$ and $a\in S$. Let us define
	\[
	I_G(a):=\{z\in S\mid \text{ for any }z'\prec z\text{ there exist }g_1,\dots, g_n\in G \text{ with }z'\prec \sum_{i=1}^ng_ia\}\,.
	\]
	It is easy to verify that $I_G(a)$ is the smallest $\CatGW$-ideal of $S$ that contains $a$. Note that, if $G$ is acting trivially on $S$, then $I_G(a)=I(a)$. In this language, we have that the action of $G$ is minimal if and only if $I_G(a)=S$ for any nonzero $a\in S$.
\end{pgr}
\begin{pgr}[Completions]
	\label{pgr:completion}
	We have already mentioned in \autoref{thm:categories-Cu-full-reflective-in-W} that $\Cu$ is a full, reflective subcategory of $\W$, with a reflector functor $\gamma\colon \W\to \Cu$, which we briefly recall as it will be used below.
	
	Let $(S,\prec)$ be a $\W$-semigroup. Consider $\bar{S}=\{(a_n)\colon a_n\prec a_{n+1} \text{ for all }n\}$, and define $(a_n)\precsim (b_n)$ if, and only if, for each $n$ there is $m$ such that $a_n\prec b_m$. Put $\gamma(S)=\bar{S}/{\sim}$, where $\sim$ is the antisymmetrization of $\precsim$. It was shown in \cite[Proposition 3.1.6]{AntPerThi14arX:TensorProdCu} that $\gamma(S)$ is a $\Cu$-semigroup with order $\leq$ induced by the relation $\precsim$. The way-below relation is given as follows: $[(a_n)]\ll [(b_n)]$ if and only if there is $m$ such that $a_n\prec b_m$ for all $n$. 
	
	Given $a\in S$, choose by \axiomW{1} a countable cofinal subset $(a_n)$ in $a^\prec$. This defines a $\W$-semigroup map $\gamma\colon S\to \gamma(S)$ by $\gamma(a)=[(a_n)]$. This map is an order-embedding in the sense that $a\prec b$ if and only if $\gamma(a)\ll\gamma(b)$, and is dense in the sense that, if $b'\ll b$ in $\gamma(S)$, then there is $a\in S$ with $b'\leq \gamma(a)\leq b$.
	
	If $S$ is a $\Cu$-semigroup,we shall denote the lattice of $\CatCu$-ideals by $\Lat_\CatCu(S)$, and if $G$ is a group that acts on $S$, then we denote the lattice of $\CatGCu$-ideals by $\Lat_\CatCu^G(S)$. 	
\end{pgr}	


We will need the following lemma:

\begin{lma}
	\label{lma:lifting} Let $(S,\prec)$ be a $\W$-semigroup and let $I$ be a closed ideal of $S$. Let $a,b,c,d\in S$ be such that $a\prec b$, $c\prec d$, and $\pi(b)\prec \pi (c)$, where $\pi\colon S\to S/I$ is the quotient map. Then, there are elements $b',c', d'\in S$ with $a\prec b'\prec c'\prec d'$ and $\pi(c')=\pi(c), \pi(d')=\pi(d)$
\end{lma}	
\begin{proof}
	By assumption, $\pi(b)\prec\pi(c)$, hence we have $b\leq_I b_0$ with $b_0\prec c$.  Since $a\prec b$, there is by definition of $\leq_I$ an element $y\in I$ such that $a\prec b_0+y$. Using \axiomW{4}, we find elements $b_0'\prec b_0$ and $y'\prec y$ such that $a\prec b_0'+y'$. Also, using \axiomW{1}, choose $y''\in S$ with $y'\prec y''\prec y$.
	
	Since $b_0'\prec b_0\prec c$ and $y'\prec y''$, we have by \axiomW{3} that $b_0'+y'\prec c+y''$. Now let $b'=b_0'+y',c'=c+y''$, and $d'=d+y$.
\end{proof}
\begin{prp}
	\label{prp:WidealsvsCuideals} Let $S$ be a $\CatW$-semigroup. The completion functor $\gamma$ induces an ideal lattice isomorphism
	\[
	\gamma\colon\Lat_\CatW(S)\to\Lat_\CatCu(\gamma(S))\text{ given by } I\mapsto\gamma(I).
	\]
	This isomorphism respects quotients in the sense that, for any $I\in\Lat_\W(S)$ we have a $\CatCu$-isomorphism $\gamma(S)/\gamma(I)\cong\gamma(S/I)$.
	
	Furthermore, if $G$ acts on $S$ (and therefore also on $\gamma(S)$; see \autoref{cor:categories-Cu-full-reflective-in-W-G}), this isomorphism restricts to an isomorphism
	\[
	\Lat_\CatW^G(S)\cong\Lat_\CatCu^G(\gamma(S))
	\]
	that also respects quotients.
\end{prp}
\begin{proof}
	Let $I$ be a $\CatW$-ideal. Notice first that an element $[(a_n)]$ in $\gamma(S)$ belongs to $\gamma(I)$ if and only if $a_n$ belongs to $I$ for all $n\in\N$. Using this, it is easy to conclude that $\gamma(I)$ is a $\CatCu$-ideal of $\gamma(S)$. Observe also that, in case $I$ is a $\CatGW$-ideal, its image $\gamma(I)$ is clearly also a $\CatGCu$-ideal.
	
	For $\CatW$-ideals $I$ and $J$ in $S$, we have $\gamma(I)\subseteq\gamma(J)$ if and only if $I\subseteq J$. For, assume that $\gamma(I)\subseteq\gamma(J)$ and $a$ is an element in $I$. Choose by \axiomW{1} a countable cofinal $\prec$-increasing sequence $(a_n)$ in $a^\prec$. Then $[(a_n)]$ belongs to $\gamma(I)$, and thus by assumption $(a_n)\subseteq J$. Since $J$ is a closed ideal, we have $a\in J$, and thus $I\subseteq J$.
	
	Next, let $K$ be an ideal of $\gamma(S)$. Define $I:=\{a\in S\colon \gamma(a)\in K\}$, where $\gamma\colon S\to\gamma(S)$ is the natural map. We claim that $I$ is a closed $\CatW$-ideal of $S$ and $\gamma(I)=K$. By definition, $\gamma(I)\subseteq K$, and it is also clear that $I$ is a $\CatW$-subsemigroup. If $a\prec b$ in $S$ and $b\in I$, then $\gamma(a)\ll\gamma (b)$ and thus $\gamma(a)$ belongs to $K$, that is, $a$ belongs to $I$. Suppose that $a^\prec\subseteq I$ for some $a\in S$. Let $x\in \gamma(S)$ be such that $x\ll\gamma(a)$. Then, there is $a'\prec a$ such that $x\ll\gamma(a')\ll\gamma(a)$. Since $a'\in I$, we have that $x\in K$. Thus $\gamma(a)^\ll\subseteq K$, and this implies that $a\in I$.
	
	Any element in $K$, viewed in $\gamma(S)$, is a supremum of an increasing sequence of elements of the form $\gamma(a)$, for $a\in S$. Hence, all such elements belong to $K$, that is, $K\subseteq\gamma(I)$, and consequently $K=\gamma(I)$.
	
	Let us finally check that the functor $\gamma$ respects quotients. Let $I\in\Lat_\W(S)$. Since $\gamma$ is a functor, we have a $\CatCu$-morphism
	\[
	\gamma(S)\to\gamma(S/I),\,\, [(a_n)]\mapsto [(\pi(a_n))],
	\]
	where $\pi\colon S\to S/I$ is the quotient map.
	
	We claim it is surjective. Let $x=[(\pi(a_n))]\in\gamma(S/I)$ be any element. Since $\pi(a_1)\prec \pi(a_2)$, there is an element $a_2'$ such that $a_2'\prec a_2$ and $\pi(a_1)\prec\pi(a_2')\prec \pi(a_2)$. Also, there is an element $a_3'\in S$ with $a_3'\prec a_3$, and such that $\pi(a_2)\prec \pi(a_3')\prec  \pi(a_3)$. Thus, for each $n$ we find elements $a_{n}'\prec a_{n}$ with $\pi(a_{n-1})\prec \pi(a_{n}')\prec \pi(a_{n})$. Clearly, the sequence 
	\[
	(\pi(a_2'),\pi(a_2),\pi(a_3'),\pi (a_3),\pi(a_4'),\pi(a_4),\dots)
	\]
	also represents $x$. Therefore, changing notation, we may assume at the outset that $a_{2n-1}\prec a_{2n}$ for each $n\geq 1$.
	
	Now apply \autoref{lma:lifting} to $a_1,\dots,a_4$ to find elements $b_1,b_2,b_3,b_4'$ with $b_1=a_1,b_1\prec b_2\prec b_3\prec b_4'$, and $\pi(b_3)=\pi(a_3), \pi(b_4')=\pi(a_4)$. A second application of \autoref{lma:lifting}, to $b_3,b_4',a_5,a_6$ yields elements $b_4, b_5, b_6'$ with $b_3\prec b_4\prec b_5\prec b_6'$, and $\pi(b_5)=\pi(a_5),\pi(b_6')=\pi(a_6)$. Continuing in this way, we get a $\prec$-increasing sequence $(b_n)$ such that $[(\pi(b_n))]=[(\pi(a_n))]=x$, hence the map $\gamma(S)\to\gamma(S)/\gamma(I)$ is surjective, as claimed.
	
	Finally, to prove that $\gamma(S)/\gamma(I)\cong\gamma(S/I)$, we need to verify that $[(a_n)]\leq_{\gamma(I)}[(b_n)]$ in $\gamma(S)$ if, and only if, $[(\pi(a_n))]\leq[(\pi(b_n))]$ in $\gamma(S/I)$.
	
	First assume that  $[(a_n)]\leq_{\gamma(I)}[(b_n)]$. Then, there is $[(c_n)]\in\gamma(I)$ such that $[(a_n)]\leq [(b_n)]+[(c_n)]$ in $\gamma(S)$. Applying $\pi$ we obtain $[(\pi(a_n))]\leq [(\pi(b_n))]$. 
	
	Conversely, by a standard argument we may assume that $[(\pi(a_n))]\ll [(\pi(b_n))]$, and we need to show that $[(a_n)]\leq_{\gamma(I)} [(b_n)]$. In this case, there is $m$ such that $\pi(a_n)\prec\pi(b_m)$ for all $n$. Since, for a given $n$, we have $a_n\prec a_{n+1}$, there is $c_n\in I$ such that $a_n\prec b_m+c_n$. Using \axiomW{1}, choose a $\prec$-increasing sequence $(a_{n,i})$ in $a_n^\prec$ such that $a_{n-1}\prec a_{n,1}$.
	As $a_{n,1}\prec a_n$, there is $c_{n,1}\prec c_n$ with $a_{n,1}\prec b_m+c_{n,1}$. Likewise, there is $c_{n,2}'\prec c_n$ with $a_{n,2}\prec b_m+c_{n,2}'$. By using \axiomW{1}, we may choose $c_{n,2}\prec c_n$ such that $c_{n,1},c_{n,2}'\prec c_{n,2}$. We thus find a $\prec$-increasing sequence $(c_{n,k})$ in $I$ with $a_{n,k}\prec b_m+c_{n,k}$ for all $k$. Set $x_n=[(c_{n,k})_k]\in \gamma(I)$, and $x=\sum_{n=1}^{\infty}x_n\in\gamma(I)$.
	
	It was shown in \cite[Proposition 3.1.6]{AntPerThi14arX:TensorProdCu} that $[(a_k)_k]=\sup [a^{(n)}]$, where $a^{(n)}=(a_1,\dots,a_{n-1},a_{n,1},a_{n,2},\dots)$. By our construction above, we have 
	\[
	[a^{(n)}]\leq [(b_k)_k]+[(c_{n,k})_k]\leq [(b_k)_k]+x,
	\]
	and this implies that $[(a_k)]\leq [(b_k)]+x$, as desired.
\end{proof}
It is clear from the above that the action of a group $G$ on a $\CatW$-semigroup $S$ is minimal if and only if that is the case for $\gamma(S)$.

\begin{prp}
	\label{prp:Gideals} 
	Let $G$ be a group, let $(S,\prec)$ be a $\CatGW$-semigroup, and denote by $\pi_G\colon S\to S/G$ the natural map. Then
	\beginEnumConditions
	\item If $I$ is a closed $\CatGW$-ideal $I$ of $S$, then $\pi_G(I)=I/G$ is naturally identified with a closed ideal of $S/G$.
	\item  $\pi_G$ induces a lattice isomorphism $\Lat_\W^G(S)\cong \Lat_\W(S/G)$. In particular, the action of $G$ is minimal if and only if $S/G$ is simple.
	\item For $I$ as in (i), the quotient map $\pi\colon S\to S/I$ induces an isomorphism
	\[
	(S/G)/(I/G)\cong (S/I)/G.
	\]
\end{enumerate}
\end{prp}
\begin{proof}
(i): We clearly have a well defined map $I/G\to S/G$ given by $[a]\mapsto [a]$. We need to show that $[a]\prec [b]$ in $S/G$ with $b\in I$ implies $a\in I$ and that $[a]^\prec\subseteq I/G$ implies $a\in I$.

If $[a]\prec [b]$ with $b\in I$, then by definition $a\leqr{G} b'\prec b$ for some $b'\in I$. Let $a'\prec a$. Then, either $a'\prec b'$, in which case $a'\in I$, or else there are $n,m$, and elements $x_0,x_0'\dots,x_n, x_n'\in S$ with $x_i=\sum_{j=1}^m d_{ij}, x_i'=\sum_{j=1}^{m}g_{ij}d_{ij}$, for some elements $g_{ij}\in G$, and such that $x_i'\prec x_{i+1}$ for each $i<n$, $a'\prec x_0$, and $x_n'\prec b'$; see \autoref{cor:dynamical-subequivalence}. Since $b'\in I$, we have that $x_n'\in I$, hence $g_{nj}\in I$ for all $j$. Using that $I$ is a $\CatGW$-ideal, we get $d_{nj}\in I$ for all $j$, whence $x_n\in I$. Thus by a recursive argument we arrive at $a'\in I$. This shows that $a^\prec\subseteq I$, hence $a\in I$ since $I$ is closed.

Now suppose that $[a]^\prec\subseteq I/G$. If $a'\prec a$, then $[a']\prec [a]$, hence $[a']=[b]$ for $b\in I$, and thus in particular $a'\leqr{G}b'$, for some $b'\in I$. As in the paragraph above, this implies that $a'\in I$, whence $a\in I$.

(ii): Using (i), we have a map $\Lat_\W^G(S)\to \Lat_\W(S/G)$ given by $I\mapsto \pi_G(I)$. The inverse is given by $K\mapsto\pi_G^{-1}(K)$, for any closed ideal $K$ of $S/G$. It is routine to verify that these assignments are inverse for one another.

(iii): By functoriality applied to the natural map $\pi_I\colon S\to S/I$ (see \autoref{rmk:categories-W-Cu-G-categorical}), we have a $\CatW$-morphism $(\pi_I)_G\colon S/G\to(S/I)/G$, given by $(\pi_I)_G([a])=[\pi_I(a)]$, which is clearly surjective. Put $J=I/G$. We claim that, for $x,y\in S/G$, we have $x\leq_J y$ if and only if $(\pi_I)_G(x)\leq(\pi_I)_G(y)$ in $(S/I)/G$. This clearly implies that $(\pi_I)_G$ induces an order-isomorphism. 

Write $x=[a]$, $y=[b]$ for $a,b\in S$, and assume first that $[a]\leq_J [b]$. Let $z\prec [\pi_I(a)]$ in $(S/I)/G$. Then, there are elements $a', a''\in S$ such that $a'\prec a''\prec a$ and $z\prec [\pi_I(a')]$.

Since also $[a'']\prec [a]$ in $S/G$ and $[a]\leq_J [b]$, there is an element $b'\in S$ with $b'\prec b$, and an element $c\in I$ such that $[a'']\leq [b']+[c]$ in $S/G$, that is, $a''\leqr{G} b'+c$. Arguing as in the paragraph above, and using \autoref{cor:dynamical-subequivalence}, since $a'\prec a''$, we have that either $a'\prec b'+c$, and thus $\pi_I(a')\prec \pi_I(b)$ or else there are $n,m$ and elements, denoted again by $x_0,x_0'\dots,x_n, x_n'\in S$, with $x_i=\sum_{j=1}^m d_{ij}, x_i'=\sum_{j=1}^{m}g_{ij}d_{ij}$, for some elements $g_{ij}\in G$, and such that $x_i'\prec x_{i+1}$ for each $i<n$, $a'\prec x_0$, while $x_n\prec b'$. This clearly implies that $\pi_I(a)\leqr{G}\pi_I(b)$, hence $(\pi_I)_G([a])\leq (\pi_I)_G([b])$ in $(S/I)/G$.

Conversely, suppose that $\pi_I(a)\leqr{G}\pi_I(b)$, and let 
$z\prec [a]$ in $S/G$. Choose elements $a', a''\in S$ such that $a'\prec a''\prec a$. We proceed as above, using \autoref{cor:dynamical-subequivalence} and assuming $m=1$ (the general case follows applying induction on $m$), hence we assume that there are elements $\pi_I(d_1),\dots,\pi_I(d_n)\in S/I$ and $g_1,\dots,g_n\in G$ such that $\pi_I(a'')\prec\sum_j \pi_I(d_j)=\pi_I(\sum_j d_j)$, and $\sum_jg_j\pi_I(d_i)=\pi_I(\sum_j g_jd_i)\prec \pi(b)$. In particular this implies that, in $S$, we have $a''\leq_I\sum_j d_j'$ for some $d_j'\prec d_j$ and $\sum_jg_jd_j\leq_I b'$, for $b'\prec b$.

Now, by definition of $\leq_I$, since $a'\prec a''$ and abusing notation, there are elements $x'\prec x$ in $I$ with
\[
a'\prec\sum_j d_j'+x'.
\]
Choose $x'', x'''$ with $x'\prec x''\prec x'''\prec x$. Choose $d_j''$ such that $\sum_j g_jd_j'\prec \sum_j g_jd_j''\prec\sum_j g_jd_j$. Since $\sum_j g_jd_j''\prec\sum_j g_jd_j$, there is $b''\in S$ with $b''\prec b'$ and elements $y',y\in I$ with $y'\prec y$ such that $\sum_j g_jd_j''\prec b''+y'$. Let $c=x+y$, which is an element in $I$. 

We claim that $a'\leqr{G} b'+c$. Indeed, if $\tilde{a}\prec a'$, then 
\[
\tilde{a}\prec \sum_{j=1}^n d_j'+x'.
\] 
Also, with $g_{n+1}=e\in G$, 
\[
\sum_j g_jd_i'+g_{n+1}x'\prec\sum_j g_jd_j''+x''\prec b''+x'''+y'\prec b'+c.
\]
Therefore $[a']\leq [b']+[c]$ in $S/G$, and consequently $[a']\leq_J[b]$. This implies, as $a'\prec a$ is arbitrary, that $[a]\leq_J[b]$.
\end{proof}
\begin{cor}
\label{lma:ppalideal} Let $S$ be a $\CatGW$-semigroup, and let $\pi_G\colon S\to S / G$ denote the natural map. For any $a\in S$, we have
\[
\pi_G(I_G(a))=I([a]).
\]
\end{cor}
\begin{proof}
Let $b$ be an element in $I_G(a)$ and let $x\in S / G$ be such that $x\prec [b]$. Then there is $b'\in S$ with $b'\prec b$ and $x\prec [b']$. By definition of $I_G(a)$, there are elements $g_1,\dots,g_n\in G$ such that $b'\prec \sum_ig_ia$. Then $x\prec [b']\leq n[a]$, by condition (ii) in the second part of \autoref{dfn:actions-W-semigroups}. This implies that $[b]\in I([a])$.

Conversely, notice that $[a]\in\pi_G(I_G(a))$ as $a\in I_G(a)$ and since $\pi_G(I_G(a))$ is an ideal of $S / G$ by \autoref{prp:Gideals}, it follows that $I([a])\subseteq \pi_G(I_G(a))$.
\end{proof}


\begin{rmk}
\label{rmk:quotients}
It is easy to adapt the argument used in \cite[Proposition 1]{CiuRobSan10CuIdealsQuot} to show that, if $A$ is a \ca{} and $I$ a closed, two-sided ideal, then $\W(A)/\W(I)\cong \W(A/I)$ as $\CatW$-semigroups.
\end{rmk}

\begin{thm}
\label{thm:dynquotients}
Let $A$ be a \ca{}, and let $G$ be a discrete group acting on $A$. For any $G$-invariant ideal $I$ of $A$, we have that $\Dyn(I)$ and $\CuDyn(I)$ are (isomorphic to) ideals of $\Dyn(A)$ and $\CuDyn(A)$, respectively. Moreover,
\[
\Dyn(A)/\Dyn(I)\cong\Dyn(A/I) \text{ and }\CuDyn(A)/\CuDyn(I)=\CuDyn(A/I).
\]
\end{thm}
\begin{proof}
We know that $\W(I)$ is a $\CatGW$-ideal of $\W(A)$, and thus $\Dyn(I)$ is a $\CatW$-ideal of $\Dyn(A)$, by \autoref{prp:Gideals}. On the other hand $\CuDyn(I)=\gamma(\Dyn(I))$, and by \autoref{prp:WidealsvsCuideals} is a $\Cu$-ideal of $\CuDyn(A)$.

Using \autoref{prp:Gideals} at the first step and \autoref{rmk:quotients} at the second step, we obtain:
\[
\Dyn(A)/\Dyn(I)\cong (\W(A)/\W(I))/G\cong (\W(A/I))/G=\Dyn(A/I).
\]
Finally, applying \autoref{prp:WidealsvsCuideals} at the second step, we obtain
\begin{align*}
	\CuDyn(A)/\CuDyn(I) &  =\gamma(\Dyn(A))/\gamma(\Dyn(I))\cong \gamma(\Dyn(A)/\Dyn(I))\\ & \cong\gamma(\Dyn(A/I))
	\cong\CuDyn(A/I).\qedhere
\end{align*}
\end{proof}

\section{Dynamical strict comparison} \label{sec:comparison}

In this section, we introduce the notion of dynamical strict comparison for abstract $\CatGW$-semigroups and the corresponding notion for C*-algebras acted upon by a group.

\begin{pgr}[States and functionals]
	\label{pgr:functionals}
	Given a preordered semigroup $(S,\leq)$, we shall denote by $\mathrm{St}(S)$ the set of \emph{states} on $S$, that is, the set of those additive, zero- preserving, and order-preserving maps $\lambda\colon S\to [0,\infty]$. If $(S,\prec)$ is a $\CatW$-semigroup, we equip it with the induced preorder $a\leq_\prec b$ (defined by $a^\prec \subset b^\prec$), and it is natural to consider those states $\lambda\colon S\to [0,\infty]$ which satisfy the condition that $\lambda(a)=\sup\limits_{a'\prec a}\lambda(a')$. We shall refer to these states as \emph{$\CatW$-functionals}, and denote this set by $\fun{\!\W}{}(S)$. In general functionals are not required to preserve the relation $\prec$ and thus are not necessarily $\W$-morphisms.
	
	Given a state $\lambda$ on a $\CatW$-semigroup $S$, we may define $\bar{\lambda}$ by $\bar{\lambda}(a)=\sup\limits_{a'\prec a}\lambda(a')$. It is then easy to verify that $\bar{\lambda}$ belongs to $\fun{\!\W}{}(S)$ and $\bar{\lambda}\leq \lambda$, with equality precisely when $\lambda\in \fun{\!\W}{} (S)$; see the arguments in \cite[Proposition 4.1]{Ror92StructureUHF2}.
	
	If now $G$ is a group acting on a preordered semigroup $S$, we define the set of \emph{$G$-invariant states} as 
	\[
	\mathrm{St}^G(S)=\{\lambda\in \mathrm{St}(S)\mid \lambda(ga)=\lambda(a)\text{ for all }a\in S\}\,.
	\]
	When $S$ is a $\CatGW$-semigroup, we shall use the notation $\fun{\!\W}{G}(S)$ to refer to the subset of $\fun{\!\W}{}(S)$ consisting of those $\W$-functionals that are moreover $G$-invariant.
	
	Next, let $S$ be a $\Cu$-semigroup. Recall that, in this context, a \emph{functional} on $S$ is a state $\lambda\colon S\to [0,\infty]$ that further preserves suprema of increasing sequences. We denote, as customary, the set of all functionals on $S$ by $\fun{}{}(S)$, which is a compact convex Hausdorff space, as shown in \cite[Theorem 4.8]{EllRobSan11Cone} (see also \cite{Rob13Cone}), with the following topology: a net $(\lambda_i)_{i\in I}$ in $\fun{}{}(S)$ converges to $\lambda$ in $\fun{}{}(S)$ provided
	\[
	\limsup\limits_{i\in I}\lambda_i(a')\leq\lambda(a)\leq\liminf\limits_{i\in I}\lambda_i(a)
	\]
	for any elements $a',a\in S$ with $a'\ll a$.
	
	If $S$ is a $\CatW$-semigroup, the precise relationship between $\fun{\!\W}{}(S)$ and $\fun{}{}(\gamma(S))$ is given via the natural completion morphism $\gamma\colon S\to\gamma(S)$, as shown below. (See \autoref{pgr:completion})
\end{pgr}

\begin{prp}
	\label{prp:WfunCufun} Let $(S,\prec)$ be a $\W$-semigroup. Then:
	\beginEnumConditions
	\item There is an affine homeomorphism $\fun{\!\W}{}(S)\cong\fun{}{}(\gamma(S))$.
	\item If, further, $S$ is a $\CatGW$-semigroup and $\pi_G\colon S \to S / G$ is the natural map, then the assignment $\lambda\mapsto\lambda\circ\pi_G$ defines an affine homeomorphism $\fun{\!\W}{}(S / G)\cong \fun{\!\W}{G}(S)$. In particular, if $S$ is a $\CatGCu$-semigroup, then $\fun{}{}(\gamma(S/G))\cong\fun{}{G}(S)$.
\end{enumerate}
\end{prp}	
\begin{proof}
(i): 	
For $\lambda\in\fun{\!\W}{} (S)$, define $\tilde{\lambda}\in \fun{}{}(S)$ as follows: given $a\in\gamma(S)$, choose a $\prec$-increasing sequence $(a_n)$ such that $a=\sup\gamma(a_n)$, and set
\[
\tilde\lambda(\sup\gamma(a_n))=\sup\lambda(a_n). 
\]
This is well defined, for if $a=\sup\gamma(a_n)=\sup\gamma(b_n)$, for $\prec$-increasing sequences $(a_n)$ and $(b_n)$, then for each $n$, there is $m$ such that $\gamma(a_n)\ll\gamma(b_m)$, whence $a_n\prec b_m$. This implies that $\lambda(a_n)\leq \lambda(b_m)$, and thus $\sup_n\lambda(a_n)\leq\sup_m\lambda(b_m)$. By symmetry we have equality.

It is clear that $\tilde{\lambda}$ is additive and the argument above shows it is also order-preserving. 
To show that it preserves suprema of increasing sequences, we may use the argument in \cite[Proposition 3.5]{AntAraBosPerVil23}, based in turn on how suprema are constructed in $\gamma(S)$; see \cite[Proposition 3.1.6]{AntPerThi14arX:TensorProdCu} for further details.

Conversely, given $\mu\in \fun{}{}(\gamma(S))$, define $r(\mu)$ on $S$ by restriction, that is, $r(\mu)(a)=\mu(\gamma(a))$. We clearly have that $r(\mu)$ is additive. Observe that, by construction of $\gamma(S)$, we have $\gamma(a)=\sup_{a'\prec a}\gamma(a')$ for all $a\in S$. This implies that $r(\mu)$ is order-preserving and, since $\mu$ preserves suprema, we also have 
\[
r(\mu)(a)=\mu(\sup\limits_{a'\prec a}\gamma(a'))=\sup\limits_{a'\prec a}\mu(\gamma(a'))=\sup\limits_{a'\prec a}r(\mu)(a')
\]
and thus $r(\mu)\in\fun{\!\W}{}(S)$.

Now, if $\lambda\in\fun{\!\W}{}(S)$ and $\tilde{\lambda}$ is as in the first part of the proof we have, for $a\in S$, that 
\[
r(\tilde{\lambda})(a)=\tilde{\lambda}(\gamma(a)))=\tilde{\lambda}(\sup\limits_{a'\prec a}\gamma(a'))=\sup\limits_{a'\prec a}\lambda(a')=\lambda(a),
\]
and hence $r(\tilde{\lambda})=\lambda$.

Conversely, given $\mu\in\fun{}{}(\gamma(S))$, write $\lambda:=r(\mu)$ and then, for each $a\in \gamma(S)$ written as $a=\sup\gamma(a_n)$ for a $\prec$-increasing sequence in $S$, we have 
\[
\tilde{\lambda}(a)=\sup\lambda(a_n)=\sup r(\mu)(a_n)=\sup\mu(\gamma(a_n))=\mu(\sup\gamma(a_n))=\mu(a).
\]
This shows that the assignments $\lambda\mapsto\tilde{\lambda}$ and $\mu\mapsto r(\mu)$ are inverses for one another. It is easy to verify that both assignments are also affine. 

Finally, if we equip $\fun{\!\W}{}(S)$ with the topology under which a net $(\lambda_i)_{i\in I}$ converges to $\lambda$ precisely when 
\[
\limsup_{i\in I}\lambda_i(a')\leq\lambda(a)\leq \liminf_{i\in I}\lambda_i(a)
\]
for any elements $a',a\in S$ with $a'\prec a$, it follows that what we have defined is a homeomorphism.

(ii):  Since $\pi_G$ is a $\CatW$-morphism, given $\lambda\in \fun{\!\W}{}(S/G)$, we have that $\lambda\circ\pi_G\in \fun{\!\W}{}(S)$. Next, take $a\in S$ and $g\in G$. Then by construction we have $[ga]=[a]$ in $S/G$, and thus $\lambda(\pi_G(ga))=\lambda(a)$. This shows that $\lambda\circ\pi_G$ belongs to $\fun{\!\W}{G}(S)$ and therefore the map $\Lambda\colon \fun{\!\W}{}(S/G)\to \fun{\!\W}{G}(S)$ given by $\Lambda(\lambda)=\lambda\circ\pi_G$ is well defined.

Now, let $\bar{\lambda}\in \fun{\!\W}{G}(S)$ be given. For $a\in S$, put $\lambda([a])=\bar{\lambda}(a)$. Let us check that $\lambda\in\fun{\!\W}{}(S/G)$. First, if $[a]\leq [b]$, that is, $a\leqr{G}b$, and $c\prec a$, then by \autoref{cor:dynamical-Cuntz-subequivalence} either $c\prec b$ or else there are $n,m$ and elements $x_0,x_0'\dots,x_n, x_n'\in S$ with $x_i=\sum_{j=1}^m d_{ij}, x_i'=\sum_{j=1}^{m}g_{ij}d_{ij}$, for some elements $g_{ij}\in G$, and such that $x_i'\prec x_{i+1}$ for each $i<n$, $c\prec x_0$, and $x_n'\prec b$. Since $\bar{\lambda}$ is $G$-invariant, this clearly implies that $\bar{\lambda}(x_i')=\bar{\lambda}(x_i)$ for all $i$ and thus $\bar{\lambda}(c)\leq \bar{\lambda}(b)$. Now use that $\bar{\lambda}(a)=\sup\limits_{c\prec a}\bar{\lambda}(c)$ to conclude that $\bar{\lambda}(a)\leq \bar{\lambda} (b)$ and therefore $\lambda$ is well defined and order-preserving. It is easy to verify that $\lambda([a])=\sup\limits_{x\prec [a]}\lambda(x)$, and also that $\Lambda$ is an affine homeomorphism.

For the second part of the statement, applying (i) at the first and last step, and (ii) at the second step, we get 
\[
\fun{}{}(\gamma(S/G))\cong\fun{\!\W}{}(S/G)\cong\fun{\!\W}{G}(S)\cong\fun{}{G}(S)\qedhere
\]
\end{proof}

\begin{rmk}
We have kept our approach above as it gives details of the exact correspondence. One can however apply results about $\W$-morphisms to deal with $\W$-func\-tio\-nals, in particular to obtain a shorter, more conceptual proof for \autoref{prp:WfunCufun}. To do so, take the $\Cu$-semigroup $M_\infty$ defined in \cite[Example 4.14]{AntPerThi20:Absbivariant} and check that $\fun{\!\W}{}(S)=\W(S,M_\infty)$ and also that $\fun{}{}(\gamma(S))=\Cu(\gamma(S),M_\infty)$. (It was shown in \cite[Proposition 4.15]{AntPerThi20:Absbivariant} that $M_{\infty}\cong \Cu(N)$ for every $\mathrm{II}_{\infty}$-factor $N$.) Now, using the fact that $M_\infty$ is a $\Cu$-semigroup and that $\Cu$ is a reflexive subcategory of $\W$ at the second step, we obtain
\[
\fun{\!\W}{}(S)=\W(S,M_\infty)\cong\Cu(\gamma(S),M_\infty)=\fun{}{}(\gamma(S)).
\]
We also get, applying \autoref{thm:dynamical-universal-property},
\[
\fun{\!\W}{}(S/G)=\W(S/G, M_\infty)\cong \W^G(S,M_\infty)=\fun{\!\W}{G}(S).
\]
\end{rmk}


\begin{pgr}[Dynamical strict comparison for $\CatGW$-semigroups]	\label{pgr:dynstrictcomp}
We say that a $\CatW$-semigroup $(S,\prec)$ has \emph{strict comparison} provided the following condition holds: for $a,b\in S$ with $a\in I(b)$, if $\lambda(a) < \lambda(b)$ for all $\lambda\in\fun{\W}{} (S)$ with $\lambda(b)<\infty$, it follows that $a^\prec\subset  b^\prec$. 

It follows from arguments in, for example, \cite{Ror92StructureUHF2} or \cite[Proposition 3.2]{Ror04StableRealRankZ}, that this is equivalent to the property of \emph{almost unperforation} in $S$: for $a,b\in S$, if $(k+1) a \prec k b$ for some natural number $k$, then $a^\prec  \subset b^\prec$. 	
A related but stronger property is that of \emph{unperforation}: for $a,b\in S$, if $k a \prec k b$ for some natural number $k$, then $a^\prec\subset  b^\prec$. 

In a similar manner, we say that a $\CatGW$-semigroup $S$ has \emph{dynamical strict comparison} provided the following condition holds: for $a,b\in S$ with $a\in I_G(b)$, if $\lambda(a)<\lambda(b)$ for all 
$\lambda\in\fun{\!\W}{G}(S)$ with $\lambda(b)<\infty$, it follows that $a \dyn b$.

A $\Cu$-semigroup $S$ is almost unperforated if $(k+1)a\leq kb$, for $a,b\in S$ implies $a\leq b$. Using the construction of the completion $\gamma(S)$ for any $\W$-semigroup $(S,\prec)$ as sketched in \autoref{pgr:completion}, it is a routine exercise to check that $(S,\prec)$ is almost unperforated if, and only if, so is $\gamma(S)$.
\end{pgr}


\begin{prp}
\label{prp:dynstrictcomp} Let $S$ be a $\CatGW$-semigroup. Then, the following conditions are equivalent:
\beginEnumConditions
\item $S$ has dynamical strict comparison.
\item For $a,b\in S$ with $a\in I_G(b)$, if $\lambda(a)<\lambda(b)$ for all $\lambda\in\mathrm{St}^G(S)$ such that $\lambda(b)=1$, then $a\leqr{G} b$.
\item $S / G$ is almost unperforated.
\item $\gamma (S / G)$ is almost unperforated.
\end{enumerate}
\end{prp}
\begin{proof}
(i)$\iff$(ii): This is easy to verify.

(ii)$\implies$(iii): Suppose that, for elements $[a]$ and $[b]\in S / G$ and $k\in\N$, we have $(k+1)[a]\prec k[b]$. If $x\prec [a]$, then $(k+1)x\prec (k+1)[a]\prec k[b]$, and since $x$ is arbitrary, this implies that $[a]\in I([b])$. Thus, by \autoref{lma:ppalideal} there is $c\in I_G(b)$ with $[a]=[c]$. Replacing $a$ by $c$, we may assume that $a$ belongs to $I_G(b)$.

Let $a'\in S$ with $a'\prec a$. 
Then, since $(k+1)[a]\prec k[b]$ there is $b'\in I_G(b)$ with $b'\prec b$ such that $(k+1)a' \leqr{G} kb'$. Therefore, for any $\lambda\in\mathrm{St}^G(S)$ such that $\lambda(b)=1$, we have $\lambda(a')\leq\frac{k}{k+1}\lambda(b')<\lambda (b')\leq\lambda (b)$. By assumption, this implies that $a'\leqr{G} b$. Since $a'\prec a$ was arbitrary, we see that $a \leqr{G} b$.

(iii)$\implies$ (i): In order to verify that $S$ has dynamical strict comparison, let $a,b$ be elements in $S$ with $a\in I_G(b)$,  and assume that $\lambda(a)<\lambda(b)$ for any $\lambda\in\fun{\W}{G}(S)$.

By \autoref{lma:ppalideal}, we have $[a]\in I([b])$. Take $a' \in S$ with $a'\prec a$. Then $[a']\prec [a]$ and thus there is $k\in\N$ with $[a']\prec k[b]$. Next, let $\mu$ be a state on $S / G$ normalized at $[b]$. Put $\mu_0:=\mu\circ\pi_G$, which defines a $G$-invariant state on $S$ normalized at $b$, and let $\bar{\mu}_0$ be the functional defined as in \autoref{pgr:functionals}, that is, $\bar{\mu}_0(x)=\sup\limits_{x'\prec x}\mu_0(x')$. Note that $\bar{\mu}_0\in\fun{\!\W}{G}(S)$. We have that $\bar{\mu}_0(b)\leq 1$ and using our assumption,
\[
\mu([a'])=\mu_0(a')\leq\bar{\mu}_0(a)<\bar{\mu}_0(b)\leq\mu_0(b)=\mu([b])\,.
\]
Since $S/G$ is almost unperforated, this implies that $[a']\prec [b]$ in $S / G$, that is, $a'\leqr{G}b'$ for some $b'\prec b$. As $a'\prec a$ is arbitrary, we obtain $a \leqr{G} b$, as desired.

That (iii) and (iv) are equivalent has been observed in \autoref{pgr:dynstrictcomp}.	
\end{proof}


\begin{pgr}[Quasitraces] \label{pgr:quasitraces}
For a C*-algebra $A$, recall that 
a \emph{quasitrace} on $A$ is a map $\tau \colon A_+ \to [0, \infty]$ such that $\tau(0) = 0$, $\tau(x x^*) = \tau(x^* x)$ for any $x \in A$, and $\tau(a + b) = \tau(a) + \tau(b)$ for any $a, b \in A_+$ with $ab = ba$. 
A quasitrace $\tau$ on $A$ is a \emph{$2$-quasitrace} if it extends to a quasitrace $\tau_2$ on $M_2(A)$ with $\tau_2 (a \otimes e_{11}) = \tau(a)$ for any $a \in A_+$. It is lower semicontinuous if $\tau(a)=\sup_{t>0}\tau((a-t)_+)$ for all $a\in A_+$. We will denote by $\QT_2(A)$, or simply by $\QT(A)$, the set of lower semicontinuous 2-quasitraces.
By \cite[Remark~2.27(viii)]{BlaKir04PureInf}, any lower semicontinuous 2-quasitrace has a unique extension to a lower semicontinuous n-quasitrace $\tau_n$ on $M_n(A)$ with $\tau_n (a \otimes e_{11}) = \tau(a)$ for any $a \in A_+$. 
By \cite[Theorem 4.4]{EllRobSan11Cone}, there is a homeomorphism between $\QT(A)$ and $\fun{}{}(\Cu(A))$. More precisely, given a lower semicontinuous 2-quasitrace $\tau$ on $A$, the corresponding functional is given by $[a] \mapsto d_\tau (a)$, where $d_\tau \colon \Cu(A) \to [0,\infty]$ is given by $[a]\mapsto \lim_{n \to \infty} \tau ( a^{1/n} )$. 

If $A$ is a C*-algebra and $G$ is a group acting on $A$, the above correspondence restricts to a one-to-one correspondence between the $G$-invariant functionals $\fun{}{G}(\Cu(A))$ on $\Cu(A)$ and $G$-invariant 2-quasitraces on $A$, which will be denoted by $\QT_G(A)$. 
\end{pgr}


\begin{pgr}[C*-Dynamical strict comparison]
\label{pgr:dynstrict}
Let $G$ be a discrete group that acts on a C*-algebra $A$. We say that $A$ has \emph{dynamical strict comparison} provided $[a]\leq [b]$ in $\Dyn(A)$ whenever $d_\tau(a)<d_\tau(b)$ for any $G$-invariant 2-quasitrace $\tau$. This is equivalent to the fact that $\W(A)$ has dynamical strict comparison in the sense of \autoref{pgr:dynstrictcomp}.  In turn, this is equivalent to saying that $\Dyn(A)$ (or $\CuDyn(A)$) is almost unperforated; see \autoref{prp:dynstrictcomp}.

We also say that $A$ is \emph{dynamically unperforated} provided that $\Dyn(A)$ is unperforated (equivalently, if $\CuDyn(A)$ is unperforated).
\end{pgr}


\begin{cor}
\label{cor:strictdyn}
Let $G$ be a discrete group acting on a C*-algebra $A$, let $I$ be a $G$-invariant ideal of $A$, and let $h\in A_+$ be a $G$-invariant element. Then:
\beginEnumConditions
\item  If $A$ has dynamical strict comparison, then so do $I$ and $A/I$.
\item  If $A$ is dynamically unperforated, then so are $I$, $A/I$, and  $\overline{hAh}$.
\end{enumerate}

\end{cor}
\begin{proof}
(i):	That $A$ has dynamical strict comparison is equivalent, by \autoref{prp:dynstrictcomp}, to the fact that $\Dyn(A)$ is almost unperforated. Now it is trivial to verify that this condition passes both to ideals and quotients. Thus, $\Dyn(I)$ and $\Dyn(A)/\Dyn(I)$ are both almost unperforated. Since the latter is isomorphic to $\Dyn(A/I)$, by \autoref{thm:dynquotients}, we conclude that $I$ and $A/I$ have dynamical strict comparison.

(ii): That $I$ and $A/I$ are dynamically unperforated follows by an argument as in (i).

Now, given a $G$-invariant $h\in A_+$, we have that $B:=\overline{hAh}$ is stably isomorphic to the ideal $I$ generated by $h$, which is $G$-invariant and, by the first part of the proof, is dynamically unperforated. As in  \autoref{cor:CuWidentifications}, we have
\[
\CuDyn(I)\cong\gamma(\Dyn(I\otimes \K))\cong\gamma(\Dyn(B\otimes\K))\cong\CuDyn(B),
\]
and thus $B$ is also dynamically unperforated.
\end{proof}


Given a discrete group $G$ acting on a C*-algebra $A$,  one can use the universal property of the dynamical Cuntz semigroup (see \autoref{thm:dynamical-universal-property}) to obtain the following commutative diagram, where the top arrow is induced by the natural embedding $\iota\colon A\to A\rtimes G$ and $\kappa:=\gamma(\iota_G)$
\[
\xymatrix{
\CatCu(A) \ar[r]^{\!\!\!\!\!\!\!\!\!\!\!\!\!\!\!\!\!\!\!\!\!\!\!\!\!\!\!\!\!\!\iota} \ar[d]_{\pi_G} &	\CatCu(\iota(A))\subseteq\CatCu(A\rtimes G)   \\
\Cu(A)/G \ar[d]_{\gamma}\ar@{-->}[ru]_{\iota_G} & \\
\CuDyn(A) \ar@{-->}[ruu]_{\kappa} &
}
\] 

\begin{pgr}[Soft elements]
If $S$ is a $\Cu$-semigroup, we will call an element $a\in S$ \emph{soft} if, given $a'\ll a$, there is $k\in\N$ such that $(k+1)a'\leq ka$. This is the weakest of three versions of softness, as considered in \cite[Definition 4.3]{ThiVil24}, and referred there as ``functional softness''. This condition was already considered in \cite[Definition 5.3.1]{AntPerThi14arX:TensorProdCu}. It follows from \cite[Proposition 4.6, Corollary 4.7]{ThiVil24} that all notions of softness agree for residually stably finite C*-algebras.

Let us denote by $S_{\rm soft}$ the subset of soft elements, which is always a subsemigroup of $S$ by \cite[Theorem 5.3.11]{AntPerThi14arX:TensorProdCu} that satisfies \axiomO{1}. For a simple, stably finite, C*-algebra, $\Cu(A)_{\rm soft}$ is actually a $\Cu$-semigroup, as follows from \cite[Proposition 5.3.18]{AntPerThi14arX:TensorProdCu}.

It is well known and easy to show that a $\Cu$-morphism maps soft elements to soft elements.
\end{pgr}


\begin{prp}
\label{prp:soft} Let $S$ be a $\CatGCu$-semigroup, let $T$ be a $\Cu$-semigroup, and let $f\colon S\to T$ be an invariant $\Cu$-morphism such that the induced map $\fun{}{}(T)\to\fun{}{G}(S)$ is surjective.  If $S/G$ is almost unperforated, then the restriction of the induced map $\gamma(f_G)\colon \gamma(S/G)\to T$ to the subsemigroup of soft elements satisfies that, if $a,b$ are soft with $a\in I(b)$ and $\gamma(f_G)(a)\leq \gamma(f_G)(b)$, then $a\leq b$.
\end{prp}
\begin{proof}
Write $\varphi=\gamma(f_G)$. Let $a,b$ be soft elements in $\gamma(S/G)$ with $a\in I(b)$ and let $a'\ll a$. Then, there is $k\in \N$ such that $(k+1)a'\leq ka$, and thus $(k+1)\varphi(a')\leq k\varphi(a)\leq \varphi(b)$.

Now take any functional $\lambda\in \fun{}{}(T)$. We have $(k+1)\lambda(\varphi(a))\leq k\lambda(\varphi(b))$, hence $(\lambda\circ\varphi)(a)<(\lambda\circ\varphi)(b)$. 

Using our assumption on functionals we have that any functional on $\gamma(S/G)$ has the form $\lambda\circ\varphi$, for some functional $\lambda$ on $T$. Therefore, since $\gamma(S/G)$ is almost unperated and $a\in I(b)$, we obtain that $a\leq b$.
\end{proof}


\begin{thm}\label{thm:DynamicalOrderEmbedding} Let $G$ be a discrete group acting minimally on a C*-algebra $A$. If $A$ has dynamical strict comparison, then the $\CatCu$-morphism $\kappa_{|\CuDyn(A)_{{\rm soft}}}$ is an order embedding.
\end{thm}	
\begin{proof}
Upon identification of $\CuDyn(A)$ with $\gamma(\Cu(A)/G)$, the assumption that $A$ has dynamical strict comparison means that $\CuDyn(A)$ is almost unperforated (see \autoref{prp:dynstrictcomp}). 

Therefore, in order to apply \autoref{prp:soft}, we need to check that the induced map $\fun{}{}(\Cu(A\rtimes G))\to\fun{}{G}(\Cu(A))$ is surjective. We may identify $\QT(A\rtimes G)$ with $\fun{}{}(\Cu(A\rtimes G))$ and $\QT_G(A)$ with $\fun{}{G}(\Cu(A))$; see the comments in \autoref{pgr:quasitraces}. Thus, we have to prove that the natural map $\QT(A\rtimes G)\to \QT_G(A)$ is surjective. Let $E\colon A\rtimes G\to A$ be the conditional expectation. If $\tau$ is a $G$-invariant quasitrace on $A$, then we have that $\tilde{\tau}:=\tau\circ E$ is a quasitrace on $A\rtimes G$ that restricts to $\tau$. Thus the natural map is surjective, as was to be shown.
\end{proof}

\section*{Acknowledgements}
The authors would like to thank G.~A.~Elliott for a helpful comment that partly inspired this work. 

\bibliographystyle{aomalphaMy}

\end{document}